\documentclass[12pt]{amsart}
\usepackage{latexsym,amsxtra,amscd,ifthen}
\usepackage{amsfonts}
\usepackage{verbatim}
\usepackage{amsmath}
\usepackage{amsthm}
\usepackage{amssymb}
\usepackage{url}
\theoremstyle{plain}

\newtheorem{theorem}{Theorem}
\newtheorem{lemma}[theorem]{Lemma}
\newtheorem{proposition}[theorem]{Proposition}
\newtheorem{corollary}[theorem]{Corollary}
\newtheorem*{corollary*}{Corollary}
\newtheorem{conjecture}[theorem]{Conjecture}

\numberwithin{theorem}{section}
\numberwithin{equation}{theorem}

\theoremstyle{definition}
\newtheorem{definition}[theorem]{Definition}
\newtheorem{example}[theorem]{Example}
\newtheorem{examples}[theorem]{Examples}
\newtheorem{remark}[theorem]{Remark}
\newtheorem*{remark*}{Remark}
\newtheorem{remarks}[theorem]{Remarks}
\newtheorem{question}[theorem]{Question}
\newtheorem{questions}[theorem]{Questions}
\newtheorem{family}[theorem]{Families}
\newtheorem*{question*}{Question}

\newtheorem{obs}[theorem]{Observation}

\newcommand{\CC}{\mathbb{C}}
\newcommand{\QQ}{\mathbb{Q}}
\newcommand{\SB}{\mathbb{S}}
\newcommand{\bfs}{\mathbf{s}}

\newcommand{\FF}{\mathbb{F}}
\newcommand{\ZZ}{\mathbb{Z}}
\DeclareMathOperator \Fract { {\mathrm{Fract}} }

\DeclareMathOperator{\ddb}{ddb}
\DeclareMathOperator{\dpb}{dpb}
\DeclareMathOperator{\fdb}{fdb}

\DeclareMathOperator{\Aut}{Aut}
\DeclareMathOperator{\fht}{fht}

\DeclareMathOperator{\vht}{vht}

\DeclareMathOperator{\gr}{gr}

\DeclareMathOperator{\trdeg}{tr{.}deg}

\newcommand \bff {{\mathbf{f}}}
\newcommand \bfg {{\mathbf{g}}}

\newcommand \calO {{\mathcal{O}}}
\newcommand \calV {{\mathcal{V}}}

\newcommand \xhat {\hat{x}}
\newcommand \yhat {\hat{y}}

\newcommand \Znn {\ZZ_{\ge 0}}
\newcommand \Zpos {\ZZ_{> 0}}

\DeclareMathOperator{\Weyl}{Weyl}
\newcommand \Kweyl {K_{\Weyl}}
\newcommand \kx {\Bbbk^\times}
\newcommand \kbar {\overline{\Bbbk}}
\DeclareMathOperator{\Id}{Id}

\begin{document}

\title{Poisson fields of two variables}

\author{Ken Goodearl and James J. Zhang}

\address{Goodearl: Department of Mathematics, 
University of California, Santa Barbara, California, USA}

\email{goodearl@math.ucsb.edu}

\address{Zhang: Department of Mathematics, Box 354350,
University of Washington, Seattle, Washington 98195, USA}

\email{zhang@math.washington.edu}

\begin{abstract}
We study invariants and structures of 
Poisson fields of rational functions in 
two variables. For four particular families, 
we classify the members, establish criteria 
for isomorphisms and, with the exception 
of the Weyl Poisson field, describe the 
automorphism groups. Embeddings are also 
investigated, along with an analog of the 
Dixmier Conjecture: For which Poisson 
fields is every Poisson endomorphism an 
automorphism? The answer is negative for 
the first family, but positive answers are 
obtained for several subclasses of the 
other families. Finally, we exhibit a 
Poisson field which is not isomorphic to 
any Poisson field $\Bbbk(x,y)$ for which 
$\{x,y\}$ is a polynomial in $\Bbbk[x,y]$.
\end{abstract}



\maketitle

\section*{Introduction}
\label{zzsec0}

Poisson algebras have been used extensively 
in many areas, and several important 
questions, such as automorphism and 
isomorphism problems, for Poisson algebras 
or fields, have recently been studied by 
several authors \cite{Bl, MLU, HTWZ1, HTWZ2}. 
The subject of Poisson algebras is closely 
connected to Poisson geometry, deformation 
quantization, and noncommutative algebraic 
geometry. For example, Poisson fields are 
viewed as an analogue of skew fields that 
are essential in the study of noncommutative 
birational geometry. 

There are two important Poisson fields 
that are commonly studied and have close 
connections to other research topics. The 
first one is the {\it Weyl Poisson field} 
$$\Kweyl:=\Bbbk(x,y\mid \{x,y\}=1)$$ 
whose underlying commutative field is 
the field of rational functions of two 
variables $\Bbbk(x,y)$ and whose Poisson 
structure is determined by the equation 
$\{x,y\}=1$. Here we fix a base field 
$\Bbbk$ that is of characteristic zero. 
Throughout this introduction, we assume 
that $\Bbbk$ is algebraically closed.

The second well known instance is the 
{\it $q$-skew Poisson field}, for a scalar 
$q \in \Bbbk$, namely 
\begin{equation} 
\notag
K_{q}:=\Bbbk(x,y\mid \{x,y\}=qxy)
\end{equation}
whose underlying commutative field is 
$\Bbbk(x,y)$ and whose Poisson structure 
is determined by the equation $\{x,y\}=qxy$. 
Note that the Weyl and $q$-skew Poisson 
fields can be viewed as semiclassical 
limits of the Weyl skew field 
$\Fract \Bbbk \langle x,y\mid xy-yx=1 
\rangle$ and the skew field 
$\Fract \Bbbk \langle x,y\mid xy 
= [(p-1)q+1] yx \rangle$. 

More generally, for every element 
$\bff\in \Bbbk(x,y)$ one can define a 
Poisson algebra structure on $\Bbbk(x,y)$ 
by using the equation
\begin{equation}
\label{E0.0.1}\tag{E0.0.1}
\{x,y\}=\bff. 
\end{equation}
More precisely, there is a unique Poisson 
structure on $\Bbbk(x,y)$ such that 
\eqref{E0.0.1} holds (see Corollary 
\ref{zzcor9.4}). For simplicity, 
we use $K\{\bff\}$ to denote this 
Poisson field. In this case we say $\bff$ 
is the {\it flag} of $K\{\bff\}$. Using 
this convention, $\Kweyl=K\{1\}$ and 
$K_{q}=K\{qxy\}$. It is well-known that 
$K_{q}$ is not isomorphic to $\Kweyl$ 
(e.g., \cite[Corollary 5.3]{GoLa}). In 
fact, $K_{q}$ cannot be embedded in 
$\Kweyl$ \cite[Corollary 5.9(2)]{HTWZ2}. 
On the other hand, given two general 
elements $\bff_1, \bff_2$ in $\Bbbk(x,y)$, 
it is difficult to tell whether 
$K\{\bff_1\}$ is isomorphic to $K\{\bff_2\}$ 
or not. So we formally ask the following 
questions:

\begin{question}[Classification Problem]
\label{zzque0.1}
Can we classify all Poisson fields of 
the form $K\{\bff\}$ up to Poisson 
isomorphism? 
\end{question}

This problem is extremely difficult and 
the moduli of Poisson fields $K\{\bff\}$ 
is very complicated. We introduce some 
invariants that can help or control the 
classification. 

\begin{definition}
\label{zzdef0.2}
Let $K$ be a Poisson field of the form 
$K\{\bff\}$ as defined above. The {\it 
flag height} of $K$ is defined to be
\begin{equation}
\label{E0.2.1}\tag{E0.2.1}
\fht(K) :=\min\{\deg \bfg \; \mid \; 
K\cong K\{\bfg\},\; \bfg\in \Bbbk[x,y]\}
\end{equation}
if there exist $\bfg\in \Bbbk[x,y]$ such 
that $K\cong K\{\bfg\}$. Otherwise, we 
define $\fht(K) :=\infty$. 
\end{definition}

For every non-negative integer $h$, let 
$Poi_{h}$ denote the moduli space of 
Poisson fields $K\{\bff\}$ with $\fht$ $h$. 
Trivially, $Poi_{0}$ consists of one 
Poisson field, which is $\Kweyl$.  
(This assumes that we view the zero 
polynomial as having degree less than $0$.) 
It is easy to check that $Poi_{1}$ is empty 
(cf.~Proposition \ref{zzpro1.1}(1)(2)). 
Dumas \cite[Proposition, p.12]{Dumas} 
proved that every Poisson field 
$K\{\bff\}$ with $\fht$ 2 is isomorphic 
to $K_{q}$ $(\cong K_{-q})$ for some 
$q\in \Bbbk^{\times}$.  This also gives 
a classification of the Poisson fields 
in $Poi_{2}$. In \cite{GZ2} we work out 
a description of $Poi_{3}$ (as well as 
some partial classifications of $Poi_{h}$ 
when $h=4,5$). Suppose $\bff\in \Bbbk[x,y]$ 
has degree 3. If $\bff$ is reducible (or 
more generally if the curve $\bff=0$ 
is singular), then $K\{\bff\}$ is isomorphic 
to either $\Kweyl$ or $K_q$. If the curve 
$\bff=0$ is smooth, then $K\{\bff\}$ is 
completely determined by the (elliptic) 
curve $\bff=0$. So the moduli of $Poi_{3}$ 
are related to the moduli of elliptic 
curves \cite{GZ2}. For $h=4$ or $5$, the 
big picture of $Poi_{h}$ is mysterious. 
When $\bff$ is irreducible of degree 4, we 
have some conjectures \cite{GZ2}. When 
$\bff$ is reducible of degree 4, it is 
also very complicated to understand 
$K\{\bff\}$.

For higher $h$, it is impossible to 
completely solve the classification problem 
of $Poi_{h}$ using currently available 
tools. The approach in this paper is to 
consider several families of reducible 
$\bff$ and the corresponding Poisson 
fields. Here is a list of the Poisson 
fields that we will study in detail.

\begin{family}
\label{zzfam0.3}
\begin{enumerate}
\item 
$K\{\bff\}$ where $\bff=q x^a y^b$ for 
some $q\in \Bbbk^{\times}$ and $a,b\in {\ZZ}$.
\item 
$K\{\bff\}$ where $\bff=p(x) xy $ for 
some $p(x)\in \Bbbk[x]$.
\item
$K\{\bff\}$ where $\bff\in \Bbbk[x,y]$ is 
homogeneous.
\item 
$K\{\bff\}$ where $\bff = f(x) g(y)$ for 
some $f(x)\in \Bbbk[x]$ and $g(y)\in \Bbbk[y]$.
\end{enumerate}
\end{family}

To better understand the classification 
project (Question \ref{zzque0.1}), we 
consider as well several subprojects, as 
below.

\begin{question}[Isomorphism Problem]
\label{zzque0.4}
Let $\bff_1$ and $\bff_2$ be two elements 
in $\Bbbk(x,y)$. Are there any effective 
criteria on when $K\{\bff_1\}\cong K\{\bff_2\}$?
\end{question}

It is not hard to see that $K\{g(x)\} \cong 
K\{1\}$ for any nonzero rational function 
$g(x) \in \Bbbk(x)$ (Proposition 
\ref{zzpro1.1}(1)); and with some effort 
one can show that $K\{x^{n} y\} \not\cong 
K\{x^{m} y\}$ if $n>m>0$ (Corollary 
\ref{zzcor4.2} and Proposition 
\ref{zzpro4.6}). Answers to Question 
\ref{zzque0.4} within the families (1)--(4) 
are given in Proposition \ref{zzpro4.6} 
and Corollaries \ref{zzcor5.5}, 
\ref{zzcor6.4}, \ref{zzcor7.4}. To save 
space we will not give full statements 
about the isomorphism problem in this 
introduction. 

\begin{question}[Embedding Problem]
\label{zzque0.5}
Let $\bff_1$ and $\bff_2$ be two elements 
in $\Bbbk(x,y)$. Are there any effective 
criteria on when there is a Poisson algebra 
morphism $K\{\bff_1\}\to K\{\bff_2\}$?
\end{question}

The Embedding Problem within family (1) 
is solved in Proposition \ref{zzpro4.10}.

Motivated by the Embedding Problem, we 
introduce the following concepts. 

\begin{definition}
\label{zzdef0.6}
Let $K$ and $K'$ be Poisson fields of 
the form $K\{\bff\}$.
\begin{enumerate}
\item
We say $K$ is {\it height cohereditary} 
if $\fht(K) \leq
\fht(K')$ for arbitrary $K' \supseteq K$.
\item
We say $K$ is {\it height hereditary}
if $\fht(K')\leq \fht(K)$ for arbitrary 
$K' \subseteq K$.
\end{enumerate}
\end{definition}

\begin{question}
\label{zzque0.7} 
Is there a Poisson field $K\{\bff\}$ that 
is not height (co)hereditary?
\end{question}

By definition, $\Kweyl$ is height cohereditary. 
It follows from \cite[Corollary 5.9(2)]{HTWZ2} 
that $K_{q}$ is height cohereditary. In 
this paper we will show that several 
families of Poisson fields are height 
cohereditary, see Lemmas \ref{zzlem6.2} 
and \ref{zzlem7.2}. But we cannot verify 
any example that is height hereditary; 
in particular, we do not know if 
$\Kweyl$ (resp., $K_{q}$ for $q\neq 0$) 
is height hereditary. 

There is another notion of ``height'', 
called {\it $1$-valuation height} and 
denoted by $\vht_1$, see Definition 
\ref{zzdef3.13}. The $1$-valuation height 
has a few advantages:

\begin{enumerate}
\item[(1)]
$\vht_1$ is defined for all Poisson 
fields with nontrivial Poisson bracket;
\item[(2)]
$\vht_1$ has the cohereditary property 
[Lemma \ref{zzlem3.14}(2)];
\item[(3)]
$\vht_1$ is bounded above by $\fht$ 
when applied to $K\{\bff\}$ 
[Lemma \ref{zzlem3.14}(4)].
\end{enumerate}

In general, $\vht_1(K)\neq \fht(K)$, 
see Lemma \ref{zzlem4.4}(3). We 
compute both $\vht_1$ and $\fht$ for 
Poisson fields in the four families 
mentioned above.

The automorphism problem is always 
an interesting question.

\begin{question}[Automorphism Problem]
\label{zzque0.8}
Let $\bff$ be a nonzero element in 
$\Bbbk(x,y)$. Are there any effective 
ways of computing the Poisson algebra 
automorphism group of $K\{\bff\}$?
\end{question}

Since the algebra automorphism group 
of $\Bbbk(x,y)$ is called the {\it 
Cremona group} of rank 2 and denoted 
by $Cr_2(\Bbbk)$, we call the Poisson 
algebra automorphism group of 
$K\{\bff\}$ the {\it Poisson 
Cremona group} associated to the 
flag $\bff$ and denote it by 
$Cr_2(\Bbbk,\bff)$. Note that 
$Cr_2(\CC,qxy)$ $(=\Aut_{Poi}(K_q))$ 
was computed by Blanc in 
\cite[Theorem 1]{Bl} (solving a 
conjecture of Usnich \cite{Us}). 
However, it is an open problem to 
describe $Cr_2(\Bbbk, 1)$ 
$(=\Aut_{Poi}(\Kweyl))$.  The group 
$Cr_2(\Bbbk, 1)$ does contain the 
automorphism group of the Weyl 
Poisson algebra (i.e., $\Bbbk[x,y]$ 
with $\{x,y\} = 1$), which was 
determined in the famous paper of 
Jung \cite{Jung}, but 
$Cr_2(\Bbbk, 1)$ is much larger.

In this paper we will compute 
$Cr_2(\Bbbk,\bff)$ for several 
classes of flags $\bff$, see 
Proposition \ref{zzpro4.8}, Corollary \ref{zzcor5.7}, and Theorems 
\ref{zzthm6.6}, \ref{zzthm7.5}.  
In the families (1) and (2), 
$Cr_2(\Bbbk,\bff)$ is infinite, 
while it could be finite in 
families (3), (4), and may even 
be trivial (Corollary \ref{zzcor7.7}).

A Poisson field $K$ is said to have 
the \emph{Dixmier property} if 
every Poisson endomorphism of $K$ 
is an automorphism.  We have named 
it so in reference to the 
\emph{Dixmier Conjecture} which 
asserts that all endomorphisms of 
a Weyl algebra $A_n(\Bbbk)$ in 
characteristic $0$ should be 
automorphisms.  We ask

\begin{question}[Dixmier Problem]
\label{zzque0.9}
For which nonzero $\bff$ in $\Bbbk(x,y)$ 
does $K\{\bff\}$ have the Dixmier 
property?
\end{question}

It is well-known that both $\Kweyl$ 
and $K_q$ do not have the Dixmier 
property, and neither do the other 
Poisson fields in the family (1): 
Corollary \ref{zzpro4.12}. However we 
prove that several classes of $K\{\bff\}$ 
within the families (2), (3), (4) have 
the Dixmier property: Corollary 
\ref{zzcor5.3} and Propositions 
\ref{zzpro6.3}(4), \ref{zzpro7.3}(3).

In the families (1)--(4), all $\bff$ 
are polynomials in $\Bbbk[x,y]$.
The question arises, whether all 
$K\{\bff\}$ have this form:

\begin{question}[Polynomial Flag Problem]
\label{zzque0.10}
Let $\bff$ be a nonzero element in 
$\Bbbk(x,y)$. Is there a polynomial 
$\bfg\in \Bbbk[x,y]$ such that
$K\{\bff\}\cong K\{\bfg\}$? 
Equivalently, does every $K\{\bff\}$
have finite flag height?
\end{question}

This question has a negative answer: 
Example \ref{zzexa8.5}.

Finally, we repeat one problem from 
the paper \cite{GZ2}.  As in 
\cite{HTWZ1}, the \emph{Sklyanin 
Poisson field} of GK-dimension $2$, 
here denoted $K_{Skly}$, is 
$\Fract(A_{\Omega}/ \langle 
\Omega-1 \rangle)$ where $A := 
\Bbbk[x,y,z]$ with the Poisson bracket
$$
\{f,g\}_\Omega := \det \begin{pmatrix}
f_x&f_y&f_z\\ g_x&g_y&g_z\\ \Omega_x&\Omega_y&\Omega_z \end{pmatrix} \qquad \forall\; f,g \in A
$$
and $\Omega := x^3+y^3+z^3+ \lambda xyz$ 
for some $\lambda\in\kx$ such that 
$\lambda^3 \ne - 27$. 

\begin{conjecture}
\label{zzcon0.11}
Every Poisson field $K\{\bff\}$ contains 
either $\Kweyl$, or some $K_{q}$, or 
some $K_{Skly}$.
\end{conjecture}

This conjecture holds for family (1), 
by Corollary \ref{zzcor4.2}(2).
\medskip

The paper is organized as follows. 
Section 1 contains some initial 
classification and subfield results, 
while Section 2 develops some calculations 
with logarithmic derivatives of 
rational functions that are needed in 
later sections.  Section 3 is concerned 
with Poisson valuations and related 
invariants from \cite{HTWZ1, HTWZ2}.  
The Automorphism, Isomorphism, Embedding, 
and Dixmier Problems for the families 
(1)--(4) are studied in Sections 4-7, 
while Section 8 gives the solution to 
the Polynomial Flag Problem.  Finally, 
Section 9 is an appendix with 
some general Poisson algebra results.

\subsection{Conventions}
\label{zzsec0.1}
Throughout the paper, $\Bbbk$ is a fixed
base field of characteristic zero, 
all algebras and fields are commutative 
$\Bbbk$-algebras, and $K$ is a field 
extension of $\Bbbk$.  We do not 
assume $\Bbbk$ algebraically closed 
except where necessary. When $K$ is 
not denoted a general (Poisson) field, 
we assume that $K = \Bbbk(x,y)$ is a 
rational function field in two 
variables.

For any $\bff \in K= \Bbbk(x,y)$, 
there is a unique Poisson bracket on 
$K$ such that $\{x,y\} = \bff$ 
(Corollary \ref{zzcor9.4}). The Poisson 
field with this bracket will be denoted 
$K\{\bff\}$.

The symbol $\cong$ between Poisson fields 
$K$, $L$ means that $K$ and $L$ are 
Poisson-iso\-mor\-phic.

We denote the Poisson bracket on $\Kweyl$ 
by $\{-,-\}_w$, and we recall that this 
bracket is given by
$$
\{f,g\}_w = \frac{\partial f}{\partial x} \frac{\partial g}{\partial y} 
- \frac{\partial f}{\partial y} 
\frac{\partial g}{\partial x} 
\qquad \forall\; f,g \in K.
$$

%%%%%%%%%%%%%%%%%%%%%%%%%%%%%%%%
\section{Some initial results}  
\label{zzsec1}

We give some classification results and 
some instances of Poisson subfields that 
involve only direct bracket calculations 
and changes of variables.

%%%%%%%%%%%%%%%%%%%%

\subsection{Poisson fields isomorphic to 
the Weyl or $q$-skew Poisson fields}  
\label{zzsec1.1} 

Dumas showed in \cite{Dumas} that when 
$\bff \in \Bbbk[x,y]$ is a nonzero 
polynomial with (total) degree at most 
$2$, the Poisson field $K\{\bff\}$ is 
isomorphic to either $\Kweyl$ or $K_q$ 
for some $q \in \kx$.  We give some 
extensions of Dumas' arguments and 
consequences.

\begin{proposition}  
\label{zzpro1.1}
Let $K = K\{\bff\}$ with $\bff \in K^\times$.
\begin{enumerate}
\item[(1)] 
If $\bff \in \Bbbk(x)$ or 
$\bff \in \Bbbk(y)$, then $K \cong \Kweyl$.
\item[(2)] 
If $\bff = f+g$ with $f \in \Bbbk(x)$, 
$g \in \Bbbk[y]$, and $\deg g \le 1$,  
then $K \cong \Kweyl$.
\item[(3)] 
Suppose $\bff = xf+g$ with $f,g \in 
\Bbbk(y)$ and $f \ne 0$. Then 
$K = \Bbbk(x',y)$ with $\{x',y\} = x'f$.
\begin{itemize}
\item 
If $f \in \kx$, then $K \cong \Kweyl$.
\item 
If $f = qy+r$ with $q \in \kx$ and 
$r \in \Bbbk$, then $K \cong K_q$.
\end{itemize}
\item[(4)] 
Suppose $\bff = x^2f + xg + h$ with 
$f,g,h \in \Bbbk(y)$ and $f \ne 0$. 
Then $K = \Bbbk(x'',y)$ with 
$\{x'',y\} = (x'')^2 + h''$ and 
$h'' := fh - \frac 14 g^2 \in \Bbbk(y)$.
\begin{itemize}
\item 
If $h'' \in \Bbbk[y]$ with 
$\deg h'' \le 1$, then $K \cong \Kweyl$.
\item 
If $h'' = \alpha y^2 + \beta y 
+ \gamma$ with $\alpha \in \kx$ 
and $\beta,\gamma \in \Bbbk$, and if 
$- \alpha$ has square roots in $\Bbbk$, 
then $K \cong K_q$ with 
$q := 2 \sqrt{- \alpha} \in \kx$.
\end{itemize}
\end{enumerate}
\end{proposition}

\begin{proof}
(1) If $\{x,y\} = f \in \Bbbk(x)$, 
take $y' := f^{-1} y$. Then 
$K = \Bbbk(x,y')$ and $\{x,y'\} 
= f^{-1} \{x,y\} = 1$, so 
$K \cong \Kweyl$. The case when 
$\{x,y\} \in \Bbbk(y)$ is symmetric.

(2) If $\deg g \le 0$, then 
$\{x,y\} \in \Bbbk(x)$, and so  
$K \cong \Kweyl$ by (1). Now suppose 
that $g = \beta y + \gamma$ with 
$\beta \in \kx$, $\gamma \in \Bbbk$. 
If $y' := f+g$, then $K = \Bbbk(x,y')$ 
and 
$$\{x,y'\} = \beta \{x,y\} = \beta y',$$
and again (1) implies that 
$K \cong \Kweyl$.

(3) Set $x' := xf+g$. Then $K = 
\Bbbk(x',y)$ and $\{x',y\} = 
\{x,y\} f = x' f$.

When $f \in \kx$, we have $\{x',y\} 
\in \Bbbk(x')$, so $K \cong \Kweyl$ 
by (1).

Now suppose $f = qy+r$ with $q \in 
\kx$ and $r \in \Bbbk$. Then $K = 
\Bbbk(x',f)$ and $\{x',f\} = 
q \{x',y\} = q x' f$, so $K \cong K_q$.

(4) First set $x' := x f$. Then 
$K = \Bbbk(x',y)$ and 
$$\{x',y\} = \{x,y\} f = x^2 f^2 
+ x gf + hf = (x')^2 + x' g + hf.$$
Now set $x'' := x' + \frac12 g$. Then 
$K = \Bbbk(x'',y)$ and 
$$\{x'',y\} = \{x',y\} = (x')^2 
+ x' g + hf = (x'')^2 - \tfrac14 g^2 + hf,$$
so we have the first part of (4).
The first bullet then follows from (2).

Suppose $h'' = \alpha y^2 + \beta y 
+ \gamma$ with $\alpha \in \kx$, 
$\beta,\gamma \in \Bbbk$. Then 
$h'' = \alpha ( y + \frac\beta{2\alpha} )^2 
+ ( \gamma - \frac{\beta^2}{4\alpha})$. 
If $y' := y + \frac\beta{2\alpha}$, then 
$K = \Bbbk(x'',y')$ and 
$$\{x'',y'\} = \{x'',y\} 
= (x'')^2 + h'' 
= (x'')^2 + \alpha (y')^2 + \gamma'$$
where $\gamma' := \gamma - 
\frac{\beta^2}{4\alpha}$. 

By assumption, there exists 
$\zeta \in \kx$ such that $\zeta^2 
= - \alpha$. If $x^* := x'' - \zeta y'$ 
and $y^* := x'' + \zeta y'$, then 
$K = \Bbbk(x^*,y^*)$ and
$$\{x^*,y^*\} = 2 \zeta \{x'',y'\} 
= 2 \zeta \bigl( (x'')^2 + \alpha (y')^2 
+ \gamma' \bigr) 
= 2 \zeta x^* y^* + 2 \zeta \gamma' \,.$$
Setting $q := 2 \zeta$ and 
$x^{**} := x^* + \gamma' (y^*)^{-1}$, we 
see that $K = \Bbbk(x^{**},y^*)$ and
$$\{x^{**},y^*\} = \{x^*,y^*\} 
= q x^* y^* + q \gamma' 
= q \bigl( x^{**} -  \gamma' (y^*)^{-1} \bigr) 
y^* + q \gamma' = q x^{**} y^*.$$
Thus, $K \cong K_q$ in this case.
\end{proof}

\begin{corollary}  
\label{zzcor1.2}
\cite[Proposition, \S1.2.1, p.12]{Dumas}
Assume that all elements of $\Bbbk$ have 
square roots in $\Bbbk$. Let $K = K\{\bff\}$ 
where $\bff\in K^{\times}$.
\begin{enumerate}
\item[(1)]
If $\bff$ is a polynomial in $\Bbbk[x,y]$ 
with total degree $\le 2$, then either 
$K \cong \Kweyl$ or $K \cong K_q$ for 
some $q \in \kx$.
\item[(2)]
If $\bff= qxy+ a x+ g(y)$ for 
$q\in \Bbbk^{\times}$, $a\in \Bbbk$, 
and $g(y)\in \Bbbk(y)$, then 
$K\cong K_q$.
\item[(3)]
If $\bff=a x^2 +bx + cy+d$ for 
$a \in \kx$ and $b,c,d \in \Bbbk$, then 
$K\cong \Kweyl$.
\end{enumerate}
\end{corollary}

\begin{proof}
(1) By assumption, $\{x,y\} = x^2 f_0 
+ x f_1 + f_2$ for some $f_i \in \Bbbk[y]$ 
with $\deg f_i \le i$. The cases $(f_0=f_1=0)$, $(f_0=0,\; f_1 \ne 0)$, and $(f_0 \ne 0)$ 
are covered by parts (1), (3), and (4) 
of Proposition \ref{zzpro1.1}, respectively.

(2) This is covered by Proposition 
\ref{zzpro1.1}(3).

(3) This is a special case of  
Proposition \ref{zzpro1.1}(2). 
\end{proof}

\begin{corollary}  
\label{zzcor1.3}
Let $K = K\{xf\}$ for some nonzero 
polynomial $f \in \Bbbk[y]$.
\begin{enumerate}
\item[(1)] 
If $f \in \kx$, then $K \cong \Kweyl$.
\item[(2)] 
If $f = qy+r$ with $q \in \kx$ and 
$r \in \Bbbk$, then $K \cong K_q$.
\item[(3)] 
Suppose $f = \alpha y^2 + \beta y + \gamma$ 
with $\alpha \in \kx$ and $\beta,\gamma 
\in \Bbbk$. 
\begin{itemize}
\item 
If $\beta^2 = 4 \alpha \gamma$, then 
$K \cong \Kweyl$.
\item 
If $\beta^2 \ne 4 \alpha \gamma$ and 
$\beta^2 - 4 \alpha \gamma$ has square 
roots in $\Bbbk$, then $K \cong K_q$ 
where $q = \sqrt{\beta^2 - 4 \alpha \gamma} 
\in \kx$.
\end{itemize}
\end{enumerate}
\end{corollary}

\begin{proof}
(1) and (2) follow from Proposition 
\ref{zzpro1.1}(1)(3).

(3) Write $f = \alpha ( y' )^2 + \gamma'$ 
where $y' := y + \frac\beta{2\alpha}$ 
and $\gamma' := \gamma - 
\frac{\beta^2}{4 \alpha}$. Then 
$K = \Bbbk(x,y')$ and
$$\{x,y'\} = \{x,y\} 
= x \bigl( \alpha ( y' )^2 + \gamma' \bigr).$$

$\bullet$ Suppose that $\beta^2 
= 4 \alpha \gamma$, so that $\gamma' = 0$. 
Then $K = \Bbbk( (y')^{-1}, x)$ and 
$$\{(y')^{-1}, x\} = -(y')^{-2} \{y',x\} 
= (y')^{-2} \alpha x (y')^2 =\alpha x,$$ 
whence $K \cong \Kweyl$ by Proposition 
\ref{zzpro1.1}(1).

$\bullet$ Suppose that $\beta^2 
\ne 4 \alpha \gamma$, so that 
$\gamma' \ne 0$. Then 
$K = \Bbbk(\gamma' x, xy')$ and
$$\{\gamma' x, xy'\} = \gamma' x \{x,y'\} 
= \gamma' x^2 \bigl( \alpha ( y' )^2 
+ \gamma' \bigr) = (\gamma' x)^2 
+ \gamma' \alpha (xy')^2,$$
so Proposition \ref{zzpro1.1}(4) 
implies that $K \cong K_q$ where
$q := 2 \sqrt{- \gamma' \alpha} 
= \sqrt{\beta^2 - 4 \alpha \gamma}$.
\end{proof}

What happens in Corollary \ref{zzcor1.3} 
when $\deg f \ge 3$? As we show later, the 
resulting Poisson fields need not be 
isomorphic to $\Kweyl$ or to any 
$K_q$ -- see Proposition \ref{zzpro4.6} 
and the remarks following Proposition 
\ref{zzpro5.2}, Lemma \ref{zzlem6.2} 
(after interchanging $x$ and $y$).  In 
particular, $K\{xy^3\}$ is not 
isomorphic to $\Kweyl$ or to any $K_q$.

%%%%%%%%%%%%%%%%%%%%

\subsection{Some Poisson subfields}
\label{zzsec1.2}
Understanding subfields of $K\{\bff\}$ is helpful for 
the Embedding Problem [Question \ref{zzque0.5}].
We exhibit a number of Poisson subfields of our standard 
Poisson fields. Here and later, we use the following 
observation without explicit mention:

\begin{obs} \cite[Lemma 1.6(1)]{HTWZ1}
\label{zzobs1.4}
Let $K$ be a Poisson 
field over $\Bbbk$ and $f,g \in K$ such that $\{f,g\} \ne 0$. Then $f$ 
and $g$ are algebraically independent over $\Bbbk$.
\end{obs}

For any rational function field in one variable, we abbreviate the standard derivative by $'$.  In particular, $f' := df/dx$ and $g' := dg/dy$ for $f \in \Bbbk(x)$ and $g \in \Bbbk(y)$.

\begin{examples}  
\label{zzexa1.5}  
\textbf{(1)} 
\emph{Let $K = \Kweyl$. For any $f \in \Bbbk(x) \setminus \Bbbk$, 
the subfield $\Bbbk(f,y/f')$ is a Poisson subfield of $K$ 
isomorphic to $K$. If $f \in \Bbbk[x]$, then 
$[K:\Bbbk(f,y/f')] = \deg f$.}

\emph{Similarly, for any $g \in \Bbbk(y) \setminus \Bbbk$, the 
subfield $\Bbbk(x/g',g)$ is a Poisson subfield of $K$ 
isomorphic to $K$. If $g\in \Bbbk[y]$, then 
$[K:\Bbbk(x/g',g)] = \deg g$.}

In the first case, set $L := \Bbbk(f,y/f')$. Since 
$$
\{f,y/f'\} = f'\cdot (1/f') - 0\cdot (- y (f')^{-2} f'') = 1,
$$
$L$ is a Poisson subfield isomorphic to $K$. 

Now suppose that $f \in \Bbbk[x]$, and note that $L(x) = K$. 
If $\deg f = 1$, then $x \in L$ and so $[K:L] = 1$. Next, 
assume that $\deg f = n > 1$. We have a polynomial 
$f(t) - f(x) \in L[t]$ with $t$-degree $n$ having $x$ as a 
root, so $[K:L] \le n$. If this inequality is strict, 
$\sum_{i=0}^m b_i x^i = 0$ for some $b_i \in L$ with $m < n$ 
and $b_m \ne 0$. After multiplying this equation by a suitable 
element of $\Bbbk[f,y/f']$, we may assume that all 
$b_i \in \Bbbk[f,y/f']$. However, this contradicts the fact 
that $\Bbbk[x,y/f']$ is a free module over $\Bbbk[f,y/f']$ 
with basis $(1,x,\dots,x^{n-1})$. Therefore $[K:L] = n$, as 
claimed.

The second case is proved in the same fashion.

\textbf{(2)}
\emph{Let $K = \Kweyl$. For any nonzero integer $a$, the 
subfield $\Bbbk(x^a,xy)$ is a Poisson subfield of $K$ 
isomorphic to $K$, with $[K:\Bbbk(x^a,xy)] = |a|$. 
Similarly, $\Bbbk(xy,y^a)$ is a Poisson subfield of $K$ 
isomorphic to $K$, with $[K:\Bbbk(xy,y^a)] = |a|$.}

In the first case, set $f(x) = x^a$ and note that
$$
\Bbbk(x^a,xy) = \Bbbk(f(x), af(x)y/f'(x)) 
= \Bbbk(f(x), y/f'(x)),
$$
and so $\Bbbk(x^a,xy) \cong K$ by part (1). If $a>0$, 
part (1) also shows that  $[K:\Bbbk(x^a,xy)] = a$. If 
$a<0$, we have $[K:\Bbbk(x^a,xy)] = - a$ because 
$\Bbbk(x^a,xy) = \Bbbk(x^{-a},xy)$.

The second case is similar.

\textbf{(3)}
\emph{Let $K = K_q$ for some $q \in \kx$. For any 
integers $a,b,c,d$ with $ad \ne bc$, the subfield 
$\Bbbk(x^ay^b, x^cy^d)$ is a Poisson subfield of $K$ 
isomorphic to $K_{(ad-bc)q}$, and 
$[K: \Bbbk(x^ay^b, x^cy^d)] = |ad-bc|$.}

The Poisson statements are immediate from the fact that
\begin{equation}
\label{E1.5.1}\tag{E1.5.1}
\begin{aligned}
\{ x^ay^b, x^cy^d \} &= \bigl( (ax^{a-1}y^b) (dx^cy^{d-1}) 
- (bx^ay^{b-1}) (cx^{c-1}y^d) \bigr) \{x,y\}  \\
&= (ad-bc) q (x^ay^b) (x^cy^d).
\end{aligned}
\end{equation}
For any integral matrix 
$A=\left( \begin{smallmatrix} r&s\\ t&u \end{smallmatrix}\right)$ 
with $\det A \ne 0$, there is a $\Bbbk$-algebra endomorphism 
$\phi_A$ of $K$ sending $x \mapsto x^ry^s$ and 
$y \mapsto x^ty^u$. Moreover, for any $B \in M_2(\ZZ)$ with 
$\det B \ne 0$, we have $\phi_A \circ \phi_B = \phi_{BA}$. 
When $B \in GL_2(\ZZ)$, the map $\phi_B$ is invertible, 
whence $\phi_A(K) = \phi_{BA}(K)$.

Now set 
$A:=\left(\begin{smallmatrix} a&b\\ c&d\end{smallmatrix}\right)$, 
so that $\phi_A(K) = \Bbbk(x^ay^b, x^cy^d)$. There is a matrix 
$B \in GL_2(\ZZ)$, a composition of matrices corresponding to 
suitable row operations, such that $BA$ is upper triangular, 
say 
$BA=\left( \begin{smallmatrix} r&s\\ 0&u \end{smallmatrix}\right)$. 
Then $\Bbbk(x^ay^b, x^cy^d) = \phi_{BA}(K) = \Bbbk(x^ry^s, y^u)$ 
and $ru = \det(BA) = \pm(ad-bc)$. It follows that 
$[K: \Bbbk(x^ay^b, x^cy^d)] = |r|\cdot|u| = |ad-bc|$.

\textbf{(4)}
\emph{Let $K = K\{qx^{ac+1}y^{bd+1}\}$ for some $q \in \kx$ and 
some integers $a,b,c,d$ with $c,d \ne 0$. Then $\Bbbk(x^c,y^d)$ 
is a Poisson subfield of $K$ isomorphic to 
$K\{qcd x^{a+1} y^{b+1}\}$, and $[K:\Bbbk(x^c,y^d)] = |cd|$. If 
there exist $\xi,\zeta \in \kx$ such that $\xi^a \zeta^b = qcd$, 
then $\Bbbk(x^c,y^d) \cong K\{x^{a+1} y^{b+1}\}$.}

For the main case, we observe that
$$
\{x^c, y^d\} = c x^{c-1} d y^{d-1} \{x,y\} 
= cdq x^{ac+c} y^{bd+d} = qcd (x^c)^{a+1} (y^d)^{b+1}.
$$
If there exist $\xi,\zeta \in \kx$ with $\xi^a \zeta^b 
= qcd$, then we have
\begin{align*}
\{ \xi x^c, \zeta y^d \} 
&= \xi \zeta qcd (x^c)^{a+1} (y^d)^{b+1} 
= \xi^{-a} \zeta^{-b} qcd (\xi x^c)^{a+1} (\zeta y^d)^{b+1}  \\
&= (\xi x^c)^{a+1} (\zeta y^d)^{b+1},
\end{align*}
and the subsidiary case follows. The degree equality is clear.

\textbf{(5)}
\emph{Let $K = K\{(qx^a-b)xy\}$ for some $q,b \in \kx$ and 
some nonzero integer $a$. Then $\Bbbk(x^a,y)$ is a Poisson 
subfield of $K$ isomorphic to $K_{ab}$, and 
$[K: \Bbbk(x^a,y)] = |a|$.}

Observe that $\Bbbk(x^a,y) = \Bbbk(bx^{-a}-q,y)$ and
$$
\{ bx^{-a}-q, y \} = -ba x^{-a-1} \{x,y\} 
= - ba x^{-a-1} (qx^a-b)xy = ab (bx^{-a}-q) y.
$$
The result follows.
\end{examples}

\begin{lemma}
\label{zzlem1.6}
Let $K = K\{f(x)g(y)\}$ for some nonzero $f \in \Bbbk(x)$ 
and $g \in \Bbbk(y)$, and let $r(t) \in \Bbbk(t)^\times$ 
and $\lambda \in \kx$. If there exists $p \in \Bbbk(x)^\times$ 
with $r(p) \ne 0$ and $\frac{p'(x)}{r(p(x))} = \frac\lambda{f(x)}$, then 
$\Bbbk(p,y)$ is a Poisson subfield of $K$ isomorphic to 
$K\{\lambda r(x)g(y)\}$. 
\end{lemma}

\begin{proof}
If there exists $p \in \Bbbk(x)$ with 
$\frac{p'(x)}{r(p(x))} = \frac\lambda{f(x)}$, then 
$$
\{p,y\} = p'(x) \{x,y\} = p'(x) f(x) g(y) = \lambda r(p) g(y),
$$
whence $\Bbbk(p,y) \cong K\{\lambda r(x)g(y)\}$. 
\end{proof}

\begin{corollary}
\label{zzcor1.7}
Let $K = K\{f(x)g(y)\}$ for some nonzero $f \in \Bbbk(x)$ 
and $g \in \Bbbk(y)$, and let $\lambda,\mu \in \kx$.
\begin{enumerate}
\item
Let $m \in \ZZ$. If there exists $p \in \Bbbk(x)^\times$ with 
$\frac{p'(x)}{p(x)^m} = \frac\lambda{f(x)}$, then $\Bbbk(p,y)$ 
is a Poisson subfield of $K$ isomorphic to 
$K\{\lambda x^mg(y)\}$. 
\item
If there exists $p \in \Bbbk(x)^\times$ with 
$p'(x) = \frac\lambda{f(x)}$, then $\Bbbk(p,y)$ is a Poisson 
subfield of $K$ isomorphic to $\Kweyl$. 
\item
If there exists $p \in \Bbbk(x)^\times$ with 
$\frac{p'(x)}{p(x)} = \frac\lambda{f(x)}$, then $\Bbbk(p,y)$ 
is a Poisson subfield of $K$ isomorphic to $K\{\lambda xg(y)\}$. 
Similarly, if there exists $q \in \Bbbk(y)^\times$ with 
$\frac{q'(y)}{q(y)} = \frac\mu{g(y)}$, then $\Bbbk(x,q)$ is a 
Poisson subfield of $K$ isomorphic to $K\{\mu f(x)y\}$. In case 
such $p$ and $q$ both exist, $\Bbbk(p,q)$ is a Poisson 
subfield of $K$ isomorphic to $K\{\lambda \mu xy\}=K_{\lambda\mu}$.
\end{enumerate}
\end{corollary}

\begin{proof}
(1) Apply Lemma \ref{zzlem1.6} with $r(t) = t^m$.

(2) Take $m=0$ and apply part (1) together with Proposition \ref{zzpro1.1}(1).

(3) Take $m=1$ and apply part (1).
\end{proof}

\begin{remarks}
\label{zzrem1.8}
In particular, Corollary \ref{zzcor1.7}(2) shows that if $K = K\{x^m g(y)\}$ where $m \in \ZZ \setminus \{1\}$ and
$g \in \Bbbk(y)^\times$, then 
$\Bbbk(x^{1-m},y)$ is a Poisson subfield of $K$ isomorphic to 
$\Kweyl$. 

Part (3) may be applied in the case $K = K\{f(x)y\}$ where 
$f(x) = x (x-\alpha_1) (x-\alpha_2)$ for some distinct 
$\alpha_1,\alpha_2 \in \kx$ such that $\alpha_1/\alpha_2 \in \QQ$. 
In this case, there exist nonzero $s_1,s_2 \in \ZZ$ such that 
$s_1\alpha_1 + s_2\alpha_2 = 0$. Set $s_0 := - (s_1+s_2)$ and 
$p(x) := x^{s_0} (x-\alpha_1)^{s_1} (x-\alpha_2)^{s_2}$.  One checks that
$$
\frac{p'(x)}{p(x)} = \frac{s_0}{x} + \frac{s_1}{x-\alpha_1} + \frac{s_2}{x-\alpha_2} = 
\frac\lambda{f(x)} \,,
$$
where $\lambda := s_0 \alpha_1 \alpha_2$ (cf.~Remark \ref{zzrem2.3}). 
Corollary \ref{zzcor1.7}(3) then shows that $\Bbbk(p,y)$ is a 
Poisson subfield of $K$ isomorphic to $K_\lambda$.
\end{remarks}

%%%%%%%%%%%%%%%%%%%%%%%%%%%%%%%%
\section{Logarithmic derivatives}
\label{zzsec2}

Various calculations with Poisson brackets
involve terms which are formal logarithmic
derivatives -- $f'(x)/f(x)$ or $g'(y)/g(y)$ 
for $f \in \Bbbk(x)^\times$ or $g \in 
\Bbbk(y)^\times$.  To deal with these, we
develop some elementary calculus facts 
about logarithmic derivatives.

Throughout this section, $F$ stands for a
general field of characteristic zero. As 
noted above, for $s(t)$ in a rational 
function field $F(t)$, we write $s'(t)$ 
for $ds/dt$.

\begin{lemma}  
\label{zzlem2.1} 
Let $F(t) \subseteq \overline{F}(t)$ be 
rational function fields over $F$ and 
over its algebraic closure $\overline{F}$.
\begin{enumerate}
\item[(1)]
If $s \in F(t)^\times$, then
\begin{equation} 
\label{E2.1.1}\tag{E2.1.1}
\dfrac{s'}{s} 
= \frac{z_1}{t-a_1} + \cdots + \frac{z_m}{t-a_m}
\end{equation}
in $\overline{F}(t)$, for some $z_i \in \ZZ$
and some distinct $a_i \in \overline{F}$.
\item[(2)]
If $s \in F(t)^\times$ and
\begin{equation} 
\label{E2.1.2}\tag{E2.1.2}
\dfrac{s'}{s} = \frac{c_1}{(t-b_1)^{r_1}} 
+ \cdots + \frac{c_n}{(t-b_n)^{r_n}}
\end{equation}
in $F(t)$, for some $c_j \in F^\times$, 
some distinct $b_j \in F$, and some 
$r_j \in \Znn$, then $c_j \in \ZZ$ 
and $r_j = 1$ for all $j=1,\dots,n$. 
Moreover,
$$s = \alpha (t-b_1)^{c_1} \cdots (t-b_n)^{c_n}$$
for some $\alpha \in F^\times$.
\end{enumerate}
\end{lemma}

\begin{proof} 
(1) Write $s = f/g$ for some nonzero relatively
prime $f,g \in F[t]$. Without loss of
generality, $g$ is monic. Let $a_1,\dots,a_n$ 
be the distinct roots of $f$ in $\overline{F}$, 
and let $a_{n+1},\dots,a_m$ be the distinct 
roots of $g$ in $\overline{F}$. Since $f$ 
and $g$ are relatively prime, 
$a_1,\dots,a_n,\dots,a_m$ are all
distinct. Write 
$f = \alpha (t-a_1)^{z_1} \cdots (t-a_n)^{z_n}$
for some 
$\alpha \in F^\times$ and $z_i \in \ZZ_{>0}$, 
and write 
$g = (t-a_{n+1})^{-z_{n+1}} \cdots 
(t-a_m)^{-z_m}$ for some $z_j \in \ZZ_{<0}$.
Then
$$s = \alpha (t-a_1)^{z_1} 
\cdots (t-a_n)^{z_n} (t-a_{n+1})^{z_{n+1}} 
\cdots (t-a_m)^{z_m}\,.$$
Since $(\alpha^{-1}s)'/(\alpha^{-1}s) = s'/s$, 
\eqref{E2.1.1} follows by a direct computation.

(2) By comparing \eqref{E2.1.1} and 
\eqref{E2.1.2}
in $\overline{F}(t)$, one sees that $r_i=1$ 
for all $i$ and that the set $\{a_i\}$ is 
equal to the set $\{b_j\}$. As a consequence,
$m=n$ and $c_i=z_i$ after we set $a_i=b_i$
for all $1\leq i\leq m=n$. Now let 
$u := (t-b_1)^{c_1} \cdots (t-b_n)^{c_n}$ and 
observe that $u'/u = s'/s$. Then $(s/u)'=0$,
and therefore $s/u \in F^\times$.
\end{proof}

\begin{corollary}  
\label{zzcor2.2}
Let $\gamma \in F^\times$ and
$$
f(t) = \gamma (t-a_1) \cdots (t-a_n) \in F[t]
$$
for some $a_i \in F$.
Then $\frac1f = \frac{s'}s$ for some $s \in F(t)^\times$ 
if and only if $n>0$, the $a_i$ are distinct, and
\begin{equation} 
\label{E2.2.1}\tag{E2.2.1}
\gamma^{-1} \prod_{\substack{j=1\\ j\ne i}}^n (a_i-a_j)^{-1} 
\in \ZZ \qquad \forall\; i=1,\dots,n.
\end{equation}
\end{corollary}

\begin{proof} Let $F(t) \subseteq \overline{F}(t)$ as in Lemma 
\ref{zzlem2.1}.

$(\Longrightarrow)$: By Lemma \ref{zzlem2.1}(1), 
\begin{equation} 
\label{E2.2.2}\tag{E2.2.2}
\frac1f = \sum_{j=1}^m \frac{z_j}{t-b_j}
\end{equation}
for some $z_j \in \ZZ$ and some distinct $b_j \in 
\overline{F}$. We may assume that all $z_j \ne 0$. 
Set $g(t) := \prod_{j=1}^m (t-b_j)$ and 
$g_j(t) := \frac{g(t)}{t-b_j}$ for $j \in [1,m]$. 
Multiply \eqref{E2.2.2} by $fg$ to get
\begin{equation} 
\label{E2.2.3}\tag{E2.2.3}
g = f \sum_{j=1}^m z_j g_j \,.
\end{equation} 
Consequently, $f \mid g$ in $\overline{F}[t]$. On the 
other hand, if we substitute $t=b_i$ in \eqref{E2.2.3} 
for some $i \in [1,m]$, we find that
$$
0 = g(b_i) = f(b_i) z_i \prod_{j\ne i} (b_i-b_j),
$$
whence $f(b_i) = 0$. Therefore $n=m$ and, after 
reindexing, $b_i = a_i$ for all $i$, whence $f = \gamma g$. 
In particular, $n>0$ and the $a_i$ are distinct. Now 
divide \eqref{E2.2.3} by $f$, to get
$$
\gamma^{-1} = \sum_{j=1}^m z_j g_j \,.
$$
Substituting $t=a_i$, we get 
$\gamma^{-1} = z_i \prod_{j\ne i}(a_i-a_j)$, establishing 
\eqref{E2.2.1}.

$(\Longleftarrow)$: Set 
$z_i := \gamma^{-1} \prod_{j\ne i} (a_i-a_j)^{-1} \in \ZZ$ 
for $i \in [1,n]$ and 
$h(t) := \sum_{i=1}^n z_i \frac{f}{t-a_i} \in F[t]$. Note 
that $\deg h < n$. For $i \in [1,n]$, we have
$$
h(a_i) = z_i \gamma \prod_{j\ne i} (a_i-a_j) = 1.
$$
Thus, the polynomial $h(t)-1$ has $n$ distinct roots, so it 
must be zero. Therefore
$$
\frac1f = \frac{h}{f} = \sum_{i=1}^n \frac{z_i}{t-a_i} 
= \frac{s'}{s} \,,
$$
where $s := \prod_{i=1}^n (t-a_i)^{z_i}$.
\end{proof}

\begin{remark}  
\label{zzrem2.3} 
In the linear case of Corollary \ref{zzcor2.2}, that 
is, $f(t) = \gamma(t-a_1)$, the only requirement is that 
$\gamma^{-1} \in \ZZ$. In the quadratic case, the two 
instances of \eqref{E2.2.1} are negatives of each other, 
so the requirements are $a_1 \ne a_2$ and 
$\gamma^{-1} (a_1-a_2)^{-1} \in \ZZ$. Given $a_1 \ne a_2$, 
one can always choose $\gamma$ to satisfy the second condition.

In higher degrees, the $a_i$ must be linearly dependent 
over $\QQ$, as follows. We claim that if the conditions 
in the corollary hold and $n \ge 3$, then the nonzero 
integers $z_i := \gamma^{-1} \prod_{j\ne i} (a_i-a_j)^{-1}$ 
must satisfy
\begin{equation}  
\label{E2.3.1}\tag{E2.3.1}
z_1 + \cdots + z_n = 0 \qquad\text{and}
\qquad z_1a_1 + \cdots + z_na_n = 0.
\end{equation}
As in the proof of Corollary \ref{zzcor2.2}, we have
$$
\frac1f = \sum_{i=1}^n \frac{z_i}{t-a_i} \,,
$$
and consequently
$$
\sum_{i=1}^n z_i \prod_{j\ne i} (t-a_j) = \gamma^{-1}.
$$
Since the coefficients of $t^{n-1}$ and $t^{n-2}$ on the 
left side of this equation must be zero, we find that 
$\sum_i z_i = 0$ and $\sum_i z_i \sum_{j\ne i} a_j = 0$. But
$$
\sum_i z_i \sum_{j\ne i} a_j = 
\sum_j \biggl( \sum_{i\ne j} z_i \biggr) a_j = - \sum_j z_j a_j \,,
$$
and therefore $\sum_j z_ja_j = 0$.

For an example, suppose $a_1=0,a_2,a_3$ are distinct elements 
of $\Bbbk$ with $a_2/a_3 \in \QQ$. There are nonzero integers 
$z_2$, $z_3$ such that $z_2a_2 + z_3a_3 = 0$. Set 
$z_1 := - (z_2+z_3)$, so that \eqref{E2.3.1} holds, and set 
$s(t) := t^{z_1} (t-a_2)^{z_2} (t-a_3)^{z_3}$. In this case, 
$\frac{s'}{s} = \frac1f$ where $f(t) = (z_1a_2a_3)^{-1}t (t-a_2) (t-a_3)$. 
(Cf.~Remarks \ref{zzrem1.8} and \ref{zzrem5.4} 
for an application.)
\end{remark}

%%%%%%%%%%%%%%%%%%%%%%%%%%
\section{Valuations and associated invariants}
\label{zzsec3}

In \cite{HTWZ1, HTWZ2}, the authors 
introduced the notion of a Poisson valuation 
which will be used in later sections. We 
first recall some definitions, lemmas, theorems 
from \cite{HTWZ1, HTWZ2}. Let $w$ denote an 
integer.

\begin{definition}
\label{zzdef3.1}
Recall that a \emph{discrete valuation} on an 
algebra $A$ is a map
$$\nu: A\to {\ZZ}\cup\{\infty\}$$
such that for all $a,b \in A$,
\begin{enumerate}
\item[(1)]
$\nu(a)=\infty$ if and only if $a=0$,
\item[(2)]
$\nu(ab)=\nu(a)+\nu(b)$,
\item[(3)]
$\nu(a+b)\geq \min(\nu(a),\nu(b))$.
\end{enumerate} 
It is well known that 
\begin{enumerate}
\item[($3^+$)] $\nu(a+b) = \min(\nu(a),\nu(b))$ 
when $\nu(a) \ne \nu(b)$.
\end{enumerate}

Now assume that $A$ is a Poisson algebra over 
a base field $\Bbbk$, and let $w \in \ZZ$.  
A discrete valuation $\nu$ on $A$ is called 
a (\emph{Poisson}) \emph{$w$-valuation} on $A$ if
\begin{enumerate}
\item[(4)]
$\nu(a)=0$ for all $a\in \Bbbk^{\times}$,
\item[(5)]
$\nu(\{a,b\})\geq \nu(a)+\nu(b)-w$ for all 
$a,b\in A$.
\end{enumerate} 

A $w$-valuation $\nu$ of $A$ is called 
{\it trivial} if $\nu(a)=0$ for all nonzero 
$a\in A$.  (This requires that either 
$w \ge 0$ or the Poisson bracket on $A$ 
is trivial.)  It is called {\it classical} 
if $\nu(\{a,b\})> \nu(a)+\nu(b)-w$ for all 
nonzero $a,b \in A$.  Of course, any 
$w$-valuation becomes classical when viewed 
as a $(w+1)$-valuation.  On the other hand, 
if the Poisson bracket on $A$ is nontrivial 
and $\nu$ is any $w$-valuation on $A$, there 
is a largest nonnegative integer $l$ such 
that $\nu(\{a,b\})\geq \nu(a)+\nu(b)-w+l$ 
for all nonzero $a,b \in A$, and $\nu$ is a 
non-classical $(w-l)$-valuation.

Let ${\mathcal V}_w(A)$ denote the set of 
non-trivial $w$-valuations of $A$. 
\end{definition}

\begin{remark}  
\label{zzrem3.2}
In \cite[Theorem 4.3(3)]{HTWZ1}, it is 
proved that if $K$ is a finitely generated 
field extension of $\Bbbk$, and $K$ is the
quotient field of a Poisson domain with 
GK-dimension at least $2$, then there are
infinitely many nontrivial $1$-valuations on $K$.

On the Poisson fields $K\{\bff\}$, some 
$w$-valuations may be constructed according 
to a standard recipe which we give in 
Lemma \ref{zzlem3.4}.
\end{remark}

\begin{definition}
\label{zzdef3.3}
Let $A$ be an algebra. Let
${\FF}:= (F_{i})_{i\in {\ZZ}}$
be a descending chain of $\Bbbk$-subspaces 
of $A$. 
\begin{enumerate}
\item[(1)] 
We say ${\FF}$ is a (\emph{descending})
\emph{filtration} of $A$ if it satisfies
\begin{enumerate}
\item[(a)]
$1 \in F_0\setminus F_{1}$, 
\item[(b)]
$F_i F_j\subseteq F_{i+j}$ for all $i,j$,
\item[(c)]
$\bigcup_{i\in {\ZZ}} (F_i\setminus F_{i+1})
=A\setminus \{0\}$.
\end{enumerate}
\item[(2)]
Suppose that $A$ is a Poisson algebra and 
that ${\FF}$ is a filtration of $A$. If 
further
\begin{enumerate}
\item[(d)]
$\{F_{i},F_{j}\}\subseteq F_{i+j-w}$ for 
all $i,j\in {\ZZ}$, 
\end{enumerate}
then ${\FF}$ is called a (\emph{Poisson})
\emph{$w$-filtration} of $A$.
\end{enumerate}
\end{definition}

Given a discrete valuation $\nu$ on $A$, we 
can define a descending filtration
${\FF}^{\nu}:= (F^{\nu}_i)_{i\in {\ZZ}}$ on 
$A$ by
\begin{equation}
\notag
F^{\nu}_i:=\{a\in A\mid \nu(a)\geq i\}, 
\quad {\text{for all $i \in \ZZ$}}.
\end{equation}
The {\it associated graded ring of $\nu$} 
is defined to be the associated graded ring 
of the filtration $\FF^\nu$, namely 
\begin{equation}
\notag
\gr_{\nu}(A):= \bigoplus_{i\in {\ZZ}} F^{\nu}_i/F^{\nu}_{i+1} \,.
\end{equation}
By routine verification, $\gr_{\nu}(A)$ is 
a domain.

\begin{lemma}  
\label{zzlem3.4}
Given $z_1,z_2 \in \ZZ$, there is a 
discrete valuation $\nu$ on $\Bbbk(x,y)$
with the following properties:
\begin{enumerate}
\item[(1)] 
$\nu(x) = z_1$ and $\nu(y) = z_2$.
\item[(2)]
If $u = \sum_{l=1}^m \lambda_l x^{a_l} y^{b_l} \in \Bbbk[x,y]$ for some $\lambda_l \in \kx$ and some distinct pairs $(a_l,b_l) \in \Znn^2$, then $\nu(u) = \min( a_lz_1+b_lz_2 \mid l \in [1,m] )$.
\item[(3)]
If $\bff \in \Bbbk(x,y)^\times$ and 
$w \in \ZZ$ such that $\nu(\bff) \ge 
z_1+z_2-w$, then $\nu$ is a $w$-valuation 
on $K\{\bff\}$. 
\end{enumerate}
\end{lemma}

\begin{proof}
For $i \in \ZZ$, let $F_i$ be the $\Bbbk$-span in $A := \Bbbk[x,y]$ of the set
$$
\{ x^ay^b \mid a,b \in \Znn\ \text{and}\ az_1+bz_2 \ge i\}.
$$
Clearly the family $\FF := (F_i)_{i\in\ZZ}$ satisfies conditions (b) and (c) of Definition \ref{zzdef3.3}(1). We claim that
\begin{enumerate}
\item[($*$)]
Suppose $u = \sum_{l=1}^m \lambda_l x^{a_l} y^{b_l} \in A$ for some $\lambda_l \in \kx$ and some distinct pairs $(a_l,b_l) \in \Znn^2$.  Let $i \in \ZZ$.  Then $u \in F_i$ if and only if $a_lz_1+b_lz_2 \ge i$ for all $l \in [1,m]$.
\end{enumerate}
Sufficiency follows from the definition of $F_i$.  Now suppose that $u \in F_i$ but the conclusion fails.  Without loss of generality, $a_1z_1+b_1z_2 < i$.  Since $u \in F_i$, we must have $u = \sum_{s=1}^t \mu_s x^{c_s} y^{d_s}$ for some $\mu_s \in \kx$ and some distinct pairs $(c_s,d_s) \in \Znn^2$ such that $c_sz_1+d_sz_2 \ge i$.  Then $(a_1,b_1) \ne (c_s,d_s)$ for all $s$, and so the equation $\sum_{l=1}^m \lambda_l x^{a_l} y^{b_l} = \sum_{s=1}^t \mu_s x^{c_s} y^{d_s}$ contradicts the linear independence of the monomials $x^\bullet y^\bullet$.  Thus $(*)$ is proved.

It follows from $(*)$ that $1 \in F_0 \setminus F_1$ and $\bigcap_{i \in \ZZ} F_i = \{0\}$.  Thus $\FF$ is a filtration on $A$, and its degree function defines a discrete valuation $\nu$ on $A$, where $\nu(0) := +\infty$ and $\nu(u) := \max\{i \in \ZZ \mid u \in F_i\}$ for nonzero $u \in A$.  This extends uniquely to a discrete valuation $\nu$ on $\Bbbk(x,y)$ (e.g., \cite[Chapter VI, \S3, Proposition 4]{Bo}).

(1)(2) These are immediate from $(*)$.

(3) We need to show that $\nu(\{u,v\}) \ge \nu(u) + \nu(v) - w$ for all $u,v \in K\{\bff\}$.  The argument of \cite[Lemma 2.9(3)]{HTWZ1} (which could be applied directly in case $\bff \in A$) can be used for $u,v \in A$, as follows.

First consider monomials $u = x^ay^b$ and $v = x^cy^d$ with $a,b,c,d \in \Znn$.  The desired inequality holds trivially if $\{u,v\} = 0$.  Otherwise, it follows from Lemma \ref{zzlem9.1} that
\begin{align*}
\nu(\{u,v\}) &= \nu( x^{a+c-1} y^{b+d-1} \bff ) \ge (a+c-1)z_1 + (b+d-1)z_2 + z_1+z_2-w  \\
&= \nu(u) + \nu(v) - w.
\end{align*}
Next, consider nonzero $u,v \in A$.  Write $u = \sum_{i=1}^m \lambda_i x^{a_i} y^{b_i}$ and $v = \sum_{j=1}^n \mu_j x^{c_j} y^{d_j}$ for some $\lambda_i,\mu_j \in \kx$ and some families $\bigl( (a_i,b_i) \bigr)_{i=1}^m$ and $\bigl( (c_j,d_j) \bigr)_{j=1}^n$ of distinct elements of $\Znn^2$.  In view of (2), $\nu(u) \le a_iz_1+b_iz_2$ for all $i$ and $\nu(v) \le c_jz_1+d_jz_2$ for all $j$.  Since $\{u,v\} = \sum_{i,j} \lambda_i \mu_j \{ x^{a_i} y^{b_i}, x^{c_j} y^{d_j} \}$, we see that
\begin{align*}
\nu(\{u,v\}) &\ge \underset{i,j}\min\bigl( \nu( \{ x^{a_i} y^{b_i}, x^{c_j} y^{d_j} \} ) \bigr) \ge \underset{i,j}\min\bigl( \nu(x^{a_i} y^{b_i}) + \nu(x^{c_j} y^{d_j}) - w \bigr)  \\
&= \underset{i,j}\min\bigl( a_iz_1+b_iz_2 + c_jz_1+d_jz_2 - w \bigr)  \ge \nu(u) + \nu(v) - w.
\end{align*}

Given general nonzero $u,v \in K$, write $u = u_1/u_2$ and $v = v_1/v_2$ for some nonzero $u_i,v_j \in A$.  Then
$$
\{u,v\} = u_2^{-1}v_2^{-1} \{u_1,v_1\} - u_2^{-1}v_1v_2^{-2} \{u_1,v_2\} - u_1u_2^{-2}v_2^{-1} \{u_2,v_1\} + u_1u_2^{-2}v_1v_2^{-2} \{u_2,v_2\}.
$$
Since $\nu(\{u_i,v_j\}) \ge \nu(u_i) + \nu(v_j) - w$ for $i,j=1,2$, each of the terms
$$
u_2^{-1}v_2^{-1} \{u_1,v_1\},\ u_2^{-1}v_1v_2^{-2} \{u_1,v_2\},\ u_1u_2^{-2}v_2^{-1} \{u_2,v_1\},\ u_1u_2^{-2}v_1v_2^{-2} \{u_2,v_2\}
$$
has $\nu$-value $\ge \nu(u)+\nu(v)-w$.  Therefore $\nu(\{u,v\}) \ge \nu(u)+\nu(v)-w$, as required.
\end{proof}

Now we introduce some $\gamma$-invariants 
\cite[Definition 4.1]{HTWZ1}.

\begin{definition}
\label{zzdef3.5}
Let $K$ be a Poisson field. Let $w,v$ be two integers.
\begin{enumerate}
\item[(1)]
The {\it $^{w}\Gamma_{v}$-cap} of $K$ is defined to be
\begin{equation}
\notag
{^{w}\Gamma_{v}}(K) :=\bigcap_{\nu \in {\mathcal V}_w(K)} F^{\nu}_v(K) = \{a\in K\mid \nu(a)\geq v, \
\forall \; \nu\in {\mathcal V}_{w}(K)\}
\end{equation}
if $\calV_w(K)$ is nonempty, while ${^{w}\Gamma_{v}}(K) := K$ if $\calV_w(K) = \emptyset$.  In particular, 
$$
{^{w}\Gamma_{v}}(K) = \{a\in K\mid \nu(a)\geq v, \
\forall \; w\text{-valuations}\ \nu\ \text{on}\ K \} \qquad \text{if}\ v \le 0.
$$

Note that ${^{w}\Gamma_{0}}(K)$ is a $\Bbbk$-subalgebra of $K$ 
containing $\Bbbk$, since all the $F^\nu_0(K)$ are 
$\Bbbk$-subalgebras containing $\Bbbk$.
\item[(2)]
The \emph{${^{\ast}\Gamma}$-subalgebras} of $K$ are defined 
to be the algebras in the family
$${{^\ast}\Gamma_{0}}(K):
=\bigl( {^{w}\Gamma_{0}}(K) \bigr)_{w\in {\ZZ}} \,.$$
\end{enumerate} 
\end{definition}

\begin{conjecture}  
\label{zzcon3.6}
We conjecture that if $K = K\{\bff\}$ and $\fht(K)\geq 4$, then 
${^1\Gamma}_0(K)\neq \Bbbk$.
\end{conjecture}

By Corollary \ref{zzcor1.2} and Lemmas \ref{zzlem4.3}, \ref{zzlem4.4}, 
${^1\Gamma}_0(K) = \Bbbk$ for 
$K = K\{\bff\}$ with $\fht(K) \le 2$. 
This is extended to $\fht(K) \le 3$ 
in \cite{GZ2}.  For other $K = 
K\{\bff\}$ in our four families, we 
have ${^1\Gamma}_0(K) = \Bbbk[x]$ 
or $\Bbbk[x,y]$ and $\fht(K) \ge 4$ 
(Lemmas \ref{zzlem4.5}, 
\ref{zzlem6.2}, \ref{zzlem7.2}, 
Proposition \ref{zzpro5.2}). 

The first two items of the following lemma are similar to \cite[Lemma 4.2]{HTWZ1}.
 
\begin{lemma}
\label{zzlem3.7}
Let $f: K\to Q$ be a Poisson algebra morphism between two 
Poisson fields. Let $w,v$ be integers.
\begin{enumerate}
\item[(1)]
Suppose either $\nu \circ f$ is not trivial for any
$\nu\in {\mathcal V}_w(Q)$ 
or $v \le 0$. Then $f$ maps ${^{w}\Gamma_{v}}(K)$ into ${^{w}\Gamma_{v}}(Q)$.
\item[(2)]
${^{w}\Gamma_{v}}(K)\supseteq {^{w+1}\Gamma_{v}}(K)$ for all 
$v$ and ${^{w}\Gamma_{v}}(K)\supseteq \Bbbk$ if $v\leq 0$.
\item[(3)]
${^{w}\Gamma_{0}}(K)$ is integrally closed in $K$.
\item[(4)]
Suppose $w\le -1$. Then ${^{w}\Gamma_{v}}(K)\supseteq 
{^{-1}\Gamma_{v}}(K)$ for all $v$ and 
${^{w}\Gamma_{0}}(K)={^{-1}\Gamma_{0}}(K)$. 
\item[(5)]
Suppose ${\mathcal V}_w(K)$ is non-empty. Then 
$\{{^{w}\Gamma_{i}}(K)\mid i\in {\ZZ}\}$ is a filtration 
of $K_{w}:=\bigcup_{i\in {\ZZ}} {^{w}\Gamma_{i}}(K)$.
\end{enumerate}
\end{lemma}

\begin{proof}
(1) 
Let $\nu$ be a $w$-valuation of $Q$. Define 
$\nu f := \nu \circ f$, which is a
$w$-valuation of $K$ (possibly trivial, even if $\nu$ is not). 
We consider two cases.

Case 1: Suppose $\nu f$ is not trivial for any
$\nu\in {\mathcal V}_w(Q)$.
Then $f$ maps 
$${^{w}\Gamma_{v}}(K)=\bigcap_{\mu\in {\mathcal V}_{w}(K)}
F^{\mu}_v(K) \subseteq \bigcap_{\nu\in {\mathcal V}_{w}(Q)}
F^{\nu f}_v(K) \longrightarrow \bigcap_{\nu\in {\mathcal V}_{w}(Q)}
F^{\nu }_v(Q)={^{w}\Gamma_{v}}(Q),$$ 
whence the assertion follows.

Case 2: Suppose $v \le 0$. In this 
case, $F^{\nu f}_{v}(K)= K$ whenever 
$\nu f$ is trivial, and so $f$ maps 
$$
\begin{aligned}
{^{w}\Gamma_{v}}(K)
&= \{a\in K\mid \nu(a)\geq v, \
\forall \; w\text{-valuations}\ 
\nu\ \text{on}\ K \}  \\
&\subseteq \bigcap_{\nu\in 
{\mathcal V}_{w}(Q)}
F^{\nu f}_v(K) \longrightarrow 
\bigcap_{\nu\in {\mathcal V}_{w}(Q)} 
F^{\nu }_v(Q) ={^{w}\Gamma_{v}}(Q).
\end{aligned}
$$ 
The assertion follows.

(2) By Definition \ref{zzdef3.1}(4), $\Bbbk\subseteq F^{\nu}_0(K)$
for every $w$-valuation $\nu$ on $K$. 
So 
$\Bbbk\subseteq {^{w}\Gamma_{0}}(K)
\subseteq {^{w}\Gamma_{v}}(K)$ when 
$v\leq 0$. Note that ${\mathcal V}_w(K)
\subseteq {\mathcal V}_{w+1}(K)$ by 
Definition \ref{zzdef3.1}(5),
so ${^{w}\Gamma_{v}}(K)\supseteq 
{^{w+1}\Gamma_{v}}(K)$ by definition.

(3) This follows from 
\cite[Lemma 2.4(3)]{HTWZ1}, which 
shows that $F^\nu_0(K)$ is integrally 
closed in $K$ for all 
$\nu \in \calV_w(K)$.

(4) By part (2), when $w\leq -1$, 
${^{w}\Gamma_{v}}(K)\supseteq 
{^{-1}\Gamma_{v}}(K)$ for all $v$. 
Now for each $\nu\in {\mathcal V}_{-1}(K)$, 
we have $|w| \nu\in {\mathcal V}_{w}(K)$. 
Then 
$${^{-1}\Gamma_{0}}(K)
=\bigcap_{\nu\in {\mathcal V}_{-1}(K)}
F^{\nu}_0(K)
=\bigcap_{\nu\in {\mathcal V}_{-1}(K)}
F^{|w|\nu}_0(K) \supseteq 
\bigcap_{\nu\in {\mathcal V}_{w}(K)}
F^{\nu}_0(K)={^{w}\Gamma_{0}}(K).$$
The assertion follows.

(5) The assertion follows from an easy fact: 
${^{w}\Gamma_{i}}(K){^{w}\Gamma_{j}}(K)
\subseteq {^{w}\Gamma_{i+j}}(K)$ 
for all $i,j$. 
\end{proof}

Next we state an important result from \cite{HTWZ1}.

\begin{theorem}[${^{1}\Gamma_{0}}$-Controlling theorem]
\label{zzthm3.8}
Let $A$ be a Poisson noetherian normal domain of 
GK-dimension at least 2 and $K$ be the Poisson 
fraction field of $A$. Then ${^{1}\Gamma_{0}}(K)\subseteq A$.
\end{theorem}

\begin{proof}
See \cite[Theorem 4.3(2)]{HTWZ1}.
\end{proof}

In this paper we often use the following 
immediate consequence.

\begin{corollary}  
\label{zzcor3.9} 
If $K = K\{\bff\}$ with $\bff \in 
\Bbbk[x,y]$, then ${}^1\Gamma_0(K) 
\subseteq \Bbbk[x,y]$.
\qed\end{corollary}

\begin{lemma} 
\label{zzlem3.10}
Let $K$ be a field containing $\Bbbk$, and let $\nu$ be a discrete valuation on $K$ such that $\nu = 0$ on $\kx$. 
Let $h\in K$ and $f(t)\in \Bbbk[t^{\pm 1}]$ be a nonzero
Laurent polynomial.  Write $f(t) = a_l t^l + a_{l+1} t^{l+1} + \cdots + a_n t^n$ with 
$l \le n$, coefficients $a_i \in \Bbbk$, and $a_l,a_n \ne 0$.
\begin{enumerate}
\item[(1)]
\begin{itemize}
\item 
If $\nu(h) > 0$, then $\nu(f(h)) = l \nu(h)$.
\item 
If $\nu(h) = 0$, then $\nu(f(h)) \ge 0$.
\item 
If $\nu(h) < 0$, then $\nu(f(h)) = n \nu(h)$.
\end{itemize}
\item[(2)]
$\nu(f(h)) \ge 0$ if and only if either $\nu(h) = 0$, or $l\geq 0$ and 
$\nu(h) > 0$, or $n\leq 0$ and $\nu(h) < 0$.
\item[(3)]
$\nu(f(h)) < 0$ if and only if either $l< 0$ and $\nu(h)>0$,  or $n>0$ and $\nu(h)<0$.  In the first case,  $0< -l \leq -\nu(f(h))$, while in the second, $0<n\le -\nu(f(h))$. 
\end{enumerate}
\end{lemma}

\begin{proof}
(1) Let $j_1 = l < j_2 < \cdots < j_s = n$ be those 
$i \in [l,n]$ such that $a_i \ne 0$, and note that 
$\nu(a_{j_m} h^{j_m}) = j_m\nu(h)$ for $m \in [1,s]$.

If $\nu(h) > 0$, then $\nu(a_{j_1} h^{j_1}) < 
\nu(a_{j_2} h^{j_2}) < \cdots < \nu(a_{j_s} h^{j_s})$, and so 
$$\nu(f(h)) = \min
\bigl( \nu(a_{j_m} h^{j_m}) \mid m \in [1,s] \bigr) 
= \nu(a_{j_1} h^{j_1}) = l \nu(h).$$
The case when $\nu(h) < 0$ is proved in the same manner. In 
case $\nu(h) = 0$, we have $\nu(a_{j_m} h^{j_m}) = 0$ for all 
$m$, and all we can conclude is that $\nu(f(h)) \ge 0$.

(2) and (3) follow immediately from part (1).
\end{proof}

\begin{corollary}  
\label{zzcor3.11}
Let $K$ be a Poisson field with $h\in K\setminus \Bbbk$ 
and let $f(t)$ be a Laurent polynomial
$a_{l}t^l+a_{l+1} t^{l+1}+\cdots+ a_{n} t^n$ with 
$l\leq n$, $a_i\in \Bbbk$, and $a_{l}, a_{n}\neq 0$.
Suppose $u$, $v$ are nonzero elements in $K$ such that
$\{u,v\}=uv f(h)$. Suppose $0\leq d<n $. Then 
$h\in {^{d}\Gamma_{0}}(K)$.
\end{corollary}

\begin{proof} Let $\nu$ be any nontrivial $d$-valuation of $K$. Then 
$$\nu(f(h))=\nu(\{u,v\} u^{-1} v^{-1})
\geq \nu(u)+\nu(v)-d -\nu(u)-\nu(v) =-d$$ 
by the valuation axioms. Suppose that $\nu(h)<0$. By Lemma 
\ref{zzlem3.10}(1), $\nu(f(h))=n\nu(h)\leq -n$, and 
consequently, $-n\geq -d$. This yields a contradiction as 
$d<n$. Therefore $\nu(h)\geq 0$ as required. 
\end{proof}

\begin{corollary}  
\label{zzcor3.12}
Let $K = K\{ f(x) xy\}$ for some $f \in \Bbbk[x]$ with 
$\deg f \ge 2$. Then ${^{1}\Gamma_{0}}(K) = \Bbbk[x]$.
\end{corollary}

\begin{proof}
By Corollary \ref{zzcor3.11}, $x\in {^{1}\Gamma_{0}}(K)$. 
Since ${^{1}\Gamma_{0}}(K)$ is a $\Bbbk$-subalgebra,
$\Bbbk[x]\subseteq {^{1}\Gamma_{0}}(K)$. 

By Corollary \ref{zzcor3.9}, ${^{1}\Gamma_{0}}(K)
\subseteq \Bbbk[x,y]$. Since $K = \Bbbk(x,y^{-1})$ and 
$$\{x,y^{-1}\}= - y^{-2}\{x,y\}=
-y^{-2}f(x)xy= - f(x) x y^{-1}\,,$$
Corollary \ref{zzcor3.9} also implies that 
${^{1}\Gamma_{0}}(K) \subseteq \Bbbk[x,y^{-1}]$. Combining 
these two facts, we obtain that ${^{1}\Gamma_{0}}(K)
\subseteq \Bbbk[x]$.
\end{proof}

Next we introduce another notion 
of {\it height} that has the 
(co)hereditary property and agrees 
with the flag height for several 
families of Poisson fields $K\{\bff\}$.
Let $K$ be a Poisson algebra/field 
and $w\in {\ZZ}$. Let 
${\mathcal V}_{nc, w}(K)$ denote 
the set of non-classical $w$-valuations 
of $K$. (When $w>0$, 
${\mathcal V}_{nc, w}(K)$ contains 
the trivial valuation.)  For 
$n \in \ZZ$, define
$$
{\mathcal V}_{nc, n,w}(K)
:=\{ \nu \in {\mathcal V}_{nc, w}(K) \;\mid \;
F^{\nu}_0(K)\cap {^n\Gamma}_0(K) = \Bbbk\}.
$$

\begin{definition}
\label{zzdef3.13}
Let $K$ be any Poisson field with 
nontrivial Poisson bracket, and 
let $n$ be an integer.
The {\it $n$-valuation height}
of $K$ is defined to be
\begin{equation}
\notag
\vht_n(K):=2+ \inf\{ w \in \ZZ \; \mid \; {\mathcal V}_{nc, n,w}(K)
\neq \emptyset\}.
\end{equation}
 Note that if ${\mathcal V}_{nc,n,w}(K)= \emptyset$ for every $w \in \ZZ$, 
then $\vht_n(K)=+\infty$. 
\end{definition}

Valuation height is closely related to 
the flag height when $K$ is of the 
form $K\{\bff\}$ as indicated below.

\begin{lemma}
\label{zzlem3.14}
Let $K$ and $K'$ be two Poisson fields, 
and let $n$, $n'$ be integers.
\begin{enumerate}
\item[(1)]
$\vht_n(K)$ is either $-\infty$ or $\geq 2$.
\item[(2)]
If $K\subseteq K'$ and the Poisson bracket on $K$ is nontrivial, then $\vht_n(K)\leq \vht_n(K')$.
\item[(3)]
If $n\leq n'$, then $\vht_n(K)\geq \vht_{n'}(K)$.
\item[(4)]
If $K$ is of the form $K\{\bff\}$, then 
$\vht_n(K)\leq \vht_1(K)\leq \fht(K)$
for every $n\geq 1$.
\end{enumerate}
\end{lemma}

\begin{proof}
(1) Suppose that $\vht_n(K)<2$. Then there is an integer $w<0$ 
such that ${\mathcal V}_{nc,n,w}(K)\neq \emptyset$, namely,
there is some $\nu\in {\mathcal V}_{nc,w}(K)$ such that 
$F^{\nu}_0(K) \cap {^n\Gamma}_{0}(K) = \Bbbk$. For every 
positive integer $d$, we have $d\nu \in {\mathcal V}_{nc,dw}(K)$.
Note that $F^{d\nu}_0(K)=F^{\nu}_0(K)$ by definition. Then
$F^{d\nu}_0(K) \cap {^n\Gamma}_{0}(K) = \Bbbk$. This means
that ${\mathcal V}_{nc,n,dw}(K)\neq \emptyset$, and so $\vht_n(K) \le 2+dw$.  Consequently, $\vht_n(K)=-\infty$. 

(2) If $\vht_n(K')=+\infty$, nothing needs to be proved. Now assume 
that $\vht_n(K') < +\infty$, and let $w \in \ZZ$ such that $\calV_{nc,n,w}(K') \ne \emptyset$. Then there is a non-classical 
$w$-valuation $\nu$ of $K'$ such that 
$F^{\nu}_0(K') \cap {^n\Gamma}_{0}(K') = \Bbbk$. By restriction,
$\mu:=\nu|_{K}$ is a $w$-valuation 
of $K$ (possibly trivial), and there is some $l \in \Znn$ such that $\mu$ is a non-classical $(w-l)$-valuation. By Lemma \ref{zzlem3.7}(1), ${^n\Gamma}_0(K)
\subseteq {^n\Gamma}_0(K')$, and by definition, $F^{\mu}_0(K)
\subseteq F^{\nu}_0(K')$. Consequently, $F^{\mu}_0(K)\cap 
{^n\Gamma}_{0}(K) = \Bbbk$. Therefore ${\mathcal V}_{nc,n,w-l}(K) \neq \emptyset$ and $\vht_n(K) \le 2+w-l \leq 2+w$. 

(3) follows from the definition and Lemma \ref{zzlem3.7}(2). 

(4) If $\fht(K)=+\infty$, nothing needs to be proved.
Now assume that $d:=\fht(K)$ is finite.  We may thus assume that $\deg \bff = d$.  Let $\nu$ be the discrete valuation on $K$ from Lemma \ref{zzlem3.4} with $\nu(x) = \nu(y) = -1$.  We claim that $\nu$ is a non-classical $(d-2)$-valuation.
This follows from the fact that
$$\nu(\{x,y\})=\nu(\bff)=-d= -2-(d-2) =\nu(x)+\nu(y)-(d-2).$$
Moreover, Lemma \ref{zzlem3.4} shows 
that $\nu < 0$ on the set $\Bbbk[x,y] 
\setminus \Bbbk$, which implies that $F^{\nu}_0(K)\cap {^1\Gamma}_{0}(K)
\subseteq F^{\nu}_0(K) \cap 
\Bbbk[x,y]=\Bbbk$. Therefore 
$\nu\in {\mathcal V}_{nc, 1,d-2}(K)$. 
Hence
$$d-2\geq \inf\{ w \; \mid \; 
{\mathcal V}_{nc, 1,w}(K)
\neq \emptyset\}.$$
By definition, $d=2+(d-2) \geq \vht_1(K)$.
\end{proof}

\begin{lemma}
\label{zzlem3.15}
Let $K$ be a Poisson field of the form $K\{\bff\}$, and suppose that $\vht_n(K)=\fht(K)$ for some positive integer $n$. Then $K$ is height cohereditary. 
\end{lemma}

\begin{proof}
Let $K' := K\{\bfg\}$ and $\phi: K\to K'$ an embedding.
By Lemma \ref{zzlem3.14}(2)(4),
we have
$$\fht(K)=\vht_n(K)\leq \vht_n(K') \leq \vht_1(K')\leq \fht(K').$$
Therefore $K$ is height cohereditary.
\end{proof}

Next we give a lower bound for valuation height.

\begin{lemma}  
\label{zzlem3.16}
Let $K$ be a Poisson field with $h\in K\setminus \Bbbk$ 
and let $f(t)$ be a Laurent polynomial
$a_{l}t^l+a_{l+1} t^{l+1}+\cdots+ a_{n} t^n$ with 
$l\leq n$, $a_i\in \Bbbk$, and $a_{l}, a_{n}\neq 0$.
Assume $n\geq 2$.
Suppose $u,v$ are nonzero elements in $K$ such that
$\{u,v\}=uv f(h)$. Then $\vht_{n-1}(K)\geq 2+n$.
\end{lemma}

\begin{proof} By Corollary \ref{zzcor3.11},
$h\in {^{n-1}\Gamma}_0(K)$. If $w$ is an integer and 
$\nu\in {\mathcal V}_{nc,n-1,w}(K)$, then 
$\nu(h)<0$ because $h \notin \Bbbk$. Then, using Lemma \ref{zzlem3.10}(1),
$$\nu(u)+\nu(v)-w \leq \nu(uv f(h))=
\nu(u)+\nu(v)+\nu(f(h))=
\nu(u)+\nu(v)+n \nu(h)$$
which implies that $w\geq n(-\nu(h))\geq n$. Therefore
$\vht_{n-1}(K)\geq 2+n$.
\end{proof}

In this paper, we will consider elements with the 
following property.

\begin{definition}
\label{zzdef3.17}
Let $d$ be a positive integer. 
A polynomial $\bfg\in \Bbbk[x,y]$ is called {\it $d$-flabby} if 
${^d\Gamma}_{0}(K\{\bfg\})=\Bbbk[x,y]$.  If $\bfg$ is $d$-flabby for some unspecified $d>0$, we just say that $\bfg$ is \emph{flabby}.
\end{definition}

\begin{proposition}
\label{zzpro3.18}
Let $\bff=\sum_{i=0}^m \sum_{j=0}^n c_{ij} x^i y^j$ with $c_{ij} \in \Bbbk$ and
$c_{mn}\neq 0$, where $m\geq 1$, $n\geq 1$. Let $K$ be 
$K\{\bff\}$. If $\bff$ is $d$-flabby for some $d\geq 1$, then 
$\vht_d(K)=\fht(K)=m+n$. 
As a consequence, $K$ is height cohereditary. 
\end{proposition}

\begin{proof} By Lemma \ref{zzlem3.14}(4), $\vht_d(K)
\leq \fht(K)\leq m+n$. It remains to show 
that $\vht_d(K)\geq m+n$. This is automatic if $\vht_d(K) = +\infty$, so assume $\vht_d(K) < +\infty$.

Let $w$ be an integer such that ${\mathcal V}_{nc,d,w}(K)
\neq \emptyset$. So there is a valuation  
$\nu\in {\mathcal V}_{nc,w}(K)$ such that 
$F^{\nu}_0(K)\cap {^d\Gamma}_{0}(K)=\Bbbk$.  
Then $\nu(x)<0$ and $\nu(y)<0$, which implies
that $\nu(\bff)=m \nu(x)+n \nu(y)$. Now the
equation $\{x,y\}=\bff$ implies that
$$\nu(x)+\nu(y)-w \leq \nu(\{x,y\})=
\nu(\bff)=m \nu(x)+n \nu(y).$$
An easy computation shows that
$$2+ w\geq 2+ (m-1)(-\nu(x))+(n-1) (-\nu(y))
\geq 2+ m-1+n-1=m+n,$$
which implies that $\vht_d(K)\geq m+n$.

Therefore $\vht_d(K)=\fht(K)=m+n$.  By Lemma \ref{zzlem3.15}, $K$ is height cohereditary.
\end{proof}

%%%%%%%%%%%%%%%%%%%%%%%%%

We now begin work on the families of Poisson 
fields $K\{\bff\}$ described in the introduction, and study the Automorphism, Isomorphism, Embedding, and Dixmier Problems. One key tool is the $\gamma$-invariants 
that were introduced in the last section.

Recall from Proposition \ref{zzpro1.1}(1) that if $\bff \in \Bbbk(x)^{\times}$ or $\bff \in \Bbbk(y)^\times$, then $K\{\bff\}\cong \Kweyl$. So everything is 
reduced to the Poisson Weyl field in those cases. The first larger family 
we will consider is the following.  

%%%%%%%%%%%%%%%%%%%%%%%%%

\section{Family (1): $K\{ q x^a y^b\}$}
\label{zzsec4}

In this section, we consider the Poisson fields
$K\{\bff\}$ where $\bff$ is of the form $q x^a y^b$ with $q \in \kx$ and $a,b \in \ZZ$. It is convenient to 
rewrite $\bff$ as $q x^{1+\kappa_1} y^{1+\kappa_2}$ 
with $q\in \Bbbk^{\times}$ and $\kappa_1,\kappa_2\in {\ZZ}$, and to relabel $K\{\bff\}$ as $K_{q,\kappa_1,\kappa_2}$ or $K_{q,\kappa}$ where $\kappa := (\kappa_1,\kappa_2)$.  Thus, $K_q = K_{q,0,0}$ and $\Kweyl = K_{1,-1,-1}$.

Note that $K_{q,\kappa} = K_{q,\kappa_1,\kappa_2}$ is the quotient field of the 
Poisson torus $T_{q,\kappa} = T_{q,\kappa_1,\kappa_2}$, which is the Laurent polynomial ring
$\Bbbk[x^{\pm 1}, y^{\pm 1}]$ with 
Poisson structure determined by $\{x,y\} = q x^{1+\kappa_1} y^{1+\kappa_2}$.

\begin{lemma}
\label{zzlem4.1}
Retain the above notation, and let $q, q' \in \kx$ and $\kappa, \kappa' \in \ZZ^2 \setminus \{(0,0)\}$.  For parts {\rm(3)--(5)}, assume that $\Bbbk$ is algebraically closed. 
\begin{enumerate}
\item[(1)]
If there exist $\alpha,\beta \in \kx$ such that $\alpha^{\kappa_1} \beta^{\kappa_2} = q$, then
$T_{q,\kappa}\cong T_{1,\kappa}$ and $K_{q,\kappa}\cong K_{1,\kappa}$.
\item[(2)]
$T_{q,\kappa} \cong T_{q,\kappa_0,0}$ and $K_{q,\kappa} \cong K_{q,\kappa_0,0}$ where $\kappa_0:=\gcd(\kappa)$.
In particular, $T_{q,\kappa_0,0}\cong 
T_{q,\kappa_0,d\kappa_0}$ and $K_{q,\kappa_0,0}\cong 
K_{q,\kappa_0,d\kappa_0}$
for all $d\in {\ZZ}$.
\item[(3)] 
Let $\kappa_0$ and $\kappa'_0$ be nonzero integers.  Then
$T_{q,\kappa_0,0} \cong T_{q',\kappa'_0,0}$
if and only if $\kappa_0=\pm \kappa'_0$.  In particular, if $\kappa_0=\pm \kappa'_0$ then $K_{q,\kappa_0,0} \cong K_{q',\kappa'_0,0}$.
\item[(4)]
$T_{q,\kappa} \cong T_{q',\kappa'}$ if and only if 
$\gcd(\kappa')=\gcd(\kappa)$.  In particular, if $\gcd(\kappa')=\gcd(\kappa)$ then $K_{q,\kappa} \cong K_{q',\kappa'}$.
\item[(5)]
There are Poisson algebra embeddings $T_{1,\kappa}\to 
T_{1,d\kappa}$ and $K_{1,\kappa}\to 
K_{1,d\kappa}$ for all nonzero integers $d$.
\item[(6)]
$T_{q,\kappa}$ is Poisson-simple, meaning that it has no proper nonzero Poisson ideals.
\end{enumerate}
\end{lemma}

\begin{proof}
Since the statements about Poisson fields follow immediately from those about Poisson tori, we just need to prove the latter.

(1) This is clear since $\Bbbk[x^{\pm1}, y^{\pm1}] = \Bbbk[(\alpha x)^{\pm1}, (\beta y)^{\pm1}]$ and
$$
(\alpha x)^{1+\kappa_1} (\beta y)^{1+\kappa_2} = \alpha \beta q x^{1+\kappa_1} y^{1+\kappa_2} = \{ \alpha x, \beta y \}.
$$

(2) Let $(a,b) := \kappa/\kappa_0$, so that $\gcd(a,b) = 1$ and there are $c,d \in \ZZ$ such that $ad-bc = 1$.  Setting $\xhat := x^a y^b$ and $\yhat := x^c y^d$, we see that $\xhat^d \yhat^{-b} = x$ and $\xhat^{-c} \yhat^a = y$, so that $\Bbbk[x^{\pm1}, y^{\pm1}] = \Bbbk[\xhat^{\pm1}, \yhat^{\pm1}]$.  Moreover, using Lemma \ref{zzlem9.1} we have 
$$
\{\xhat,\yhat\} = x^{a+c-1} y^{b+d-1} q x^{1+a\kappa_0} y^{1+b\kappa_0} = q x^{a(1+\kappa_0)+c} y^{b(1+\kappa_0)+d} = q \xhat^{1+\kappa_0} \yhat.
$$
This establishes the first isomorphism, and then the second follows.

(3) If $\kappa_0 = - \kappa'_0$, we obtain $T_{1,\kappa_0,0} \cong T_{-1,\kappa'_0,0}$ from the observation that 
$$
\{x^{-1},y\} = - x^{-2} x^{1+\kappa_0} y = - x^{-1+\kappa_0} y = - (x^{-1})^{1+\kappa'_0} y
$$
in $T_{1,\kappa_0,0}$.  Consequently, $T_{q,\kappa_0,0} \cong T_{q',\kappa'_0,0}$ via part (1).

Conversely, if $T_{q,\kappa_0,0} \cong T_{q',\kappa'_0,0}$, then $T_{1,\kappa_0,0} \cong T_{1,\kappa'_0,0}$ by (1), and so there exist $u,v \in T_{1,\kappa_0,0}$ such that $\Bbbk[x^{\pm1}, y^{\pm1}] = \Bbbk[u^{\pm1}, v^{\pm1}]$ and $\{u,v\} = u^{1+\kappa'_0} v$.  We must have $u = \lambda x^a y^b$ and $v = \mu x^c y^d$ for some $\lambda,\mu \in \kx$ and $\left( \begin{smallmatrix} a&b\\ c&d \end{smallmatrix} \right) \in GL_2(\ZZ)$.  Then
$$
\lambda^{1+\kappa'_0} \mu (x^a y^b)^{1+\kappa'_0} (x^c y^d) = \{u,v\} = \lambda \mu (ad-bc) x^{a+c-1} y^{b+d-1} x^{1+\kappa_0} y,
$$
whence $a(1+\kappa'_0)+c = a+c+\kappa_0$ and $b(1+\kappa'_0)+d = b+d$.  But then $b \kappa'_0 = 0$ and so $b=0$.  Thus $a = \pm1$ and $\kappa_0 = a\kappa'_0 = \pm\kappa'_0$.

(4) follows from (2) and (3).

(5) In view of (2), we may assume that $\kappa = (\kappa_0,0)$.  View $T_{1,d\kappa}$ as $\Bbbk[x^{\pm1}, y^{\pm1}]$ with $\{x,y\} = x^{1+d\kappa_0} y$.  Set $u := x^d$ and observe that
$$
\{u,y\} = d x^{d-1} x^{1+d\kappa_0} y = d x^{d(1+\kappa_0)} y = d u^{1+\kappa_0} y \,.
$$
Thus $\Bbbk[u^{\pm1},y^{\pm1}] \cong T_{d,\kappa} \cong T_{1,\kappa}$.

(6) Suppose there is a proper nonzero Poisson ideal $I$ in $T_{q,\kappa}$.  Any prime ideal $P$ minimal over $I$ is a Poisson ideal (e.g., \cite[Lemma 1.1(c)]{Go1}), so the quotient field of $T_{q,\kappa}/P$ is a Poisson field.  Since $P$ is nonzero, the cosets of $x$ and $y$ in  $T_{q,\kappa}/P$ cannot be algebraically independent over $\Bbbk$, and hence the Poisson bracket of these cosets is zero (Observation \ref{zzobs1.4}).  But this means $\{x,y\} \in P$, which is impossible because $\{x,y\}$ is invertible in $T_{q,\kappa}$.  Therefore $T_{q,\kappa}$ is Poisson-simple.
\end{proof}

\begin{corollary}  
\label{zzcor4.2}  
Assume that $\Bbbk$ is algebraically closed, $q \in \kx$, and $\kappa \in \ZZ^2$.  
\begin{enumerate}
\item[(1)]
 $K_{q,\kappa}$ 
is isomorphic to either $K_{q}$, or $\Kweyl$, or $K_{1,n,0}$
for some $n\geq 2$. 
\item[(2)]
$K_{q,\kappa}$ has a Poisson subfield isomorphic to $K_q$ or $\Kweyl$.
\end{enumerate}
\end{corollary}

\begin{proof}
(1) follows directly from Lemma \ref{zzlem4.1}.

(2) By part (1), it is enough to deal with $K_{1,n,0}$.  This has a Poisson subfield isomorphic to $\Kweyl$ by Corollary \ref{zzcor1.7}(2), as noted in Remarks \ref{zzrem1.8}.
\end{proof}

Next we consider the $\gamma$-invariants for 
the $K_{q,\kappa}$.
\medskip

Case 1: $K=K_q=K_{q,0,0}$ for some $q\in \Bbbk^{\times}$. By
\cite[Theorem 6.2(3)]{HTWZ1}, there is no $w$-valuation on $K$ for any 
$w<0$. 

\begin{lemma}
\label{zzlem4.3}
Let $K=K_q$ for some $q\in \Bbbk^{\times}$. 
\begin{enumerate}
\item[(1)]
${^w\Gamma}_{0}(K)=\Bbbk$ for all $w>0$ and 
${^w\Gamma}_{v}(K)=0$ for all $w,v>0$. As a consequence,
$qxy$ is not $1$-flabby.
\item[(2)]
$\vht_1(K)=\fht(K)=2$.
\end{enumerate}
\end{lemma}

\begin{proof} (1) 
Note that $A=\Bbbk[x,y]$ is a Poisson subalgebra of $K$
such that $Q(A)=K$. By Corollary \ref{zzcor3.9}, 
${^{1}\Gamma_{0}}(K)\subseteq A$. 
Note that 
$B:=\Bbbk[x^{-1},y^{-1}]$ is also a Poisson subalgebra of $K$ 
(with $\{x^{-1},y^{-1}\}=qx^{-1}y^{-1}$) such that $Q(B)=K$. 
By Theorem \ref{zzthm3.8}, ${^{1}\Gamma_{0}}(K)\subseteq B$. 
Then ${^{1}\Gamma_{0}}(K)\subseteq A\cap B=\Bbbk$. Since ${^{w}\Gamma_{0}}(K)\subseteq 
{^{1}\Gamma_{0}}(K)$ and 
${^{w}\Gamma_{0}}(K)\supseteq \Bbbk$ [Lemma \ref{zzlem3.7}(2)],
${^{w}\Gamma_{0}}(K)=\Bbbk$. 

For $w,v>0$, since each ${^{w}\Gamma_{v}}(K)$ is a proper ideal of 
${^{w}\Gamma_{0}}(K)$, the second assertion follows.

(2) We have $\vht_1(K) \le \fht(K)$ by Lemma \ref{zzlem3.14}(4), and $\fht(K) \le 2$ because $K = K\{qxy\}$.  On the other hand, $\vht_1(K) \ge 2$ since there is no $w$-valuation on $K$ for any $w<0$.
\end{proof}

Case 2: $K=\Kweyl$. By \cite[Example 2.10]{HTWZ1}, 
for each integer $w$, there is a nontrivial $w$-valuation 
on $K$. This implies that ${^{w}\Gamma_{v}}(K)\neq K$
for every pair $(w,v)$. 

\begin{lemma}
\label{zzlem4.4}
Let $K=\Kweyl$.
\begin{enumerate}
\item[(1)]
$(xy)^{-1}\in {^{w}\Gamma_{-w}}(K)$ for every $w \in \ZZ$. 
\item[(2)]
${^{w}\Gamma_{0}}(K)=\Bbbk$ for all $w>0$ and 
${^{w}\Gamma_{v}}(K)=0$ for all $w,v>0$. As a 
consequence, $1$ is not $1$-flabby.
\item[(3)]
$\vht_1(K)= - \infty$ and $\fht(K)=0$.
\end{enumerate}
\end{lemma}

\begin{proof}
(1) Let $\nu$ be a $w$-valuation on $K$. Then 
$0=\nu(1)=\nu(\{x,y\})\geq \nu(x)+\nu(y)-w$. Hence
$\nu((xy)^{-1})=-(\nu(x)+\nu(y))\geq -w$. So 
$(xy)^{-1}\in {^{w}\Gamma_{-w}}(K)$. 

(2) Note that $A=\Bbbk[x,y]$ is a Poisson subalgebra of $K$
such that $Q(A)=K$, and that 
$B:=\Bbbk[x^{-1},y^{-1}]$ is a Poisson subalgebra of $K$ 
(with $\{x^{-1},y^{-1}\}=(x^{-1}y^{-1})^2$) such that 
$Q(B)=K$. By Corollary \ref{zzcor3.9},  ${^{1}\Gamma_{0}}(K)\subseteq A\cap B=\Bbbk$. The first and second assertions now follow as in Lemma \ref{zzlem4.3}.

(3) Clearly $\fht(K) = 0$ since $K = K\{1\}$.  
For any integer $w \le 0$, there exist 
nontrivial non-classical $w$-valuations on 
$K$ \cite[Example 2.10]{HTWZ1}, and so
$\calV_{nc,1,w}(K) \ne \emptyset$ due 
to ${^{1}\Gamma_{0}}(K) =\Bbbk$.  Thus 
$\vht_1(K)= - \infty$. Of course 
$\fht(K) = 0$ because $\{x,y\}=1$ has 
degree $0$.
\end{proof}

Case 3: $K = K_{1,n,0}$ for some $n\geq 2$. Thus $K = \Bbbk(x,y)$ with $\{x,y\} = x^{n+1} y$.

\begin{lemma}
\label{zzlem4.5}
Let $n\geq 2$ and $K =K_{1,n,0}$. 
\begin{enumerate}
\item[(1)]
${^{w}\Gamma_{0}}(K)=\Bbbk[x]$ for $1\leq w\leq n-1$, and
${^{w}\Gamma_{0}}(K)=\Bbbk$ for all $w\geq n$. As a consequence, 
$x^{n+1}y$ is not $1$-flabby for $n\geq 2$.
\item[(2)]
${^{w}\Gamma_{v}}(K)=0$ for all $w,v>0$.
\item[(3)]
$\vht_1(K)=\fht(K)=n+2$.
\end{enumerate}
\end{lemma}

\begin{proof}
(1) By definition, $A=\Bbbk[x,y]$ is a Poisson subalgebra 
of $K$ such that $Q(A)=K$. By 
Corollary \ref{zzcor3.9}, ${^{1}\Gamma_{0}}(K)\subseteq A$.  Note that 
$B:=\Bbbk[x,y^{-1}]$ is also
a Poisson subalgebra of $K$ (with 
$\{x,y^{-1}\}=-x^{n+1}(y^{-1})$) 
such that $Q(B)=K$. By Theorem
\ref{zzthm3.8}, ${^{1}\Gamma_{0}}(K)\subseteq B$. Then 
${^{1}\Gamma_{0}}(K)\subseteq A\cap B=\Bbbk[x]$. 
As a consequence, ${^{w}\Gamma_{0}}(K)\subseteq \Bbbk[x]$
for all $w\geq 1$. By Corollary \ref{zzcor3.11},
$x\in {^{w}\Gamma_{0}}(K)$ for all $0 \le w<n$. So, 
if $1\leq w<n$, then ${^{w}\Gamma_{0}}(K)=\Bbbk[x]$. 

If $w\geq n$, there is a $w$-valuation $\nu$ such that
$\nu(x)=-1$ and $\nu(y)=0$ (Lemma \ref{zzlem3.4}). So $x\not\in 
{^{w}\Gamma_{0}}(K)$. In fact, one sees that
any polynomial $f(x)$ of positive degree is not
in ${^{w}\Gamma_{0}}(K)$. So ${^{w}\Gamma_{0}}(K)=\Bbbk$. 

(2) It suffices to show that ${^{1}\Gamma_{1}}(K)=0$. By part (1), 
${^{1}\Gamma_{1}}(K)$ is an ideal of $\Bbbk[x]$.

Let $\nu$ be a $1$-valuation such that $\nu(x)=0$ 
and $\nu(y)=1$ (Lemma \ref{zzlem3.4}). Then $\nu(f(x)) = 0$ and hence 
$f(x)\not\in F^{\nu}_1(K)$ for all nonzero $f(x) \in \Bbbk[x]$.
The assertion follows.

(3) We have $\vht_1(K) \le \fht(K)$ by Lemma \ref{zzlem3.14}(4), and clearly $\fht(K) \le n+2$, so it suffices to show that $\vht_1(K) \ge n+2$.  If not, there is an integer $w < n$ such that $\calV_{nc,1,w}(K) \ne \emptyset$, and so there exists a non-classical $w$-valuation $\nu$ on $K$ such that $F^\nu_0(K) \cap {^{1}\Gamma_0}(K) = \Bbbk$.  Since ${^{1}\Gamma_0}(K) = \Bbbk[x]$ and $F^\nu_0(K)$ is a subalgebra of $K$, it follows that $x \notin F^\nu_0(K)$, that is, $\nu(x) < 0$.  Now
$$
(n+1)\nu(x) + \nu(y) = \nu(\{x,y\}) \ge \nu(x) + \nu(y) - w,
$$
whence $n \nu(x) \ge - w$.  But then $\nu(x) \ge - w / n > -1$, contradicting the fact that $\nu(x)$ is a negative integer.  Therefore $\vht_1(K) \ge n+2$, as desired.
\end{proof}

Now we can solve the Isomorphism Problem for all
$K\{\bff\}$ where $\bff$ are of the form 
$q x^a y^b$ and $\Bbbk$ is algebraically closed. We have already seen that every 
such Poisson field is isomorphic to either
$K_q$, or $\Kweyl$, or $K_{1,n,0}$ for some $n\geq 2$. 
Now we are ready to show that these are non-isomorphic (whether or not $\Bbbk = \kbar$).

\begin{proposition}  [Isomorphism Problem]  
\label{zzpro4.6}
The following Poisson fields are 
pairwise non-isomorphic:
\begin{enumerate}
\item[(1)]
$K_q$, except for the obvious isomorphism $K_{q}\cong K_{-q}$,
\item[(2)]
$\Kweyl$ {\rm{(}}which is isomorphic to $K_{1,1,0}${\rm{)}},
\item[(3)]
$K_{1,n,0}$, for $n\geq 2$.
\end{enumerate}
\end{proposition}

\begin{proof} By Lemmas \ref{zzlem4.3}--\ref{zzlem4.5}, the Poisson fields in cases (1), (2), (3) have flag heights $2$, $0$, and $n+2$, respectively.  Thus, except possibly within case (1), no two of these are isomorphic.  By \cite[Corollary 5.4]{GoLa}, $K_{q}\cong K_{q'}$ if and only if $q'=\pm q$, which completes the proof.

The fact that $K_{q}\not\cong \Kweyl$ was already known, by \cite[Corollary 5.3]{GoLa} or \cite[Corollary 5.9(2)]{HTWZ2}. 
\end{proof}

\begin{remark}  \label{zzrem4.7}  
If a Poisson field $K = K\{\bff\}$ has flag height $\le 1$, then $K \cong \Kweyl$ by Proposition \ref{zzpro1.1}(2), whence $\fht(K) = 0$ by Lemma \ref{zzlem4.4}.  Thus, $\fht(K) = 1$ cannot occur.  In view of Lemmas \ref{zzlem4.3}, \ref{zzlem4.5} and Example \ref{zzexa8.5}, 
$\fht(K)$ can be 2, or any number strictly 
larger than $3$, or $\infty$. By \cite{GZ2}, there is a $K=K\{\bff\}$ such that $\fht(K)=3$. In summary the range of $\fht$ is $\{\infty\} \cup ({\mathbb Z}_{\geq 0} \setminus \{1\})$.
\end{remark}

Next we consider the Automorphism 
Problem for $K\{q x^a y^b\}$. 
The automorphism group of $K_q$ when 
$\Bbbk = \mathbb{C}$ is given in 
\cite[Theorem 1]{Bl}, 
while describing that of $\Kweyl$ 
is an open problem. We will work out 
the automorphism group of $K_{1,n,0}$ 
for $n\geq 2$. There are some 
obvious Poisson automorphisms:
$$
\eta_{a,b}: x\mapsto a x, 
\quad y \mapsto b y
$$
where $a\in \Bbbk^{\times}$ 
with $a^n=1$ and $b\in 
\Bbbk(x)^{\times}$, and
$$\tau_c: x\mapsto cx, 
\quad y\mapsto y^{-1}$$
where $c \in \kx$ with $c^n=-1$. 
Note that $\tau_c^2=\eta_{c^2,1}$.

\begin{proposition}[Automorphism Problem] 
\label{zzpro4.8} Let $K = K_{1,n,0}$ where $n\geq 2$ and $\Bbbk = \kbar$. 
\begin{enumerate}
\item[(1)] Then
$$\Aut_{Poi}(K)=\langle \eta_{a,b},\, \tau_c \mid a,c \in \kx,\; b\in 
\Bbbk(x)^{\times},\; a^n=1,\; c^n = -1\rangle,$$
and there is a short exact sequence
\begin{equation}  \label{E4.8.1}  \tag{E4.8.1}
1\to \Bbbk(x)^\times \rtimes C_n
\to \Aut_{Poi}(K)\to C_2\to 1,
\end{equation}
where we identify $C_n$ with $\{ a \in \kx \mid a^n = 1 \}$ and the action of $C_n$ on $\Bbbk(x)^{\times}$ is given by the rule
$a \cdot b(x) = b(ax)$.
\item[(2)]
If $n$ is odd, then the sequence \eqref{E4.8.1} is right-split and so $$\Aut_{Poi}(K) 
\cong  (\Bbbk(x)^{\times} \rtimes C_n)\rtimes C_2 \,,$$
where $C_2$ acts on $\Bbbk(x)^{\times} \rtimes C_n$ via the automorphism
$(b,a) \longmapsto (b(-x)^{-1},a)$.
 If $n$ is even, then \eqref{E4.8.1} is not right-split.
\end{enumerate}
\end{proposition}

\begin{proof}
(1) Let $G := \langle \eta_{a,b},\, \tau_c \rangle$ be the subgroup of $\Aut_{Poi}(K)$ generated by the $\eta_{a,b}$ and $\tau_c$ as described.

Let $\phi$
be an automorphism of $K$. Since ${^1\Gamma}_{0}(K)$
is a Poisson algebra invariant, by Lemma \ref{zzlem4.5}(1)
$\phi$ restricts to an algebra automorphism of 
${^1\Gamma}_{0}(K)=\Bbbk[x]$. Thus $\phi(x)=c x+d$
for some $c\in \Bbbk^{\times}$ and $d\in \Bbbk$. 
Then, with $f=\phi(y)$,
$$\begin{aligned}
(cx+d)^{1+n} f&=\phi(x^{1+n}y)=\phi(\{x,y\}) =\{\phi(x),\phi(y)\}=\{cx+d, f\}\\
&=c\{x,f\}=cf_{y} \{x,y\} =cf_{y} x^{1+n} y \,,
\end{aligned}
$$
which implies that
\begin{equation}
\notag
\frac{f_{y}}{f} = \frac{1}{c}
\left(\frac{cx+d}{x}\right)^{1+n} \frac{1}{y} \,.
\end{equation}
Solving this differential equation (or by Lemma \ref{zzlem2.1}), we 
obtain that
$$\frac{1}{c} \left(\frac{cx+d}{x}\right)^{1+n} =: z\in {\ZZ},
\quad {\text{and}} \quad f=g(x) y^{z}$$
 for some 
$g(x)\in \Bbbk(x)^{\times}$.
So $d=0$, whence $z=c^n\in {\ZZ}$. Since $x$ and $f$ generate the field
$\Bbbk(x,y)$, we must have $z= \pm 1$, i.e., $c^n= \pm 1$. 
If $z=1$, then $\phi = \eta_{c,g}$, while if $z = -1$, then $\phi \tau_c = \eta_{c^2,g^{-1}}$.  In either case, $\phi \in G$, and we have proved that $\Aut_{Poi}(K) = G$.

Next, set $H := \{ \eta_{a,b}  \mid a \in \kx,\; b\in 
\Bbbk(x)^{\times},\; a^n=1 \}$.  This is a subgroup of $G$ since
\begin{equation*}
\begin{aligned}
\eta_{a_1,b_1} \eta_{a_2,b_2} &= \eta_{a_1a_2, b_2(a_1x) b_1}  \\
\eta_{a,b}^{-1} &= \eta_{a^{-1}, b(a^{-1}x)^{-1}}
\end{aligned}
\qquad \forall\; \text{allowable}\; a_1,a_2,a,b_1,b_2,b,
\end{equation*}
 and it is normal in $G$ because
 \begin{equation}  \label{E4.8.2}  \tag{E4.8.2}
\tau_c^{-1} \eta_{a,b} \tau_c = \tau_{c^{-1}} \eta_{a,b} \tau_c = \eta_{a,b(c^{-1}x)^{-1}} \qquad \forall\; \text{allowable}\; a,b,c.
\end{equation}
Since $\tau_c \tau_{c'} = \eta_{cc',1}$ for all allowable $c$, $c'$, we see that $[G:H] = 2$.  Observe that the map
$$
\Bbbk(x)^\times \rtimes C_n \rightarrow H, \qquad (b,a) \mapsto \eta_{a,b}
$$
is a group isomorphism. The short exact sequence \eqref{E4.8.1} follows.

(2) If $n$ is odd, $\tau_{-1}$ is defined and $\tau_{-1}^2= \Id_K$, whence $\langle \tau_{-1}\rangle$
is a subgroup of $G$ of order $2$. Therefore $G = H \rtimes C_2$, and the action of $C_2$ on $H$ follows from \eqref{E4.8.2}. 

If $n$ is even, the assertion amounts to showing that $G$ has no elements of order $2$ outside $H$.  It follows from the proof of (1) that any $\psi \in G \setminus H$ has the form $\psi = \tau_c 
\eta_{a,b}$ for some allowable $a$, $b$, $c$.  Then $\psi(x) = acx$ and $\psi^2(x) = a^2c^2x$.  Since $n$ is even, $-1 = a^nc^n$ is a power of $a^2c^2$, so $a^2c^2 \ne 1$ and $\psi^2(x) \ne x$.  Thus no element of $G \setminus H$ has order $2$, as required.
\end{proof}

The Embedding and Dixmier Problems can also be solved using the 
$\gamma$-invariants.  We start with the $K_{q,l,0}$, $l \ge 2$.

\begin{proposition}
\label{zzpro4.9}
Let $q,r \in \kx$ and consider $K_{r,m,0}$ and $K_{q,n,0}$ for $m,n\geq 2$.
\begin{enumerate}
\item[(1)]{\rm{(Embedding Problem)}} 
There is an embedding from $K_{r,m,0}$ into $K_{q,n,0}$ if and only 
if $m\mid n$ and there exist $0 \ne z \in \ZZ$, $\beta \in \kx$ with $\beta^m = nqz/mr$.
\item[(2)]
The endomorphisms of $K_{q,n,0}$ are exactly those 
$\Bbbk$-algebra endomorphisms $\phi$ of $\Bbbk(x,y)$ such that
\begin{equation}  
\label{E4.9.1}\tag{E4.9.1}
\begin{aligned}
\phi(x)&=c x, \quad c \in \kx, \quad c^n=z\in \ZZ \setminus \{0\},\\
\phi(y)&=h(x) y^z, \quad h(x)\in \Bbbk(x)^{\times}.
\end{aligned}
\end{equation}
Consequently, every endomorphism of $K_{q,n,0}$ has image 
equal to $\Bbbk(x, y^{z})$ for some $0\neq z\in {\ZZ}$.
\item[(3)]{\rm{(Dixmier Problem)}}
$K_{q,n,0}$ does not have 
the Dixmier property.
\end{enumerate}
\end{proposition}

\begin{proof}
(1) First assume that $m \mid n$ and there exist $z$, $\beta$ as stated.  Set $d := n/m$, and note that $d(m+1) = n+d$.  In $K_{q,n,0}$, we have
\begin{align*}
\{\beta x^d, y^z\} &= \beta d x^{d-1} z y^{z-1} \{x,y\} = \beta d x^{d-1} z y^{z-1} q x^{1+n}y  \\
&= \frac{\beta nqz}{m}\, x^{d+n} y^z = \beta^{m+1} r x^{d(m+1)} y^z = r(\beta x^d)^{m+1}(y^z),
\end{align*}
and so $\Bbbk(x^d,y^z)$ is a Poisson subfield isomorphic to $K_{r,m,0}$.

Conversely, assume that there is a Poisson algebra
morphism $\phi: K_{r,m,0}\to K_{q,n,0}$ for some $m,n\geq 2$. 
Let us use $y_1,y_2$ as the generators of $K_{q,m,0}$ 
satisfying $\{y_1,y_2\}= r y_1^{1+m} y_2$. By Lemma \ref{zzlem3.7}(1), $\phi$ 
induces an injective map from 
$$\Bbbk[y_1]={^1\Gamma}_{0}(K_{r,m,0}) \to {^1\Gamma}_{0}(K_{q,n,0})=\Bbbk[x].$$
Thus $g(x):=\phi(y_1)$ is a non-constant polynomial. 
Let $f=\phi(y_2)$. Then we have 
$$\begin{aligned}
r g(x)^{1+m} f
&=\phi(r y_1^{1+m} y_2)=\phi(\{y_1,y_2\}) =\{\phi(y_1),\phi(y_2)\}  \\
&=\{g(x),f\} =g'(x) f_{y}\{x,y\}=g'(x) f_{y}  q x^{1+n} y \,.
\end{aligned}
$$
Thus
$$\frac{f_{y}}{f}=\frac{r g(x)^{1+m}}{q g'(x) x^{1+n}}\,\frac{1}{y}.$$
By Lemma \ref{zzlem2.1},
\begin{equation}  
\label{E4.9.2}\tag{E4.9.2}
\frac{r g(x)^{1+m}}{q g'(x) x^{1+n}}=:z\in {\ZZ}
\quad
{\text{and}} \quad
f=h(x) y^z
\end{equation}
for some $h(x)\in \Bbbk(x)^\times$. By the first equation,
we obtain that $m\deg g= n$.
Therefore $m\mid n$.

Write $g(x) = \sum_{i=0}^k \gamma_i x^i$ with $k>0$, $\gamma_i \in \Bbbk$, and $\gamma_k \ne 0$.  By \eqref{E4.9.2}, $z \ne 0$ and $r g(x)^{1+m} = zq g'(x) x^{1+n}$.  Comparing degrees and leading coefficients in this equation, we find that $k(1+m) = k+n$ and $r \gamma_k^{1+m} = zqk \gamma_k$, whence $\gamma_k^m = qkz/r = qnz/mr$, as required.

(2) It is easy to check that if $c \in \kx$ with 
$c^n=z\in {\ZZ}$ and $h(x)\in \Bbbk(x)^{\times}$, 
then the $\Bbbk$-algebra endomorphism $\phi$ of 
$\Bbbk(x,y)$ satisfying \eqref{E4.9.1} is an 
endomorphism of $K_{q,n,0}$.

Conversely, let $\phi$ be an arbitrary Poisson endomorphism 
of $K_{q,n,0}$. Following the proof and notation of part (1), 
with $r=q$ and $m=n$, we obtain \eqref{E4.9.2} and $n \deg g = n$, 
whence $\deg g = 1$. Consequently, \eqref{E4.9.2} implies 
that $g(x)=cx$ for some $c \in \kx$ and $c^{n}=z$. 
Therefore $\phi$ satisfies \eqref{E4.9.1}.

(3) For instance, if $c=2$ and $z=2^n$, part (2) shows that $K_{q,n,0}$ has an endomorphism such that $x \mapsto cx$ and $y \mapsto y^z$.
\end{proof}

\begin{proposition}[Embedding Problem]  \label{zzpro4.10}
Among the Poisson fields $\Kweyl$, $K_q$, and $K_{1,n,0}$ there are the following embeddings:
\begin{enumerate}
\item[(1)]
$\Kweyl$ embeds in $K_{1,n,0}$ for any $n \ge 2$.
\item[(2)] 
For $p,q \in \kx$, $K_p$ embeds in $K_q$ if and only if $p \in \ZZ q$.
\item[(3)]
For $m,n \ge 2$, $K_{1,m,0}$ embeds in $K_{1,n,0}$ if and only if $m \mid n$.
\end{enumerate}
There are no other embeddings:
\begin{enumerate}
\item[(4)] 
$\Kweyl$ and $K_{1,n,0}$ for $n \ge 2$ do not embed in any $K_q$.
\item[(5)]
$K_{1,n,0}$ for $n \ge 2$ does not embed in $\Kweyl$.
\item[(6)]
No $K_q$ embeds in $\Kweyl$ or in $K_{1,n,0}$ for $n \ge 2$.
\end{enumerate}
\end{proposition}

\begin{proof}
Part (1) is given by Corollary \ref{zzcor1.7}(2) (see Remark \ref{zzrem1.8}).

(2) If $p = mq$ for some $m \in \ZZ$, then by Examples \ref{zzexa1.5}(3), $\Bbbk(x,y^m)$ is a Poisson subfield of $K_q$ isomorphic to $K_p$.

For the converse, we use the following invariant from \cite{GoLa}, applied to $K_r$ where $r = p$ or $q$.  Set 
$$
C_r(u_1,u_2) := \bigl( \{u_i,u_j\}(u_iu_j)^{-1} \bigr) \in M_2(K_r) \qquad \forall\; u_1,u_2 \in K_r^\times
$$
and $C_r := \{ C_r(u_1,u_2) \mid u_1,u_2 \in K_r^\times \}$.  By \cite[Proposition 5.2(b)]{GoLa},
$$
C_r \cap M_2(\Bbbk) = \left\{ A \begin{pmatrix} 0&r\\ -r&0 \end{pmatrix} A^{\text{tr}} \bigm| A \in M_2(\ZZ) \right\}.
$$
If $K_p$ embeds in $K_q$, then $C_p \cap M_2(\Bbbk) \subseteq C_q \cap M_2(\Bbbk)$, and consequently $p \in \ZZ q$.

(3) Necessity is given by Proposition
\ref{zzpro4.9}(1).  Conversely, if 
$m \mid n$, then $z := (n/m)^{m-1}$ is 
an integer and $(n/m)^m = nz/m$, so 
Proposition \ref{zzpro4.9}(1) implies 
that $K_{1,m,0}$ embeds in $K_{1,n,0}$.

(4)(5) By \cite[Proposition 5.2(a)]{GoLa}, 
$\{u,v\} \ne 1$ for all $u,v \in K_q$, 
and so $\Kweyl$ does not embed in $K_q$. 
By Lemmas \ref{zzlem4.3}--\ref{zzlem4.5}, 
$^1\Gamma_0(K_{1,n,0}) = \Bbbk[x]$ while 
$^1\Gamma_0(K) = \Bbbk$ for $K = \Kweyl$ 
or $K_q$.  Thus, in view of Lemma 
\ref{zzlem3.7}(1), $K_{1,n,0}$ does not embed in $\Kweyl$ or in $K_q$.

(6) \cite[Corollary 5.9(2)]{HTWZ2} shows that $K_q$ does not embed in $\Kweyl$, and we use the same method with $K_{1,n,0}$.  By Lemma \ref{zzlem3.4}, there is a $(-n)$-valuation $\nu$ on $K\{x^{n+1}y\} = K_{1,n,0}$ such that $\nu(x) = 1$ and $\nu(y) = 0$.  We can also view $\nu$ as a $0$-valuation, and then it is classical.  If $K_q$ embeds in $K_{1,n,0}$, then $\nu$ restricts to a classical $0$-valuation on $K_q$.  This is impossible, since $\nu'(\{x,y\}) = \nu'(xy) = \nu'(x) + \nu'(y)$ for all ($0$-)valuations $\nu'$ on $K_q$. Therefore $K_q$ does not embed in $K_{1,n,0}$.
\end{proof}

In \cite[\S13.6]{SvdB}, Stafford and van den Bergh showed that $D_q(\Bbbk)$, the division ring of the quantum plane $\Bbbk \langle x,y \mid xy = qyx \rangle$, is isomorphic to a proper sub-division algebra of itself.  A Poisson version of this example shows that $K_q$ fails to satisfy the Dixmier property, as follows.

\begin{example}
\label{zzexa4.11}  
Let $K = K_q$ for some $q \in \Bbbk$. We first construct special elements $f,g \in K$ such that $\{g,f\}_w = gf(xy)^{-1}$.  Namely, set $s := (y-y^{-1})^{-1}$ and take
\begin{align*}
f& := s(xy-x^{-1}y^{-1})=(y-y^{-1})^{-1}(xy-x^{-1}y^{-1}),\\
g& := s(x-x^{-1})=(y-y^{-1})^{-1}(x-x^{-1}).
\end{align*}
It is clear that $s_x=0$ and $s_y=-s^2(1+y^{-2})$. Then
$$\begin{aligned}
g_x&=s(1+x^{-2}),\\
g_y&=-s^{2}(1+y^{-2})(x-x^{-1})=s^{2}(1+y^{-2})(x^{-1}-x).
\end{aligned}
$$
Similarly, we have
$$
\begin{aligned}
f_x&= s (y+x^{-2}y^{-1}),\\
f_y&= s_y (xy-x^{-1}y^{-1})+s (x+x^{-1}y^{-2})\\
&=-s^2(1+y^{-2})(xy-x^{-1}y^{-1})+s^{2} (y-y^{-1})(x+x^{-1}y^{-2})\\
&=s^2 [(-xy+x^{-1}y^{-1}-xy^{-1}+x^{-1}y^{-3})\\
&\qquad +(xy+x^{-1}y^{-1}-xy^{-1}-x^{-1}y^{-3})]\\
&=2s^2[x^{-1}y^{-1}-xy^{-1}]=2s^2 (x^{-1}-x) y^{-1}.
\end{aligned}
$$
Now we are ready to compute
$$\begin{aligned}
\{f,g\}_w &= f_x g_y-f_y g_x  \\
&= s (y+x^{-2}y^{-1})s^{2}(1+y^{-2})(x^{-1}-x)
-2s^2 (x^{-1}-x) y^{-1}s(1+x^{-2})\\
&=s^3(x^{-1}-x)
[(y+x^{-2}y^{-1})(1+y^{-2})-2 y^{-1}(1+x^{-2})]\\
&=s^3(x^{-1}-x)
[y+x^{-2}y^{-1} + y^{-1}+x^{-2}y^{-3}-2 y^{-1}-2y^{-1}x^{-2}]\\
&=s^3(x^{-1}-x)
[y-x^{-2}y^{-1} - y^{-1}+x^{-2}y^{-3}]\\
&=s^3(x^{-1}-x)
(y-x^{-2}y^{-1})(1-y^{-2})\\
&=s^3(x^{-1}-x)(xy-x^{-1}y^{-1}) x^{-1}(y-y^{-1})y^{-1}\\
&=s (xy-x^{-1}y^{-1}) s (x-x^{-1}) (-x^{-1}y^{-1})=fg (-x^{-1}y^{-1}).
\end{aligned}
$$
Thus $\{g,f\}_w = gf(xy)^{-1}$, as announced.  Consequently,
$$
\{g,f\} =  \{g,f\}_w \, qxy = qgf,
$$
which implies that $\Bbbk(f,g)$ is a Poisson subfield of $K$ isomorphic to $K$.

Note that $K$ has a (Poisson) 
automorphism $\sigma$ such that $x \mapsto x^{-1}$ and $y \mapsto y^{-1}$.  It is clear that
both $f$ and $g$ lie in the fixed field $\Bbbk(x,y)^{\sigma}$. Therefore $\Bbbk(f,g) \subseteq K^{\sigma} \subsetneq K$, showing that $K$ fails to satisfy the Dixmier property.
\end{example}

\begin{proposition}[Dixmier Problem]  \label{zzpro4.12}
Assume that $\Bbbk$ is algebraically closed, $q \in \kx$, and $\kappa \in \ZZ^2$.  Then $K_{q,\kappa}$ fails to satisfy the Dixmier property.
\end{proposition}

\begin{proof} 
By Corollary \ref{zzcor4.2}, $K_{q,\kappa}$ is isomorphic to $K_q$ or $\Kweyl$ or $K_{1,n,0}$ for some $n \ge 2$.  The failure of the Dixmier property in these cases is covered by Example \ref{zzexa4.11}, Example \ref{zzexa1.5}(1), and Proposition \ref{zzpro4.9}(3), respectively.
\end{proof}

%%%%%%%%%%%%%%%%%%%%%%%%%

\section{Family (2): $K\{p(x)xy\}$ where $p(x)\in \Bbbk[x]$}
\label{zzsec5}

Note that the class of the $K\{q x^a y^b\}$ 
is a subclass of the $K\{p(x)xy\}$
as every $K\{q x^a y^b\}$ is isomorphic 
to some $K\{q x^{n} xy\}$ 
(Lemma \ref{zzlem4.1}(2) plus the case 
$K\{qxy\}$). We skip over the case when 
$p(x)$ has degree $0$, since that is the 
case $K\{qxy\} = K_q$. When $p(x)$ has 
degree $1$, we have the following:

\begin{example}
\label{zzexa5.1}
If $a \in \kx$ and $q\in \Bbbk$, then 
$$K\{(ax+q)xy\} \cong 
\begin{cases} K_q & {\text{if $q\neq 0$}}\\
\Kweyl & {\text{if $q=0$.}}
\end{cases}
$$
Since $K\{(ax+q)xy\} \cong K\{x(-ay^2-qy)\}$, this follows from Corollary \ref{zzcor1.3}(3).

In view of Proposition \ref{zzpro4.6}, we also have
$$K\{(a_1x+q_1)xy\} \cong K\{(a_2x+q_2)xy\} \ \iff \ q_1 = \pm q_2$$
for any $a_1,a_2 \in \kx$ and $q_1,q_2 \in \Bbbk$.
\end{example}

As a consequence, if $p(x)$ has degree 1 or less, then $K\{ p(x)xy\}$
does not have the Dixmier property.

If $\deg p(x)>1$, then the situation is different. 

\begin{proposition} 
\label{zzpro5.2}
Let $K_1 := K\{p_1(x)xy\}$ and $K_2 :=K\{p_2(x)xy\}$ where 
$p_1(x)$ and $p_2(x)$ are nonzero polynomials  in $\Bbbk[x]$. Suppose
\begin{enumerate}
\item[(a)]
$\deg p_1(x)\geq 2$, and 
\item[(b)]
$\deg p_1(x)\geq \deg p_2(x)$. 
\end{enumerate}
Then ${^1\Gamma}_{0}(K_1)=\Bbbk[x]$.
\smallskip

Now let $\phi: K_1\to K_2$ be a Poisson homomorphism. Then the following 
hold.
\begin{enumerate}
\item[(1)]
$\deg p_1(x)=\deg p_2(x)$.
\item[(2)]
$\phi(x)=\alpha x+ \beta$ for some $\alpha\in \Bbbk^{\times}$ and
$\beta\in \Bbbk$. 
\item[(3)]
If further $p_2(x)=p_1(x)$ and
$p_1(x)$ has a nonzero root in $\overline{\Bbbk}$, then 
$\phi$ is an automorphism of $K_1$.
\item[(4)]
$\vht_1(K_1)=\fht(K_1)= \deg p_1 +2$. 
\end{enumerate}
\end{proposition}

\begin{proof} 
Since $\deg p_1(x) \ge 2$, Corollary \ref{zzcor3.12} implies that 
${^1\Gamma}_{0}(K_1)=\Bbbk[x]$.

(1)(2) First we claim that $\deg p_2(x)\geq 2$, and consequently,
${^1\Gamma}_{0}(K_2)={^1\Gamma}_{0}(K_1)=\Bbbk[x]$.
By Lemma \ref{zzlem3.7}(1), $\phi$ maps 
${^1\Gamma}_{0}(K_1)$ to ${^1\Gamma}_0(K_2)$ injectively. 
Thus ${^1\Gamma}_{0}(K_2)\neq \Bbbk$. If $\deg p_2(x)\leq 1$,
then by Example \ref{zzexa5.1}, $K_2$ is isomorphic to either
$K_q$ or $\Kweyl$. In both cases ${^1\Gamma}_0(K_2)=\Bbbk$
by Lemmas \ref{zzlem4.3}(1) and \ref{zzlem4.4}(2). This 
yields a contradiction. So $\deg p_2(x)\geq 2$. By 
Corollary \ref{zzcor3.12}, ${^1\Gamma}_{0}(K_2)=\Bbbk[x]$.

Since $\phi$ maps ${^1\Gamma}_{0}(K_1)$ to ${^1\Gamma}_{0}(K_2)$, 
$\phi(x) = g(x)$ for some $g \in \Bbbk[x]\subseteq K_2$. 
Let $\phi(y)=h\in K_2$. Note that $g$ and $h$ must be 
algebraically independent over $\Bbbk$; in particular, 
$g,h \notin \Bbbk$. Now
$$p_1(g(x)) g(x) h = \phi(p_1(x) xy) = \phi(\{x,y\}) = \{g,h\} 
= \{g,h\}_w \{x,y\} = g'(x) h_y p_2(x) xy$$
and $g'(x) \ne 0$ because $g \notin \Bbbk$, whence
$$
\frac{h_y}{h} = \frac{p_1(g(x)) g(x)}{g'(x) x p_2(x)}\, \frac1y \,.
$$
Lemma \ref{zzlem2.1}(2), with $F := \Bbbk(x)$ and $t :=y$, 
implies that $p_1(g(x)) g(x)/ g'(x) x p_2(x) = :z \in \ZZ$ and 
$h = v(x) y^z$ for some $v(x) \in \Bbbk(x)^\times$. Since 
$g$ and $h$ are algebraically independent, $z \ne 0$.

Comparing degrees in the equation 
\begin{equation} 
\label{E5.2.1}\tag{E5.2.1}
p_1(g(x)) g(x) = z g'(x) x p_2(x),
\end{equation}
we find that
$$(\deg p_1)(\deg g) + \deg g = \deg g - 1 + 1 + \deg p_2\leq \deg g+\deg p_1 \,,$$
whence $\deg g = 1$ and $\deg p_1=\deg p_2=:d\geq 2$. So we have proved 
parts (1) and (2).

(3) We have $g(x) = \phi(x) = \alpha x + \beta$ as in the proof of part (2). Comparing leading coefficients in 
\eqref{E5.2.1} (with $p_1=p_2$), we see that $\alpha^d = z$. Evaluating \eqref{E5.2.1} at $x=0$ yields
$$p_1(\beta) \beta = 0.$$

We claim that $\alpha$ is a root of unity.  Suppose not.

First consider the case $\beta \ne 0$. Then $\beta$ is a root of 
$p_1$. We show that 
$$r_i := (\alpha^i + \alpha^{i-1} + \cdots + 1) \beta$$ 
is a root of $p_1$ for all $i \ge 0$. If $p_1(r_i) = 0$ for some $i$, 
then $p_1(g(r_i)) g(r_i) = 0$ by \eqref{E5.2.1}. Since 
$$g(r_i) = \alpha r_i + \beta = r_{i+1} \,,$$
which is nonzero because $\alpha^{i+1} + \alpha^i + \cdots + 1 
= (\alpha^{i+2}-1)/(\alpha-1) \ne 0$, we obtain $p_1(r_{i+1}) = 0$. 
Thus, by induction, all the $r_i$ are roots of $p_1$.  This is 
impossible, because the elements 
$r_i = (\alpha^{i+1}-1)\beta/(\alpha-1)$ are all distinct when 
$\alpha$ is not a root of unity.

Now assume that $\beta = 0$, whence \eqref{E5.2.1} reduces to
$$p_1(\alpha x) = z p_1(x).$$
By assumption, $p_1$ has a nonzero root 
$r \in \overline{\Bbbk}$. By induction, $\alpha^i r$ is a root of 
$p_1$ for all $i \ge 0$, which again yields a contradiction.

Thus $\alpha$ must be a root of unity, as claimed. Consequently, 
due to the equation $\alpha^d = z$, we conclude that $z = \pm 1$. 
As a result,
$$\phi(K_1) = \Bbbk(g,h) = \Bbbk(\alpha x + \beta, v(x) y^z) 
= \Bbbk(x,y) = K_1.$$
Therefore $\phi$ is an automorphism of $K_1$, as claimed.

(4) We have $\vht_1(K_1) \le \fht(K_1)$ by Lemma \ref{zzlem3.14}(4), and $\fht(K_1) \le d+2$ where $d := \deg p_1$.  If $\vht_1(K_1) < d+2$, then as in the proof of Lemma \ref{zzlem4.5}(3), there is an integer $w < d$ such that there exists a non-classical $w$-valuation $\nu$ on $K_1$ with $\nu(x) < 0$.  Now
$$
\nu(p_1) + \nu(x) + \nu(y) = \nu(\{x,y\}) \ge \nu(x) + \nu(y) - w,
$$
so $\nu(p_1) \ge - w$.  Since $\nu(x) < 0$, it follows from Lemma \ref{zzlem3.10}(1) that $\nu(p_1) = d \nu(x)$.  But then $\nu(x) \ge - w/d > -1$, which is impossible.  Therefore $\vht_1(K_1) \ge d+2$, establishing (4).
\end{proof}

Comparing Proposition \ref{zzpro5.2} with Lemmas \ref{zzlem4.3} and \ref{zzlem4.4}, we see that $K\{p(x)xy\}$ is not isomorphic to $\Kweyl$ or any $K_q$ when $p \in \Bbbk[x]$ is a polynomial with degree $\ge2$.

Using the above proposition we can solve a few 
problems concerning the class of Poisson fields $K\{p(x) xy\}$.
We start with the Dixmier Problem.

\begin{corollary}[Dixmier Problem]
\label{zzcor5.3}
Let $K = K\{p(x)xy\}$ for some polynomial $p \in \Bbbk[x]$. 
Then $K$ has the Dixmier property if and only if $\deg p \ge 2$ 
and $p$ has a nonzero root in $\overline{\Bbbk}$.
\end{corollary}

\begin{proof}
Sufficiency is given by Proposition \ref{zzpro5.2}(3). It 
remains to show that the Dixmier property fails if either 
$\deg p \le 1$ or $p$ has no nonzero roots in $\overline{\Bbbk}$.

If $\deg p \le 1$, then by Example \ref{zzexa5.1}, 
$K \cong K_{\Weyl}$ or $K \cong K_q$ for some $q \in \kx$. 
These cases are 
covered by Examples \ref{zzexa1.5}(1) and \ref{zzexa4.11}.

Now assume that $\deg p \ge 2$ and $p$ has no nonzero 
roots in $\overline{\Bbbk}$. Hence, $p(x) = q x^n$ for 
some $q \in \kx$ and $n \ge 2$. Then $K = K_{q,n,0}$. By Proposition \ref{zzpro4.9}(3), $K$ fails to 
satisfy the Dixmier property.
\end{proof}

\begin{remark}
\label{zzrem5.4}
Although a Poisson field with the Dixmier property does not have any proper Poisson subfields isomorphic to itself, it may 
still have many Poisson subfields. Let $K = K\{p(x)xy\}$ as 
in Proposition \ref{zzpro5.2}. Then $\Bbbk(x,y^m)$, for any 
nonzero integer $m$, is a Poisson subfield of $K$, isomorphic 
to $K\{mp(x)xy\}$.

There may even be $q$-skew Poisson subfields. Suppose there 
exist $h(x)  \in \Bbbk(x)$ and $\lambda \in \kx$ such that 
$\frac{h'(x)}{h(x)} = \frac\lambda{p(x)x}$. 
(See Remark \ref{zzrem2.3} for an instance of $p$ for 
which such $h$ and $\lambda$ exist.) Then 
$\Bbbk(h,y) \cong K_\lambda$ by Corollary \ref{zzcor1.7}(3).

Questions: Does $K$ always have a $q$-skew Poisson subfield? 
If not, for which $p \in \Bbbk[x]$ does $K\{p(x)xy\}$ have a 
$q$-skew Poisson subfield?
\end{remark}

Next we solve the Isomorphism and Automorphism Problems for the 
class of Poisson fields $K\{p(x)xy\}$. 

\begin{corollary}[Isomorphism Problem]
\label{zzcor5.5} Suppose $p_1(x)$ and $p_2(x)$ are two 
polynomials in $\Bbbk[x]$ of positive degree. Then $K\{p_1(x)xy\}
\cong K\{p_2(x)xy\}$ if and only if there are 
$a\in \Bbbk^{\times}$
and $b\in \Bbbk$ such that 
\begin{equation}
\label{E5.5.1}\tag{E5.5.1}
a^{-1}(ax-b)p_1(ax-b)=\pm xp_2(x).
\end{equation}
In this case, $\deg p_1=\deg p_2$. 
\end{corollary}

\begin{proof} If $\deg p_1(x) = 1$ and $\deg p_2(x) = 1$, 
then the assertion follows from the classification given in 
Example \ref{zzexa5.1} (details are omitted). 

Now we assume that at least one of $p_1$ and $p_2$ has degree 
at least 2. By symmetry, we may assume $\deg p_1\geq \deg p_2$.

Suppose that $a^{-1}(ax-b)p_1(ax-b)=e xp_2(x)$ where $e=\pm 1$. It 
is clear that there is a Poisson 
isomorphism $K\{p_1(x)xy\}\to K\{p_2(x)xy\}$ that sends $x \mapsto ax-b$ and $y \mapsto y^{e}$.

Conversely, assume that $\phi: K\{p_1(x)xy\}\to K\{p_2(x)xy\}$
is a Poisson isomorphism.  By the proof
of Proposition \ref{zzpro5.2} we have $\phi(x) = g(x) = \alpha x + \beta$ for some $\alpha \in \kx$, $\beta \in \Bbbk$, while $\phi(y)=v(x) y^z$ for 
some $v \in \Bbbk(x)^\times$ and some nonzero integer $z$, and by \eqref{E5.2.1},
\begin{equation}
\label{E5.5.2}\tag{E5.5.2}
p_1(g(x)) g(x) =z g'(x) x p_2(x).
\end{equation}
Since $\phi$ is an isomorphism, $z=\pm 1$. Thus \eqref{E5.5.1} is 
another form of \eqref{E5.5.2}. The assertion follows.
\end{proof}

\begin{remark}  
\label{zzrem5.6}
We can divide \eqref{E5.5.1} into two cases. Let $p_1,p_2 \in 
\Bbbk[x]$ be polynomials of positive degree. Then $K\{p_1(x)xy\} 
\cong K\{p_2(x)xy\}$ if and only if one of the following holds:
\begin{itemize}
\item
$p_2(x) = \pm p_1(ax)$ for some $a \in \kx$.
\item 
There exist $a,b \in \kx$ and $u \in \Bbbk[x]$ such that 
$p_1(x) = (x+b) u(x)$ and $p_2(x) = \pm (ax-b) u(ax-b)$.
\end{itemize}

Sufficiency is clear from Corollary \ref{zzcor5.5}. Conversely, 
if $K\{p_1(x)xy\} \cong K\{p_2(x)xy\}$, then \eqref{E5.5.1} holds 
for some $a \in \kx$, $b \in \Bbbk$. In case $b=0$, this equation 
reduces to $p_1(ax) = \pm p_2(x)$.  In case $b\ne 0$, evaluating 
\eqref{E5.5.1} at $x=0$ shows that $-b$ is a root of $p_1$, 
whence $p_1(x) = (x+b) u(x)$ for some $u \in \Bbbk[x]$. 
Returning to \eqref{E5.5.1}, it follows that $p_2(x) = 
\pm (ax-b) u(ax-b)$.
\end{remark}

\begin{corollary}  \label{zzcor5.7}
{\rm(Automorphism Problem).}  
Let $K = K\{p(x)xy\}$ where $p \in \Bbbk[x]$ is a polynomial of degree $d \ge 2$.  Set
\begin{align*}
E_p &:= \{ (a,b,e) \in \kx \times \Bbbk \times \{\pm1\} \mid a^{-1} (ax-b) p(ax-b) = e x p(x) \}  \\
G_p &:= \biggl\{ \biggl( \begin{pmatrix} a&0\\ b&1 \end{pmatrix}, e \biggr) \biggm| (a,b,e) \in E_p \biggr\}.
\end{align*}
\begin{enumerate}
\item[(1)]
$G_p$ is a subgroup of $GL_2(\Bbbk) \times \{\pm1\}$, and there is a short exact sequence
$$
1 \rightarrow \Bbbk(x)^\times \rightarrow \Aut_{Poi}(K) \rightarrow G_p \rightarrow 1 \,.
$$
\item[(2)] 
$|G_p| \le d(d+1)$. 
\item[(3)]
If $p$ has $\ZZ$-linearly independent roots $r_1,\dots,r_d$ in $\kbar$ such that $r_i^{2d} \ne r_j^{2d}$ for all $i \ne j$, then $|G_p| = 1$ and $\Aut_{Poi}(K) \cong \Bbbk(x)^\times$.
\end{enumerate}
\end{corollary}

\begin{proof}
(1) That $G_p$ is a subgroup of $GL_2(\Bbbk) \times \{\pm1\}$ is a routine check.

Given $(a,b,e) \in E_p$ and $v \in \Bbbk(x)^\times$, there is a Poisson automorphism $\phi_{a,b,e,v}$ of $K$ sending $x \mapsto ax-b$ and $y \mapsto v(x) y^e$ (as already noted when $v=1$ in the proof of Corollary \ref{zzcor5.5}).  Conversely, it follows from the proofs of Proposition \ref{zzpro5.2} and Corollary \ref{zzcor5.5} that every automorphism of $K$ has the form $\phi_{a,b,e,v}$.  Observe that
$$
\phi_{a,b,e,v} \phi_{a',b',e',v'} = \phi_{aa', ba'+b', ee', v(x)^{e'} v'(ax-b)}
$$
for all $(a,b,e), (a',b',e') \in E_p$ and $v,v' \in \Bbbk(x)^\times$.  Consequently, the rule
$$
\theta(\phi_{a,b,e,v}) = \biggl( \begin{pmatrix} a&0\\ b&1 \end{pmatrix}, e \biggr)
$$
defines a homomorphism $\theta$ from $\Aut_{Poi}(K)$ onto $G_p$, and
$$
\ker \theta = \{ \phi_{1,0,1,v} \mid v \in \Bbbk(x)^\times \} \cong \Bbbk(x)^\times \,.
$$
The displayed short exact sequence follows.

(2) Let $r_1,\dots,r_m$ be the distinct roots of $p$ in $\kbar$.

If $m=1$ and $r_1=0$, then $p(x) = \lambda x^d$ for some $\lambda \in \kx$, and the equation
\begin{equation}  \label{E5.7.1}  \tag{E5.7.1}
a^{-1} (ax-b) p(ax-b) = e x p(x),
\end{equation}
reduces to $a^{-1}(ax-b)^{d+1} = ex^{d+1}$.  This holds if and only if $b=0$ and $a^d=e$, so there are at most $2d$ choices for $(a,b,e)$ in this case, that is, $|E_p| \le 2d \le d(d+1)$.

In general, on comparing leading coefficients in \eqref{E5.7.1}, we find that $a^d = e$.  Moreover, as noted in Remark \ref{zzrem5.6}, either $b=0$ or $-b$ is a root of $p$, so $b \in \{0,-r_1,\dots,-r_m\}$.

For the case $(m,r_1) \ne (1,0)$, we may assume that $r_1 \ne 0$.  Then \eqref{E5.7.1} reduces to
$$
a^{-1} (ax-b) \prod_{i=1}^m (ax-b-r_i)^{s_i} = e x \prod_{i=1}^m (x-r_i)^{s_i}
$$
for some positive integers $s_i$, and consequently
$$
\{ a^{-1}b, a^{-1}(b+r_1), \dots, a^{-1}(b+r_m) \} = \{0,r_1,\dots,r_m\}.
$$
There are at most $m$ choices for nonzero $b$, and when $b \ne 0$ we must have $a^{-1}b = r_i$ for some nonzero $r_i$, so given $b$ there are at most $m$ choices for $a$.  Since $e = a^d$, we have at most $m^2$ choices for $(a,b,e)$ with $b \ne 0$.  On the other hand, when $b=0$ we have $a^{-1}r_1 = r_j$ for some nonzero $r_j$, leaving at most $m$ choices for $a$ and at most $m$ choices for $(a,0,e)$.  Therefore $|E_p| \le m^2+m \le d(d+1)$ in this case.

Since $|G_p| = |E_p| $, part (2) is proved.

(3) Let $(a,b,e) \in E_p$.  As noted above, $a^d = e$ and $b \in \{0,-r_1,\dots,-r_d\}$, while
$$
\{ a^{-1}b, a^{-1}(b+r_1), \dots, a^{-1}(b+r_d) \} = \{0,r_1,\dots,r_d\}.
$$

If $b \ne 0$, then $b = - r_k$ for some $k$ and $- a^{-1} r_k = a^{-1}b = r_l$ for some $l$.  Then
$$
r_k^{2d} = (- a^{-1} r_k)^{2d} = r_l^{2d} \,,
$$
which forces $l=k$ and $a = -1$.  Since $d \ge 2$, there is an index $i \in [1,d]$ different from $k$, and
$$
r_k - r_i = a^{-1}(b+r_i) = r_j
$$
for some $j$, contradicting our assumptions.  Thus $b=0$.

Now $r_1 \ne 0$, and $a^{-1} r_1 = r_i$ 
for some $i$, so $r_1^{2d} = r_i^{2d}$, 
whence $i=1$ and $a=1$.  Finally, 
$e = a^d = 1$ as well.

Thus $E_p = \{(1,0,1)\}$ under the 
present hypotheses, and thus $|G_p| = 1$.
\end{proof}

\begin{remark}
\label{zzrem5.8}
Assume now that $\Bbbk$ is algebraically closed 
and $p_1 \in \Bbbk[x]$ has positive degree. Then there is 
some $a \in \kx$ such that $a^{-1} (ax) p_1(ax)$ is monic.
Consequently, $p_2(x):=p_1(ax)$ is monic and $K\{p_1(x)xy\} 
\cong K\{p_2(x)xy\}$ [Corollary \ref{zzcor5.5}]. Thus, in 
studying $K\{p(x)xy\}$ (for example, for the Isomorphism 
and Automorphism Problems), we may assume that $p$ is monic.

Let $p_1,p_2 \in \Bbbk[x]$ be two monic polynomials of 
positive degree. We say they are \emph{equivalent}, and 
write $p_1\sim p_2$, if there are $a\in \Bbbk^{\times}$, $b\in \Bbbk$, and $e \in \{\pm1\}$ such that 
\begin{equation}
\label{E5.8.1}\tag{E5.8.1}
a^{-1}(ax-b)p_1(ax-b)=e xp_2(x).
\end{equation}
In this case $a$, $b$, $e$ are called the \emph{equivalence parameters}.

By Corollary \ref{zzcor5.5}, $\sim$ is an equivalence 
relation, and if $p_1\sim p_2$, then $\deg p_1 = \deg p_2$.
We claim that if $p_1\sim p_2$ and $d:= \deg p_1 = \deg p_2$, then the number of equivalence parameters is at most $2d(d+1)$.

Since $p_1$ and $p_2$ are assumed to be monic, it follows from \eqref{E5.8.1} that $a^d = e$.  Hence, there are at most 
$2d$ choices for $(a,e)$.  As noted in Remark \ref{zzrem5.6}, either $b=0$ or $-b$ is a root of $p_1$, so there are at most $d+1$ choices for $b$.  Thus, with $p_1$ fixed, there are at 
most $2d(d+1)$ choices for $(a,b,e)$, and therefore at most 
$2d(d+1)$ choices for $p_2$.

In summary we have a one-to-one correspondence between
the isomorphism classes of the $K\{p(x)xy\}$ with $p(x)$ 
being a monic polynomial of positive degree and 
the equivalence classes of monic polynomials $p(x)$ of 
positive degree. In other words, if we define a map 
$p(x)\mapsto K\{p(x)xy\}$, and consider the image in the 
class of isomorphism classes of all Poisson fields, then 
this is a finite-to-one map. So this shows that 
the moduli space of the Poisson fields $K\{\bff\}$ is infinite dimensional
(and very complicated).
\end{remark}

%%%%%%%%%%%%%%%%%%%%%%%%%

\section{Family (3): $K\{\bff\}$ where $\bff\in 
\Bbbk[x,y]$ is homogeneous}
\label{zzsec6}

For simplicity, we assume that $\Bbbk$ is algebraically closed
in this section. Then $\bff$ is a product of linear 
terms $ax+by$ where $a,b \in \Bbbk$, not both zero.

\begin{lemma}
\label{zzlem6.1}
Let $K = K\{\bff\}$ where $\bff\in 
\Bbbk[x,y]$ is homogeneous of degree $d\geq 1$.
\begin{enumerate}
\item[(1)]
If $\bff$ has only one linear factor {\rm{(}}up to scalars{\rm{)}}, then 
$K\cong K\{a x^d\}\cong \Kweyl$ for some $a \in \kx$.
\item[(2)]
If $\bff$ has exactly two different linear factors {\rm{(}}up to scalars{\rm{)}}, then 
$K\cong K\{q x^a y^b\}$ for some $q \in \kx$ and $a,b \in \Zpos$. As a consequence, it is isomorphic 
to either $\Kweyl$, some $K_q$, or $K_{1,n,0}$ for some $n:=d-2\geq 2$.
\item[(3)]
If $\bff$ has at least three different linear factors {\rm{(}}up to scalars{\rm{)}}, 
then $K$ is isomorphic to $K\{p(x,y)(x+y)xy\}$
for some homogeneous polynomial $p(x,y)$ of degree $d-3$.
\item[(4)]
If $\deg \bff\leq 3$, then $K$ is isomorphic to either $\Kweyl$
or some $K_q$.
\end{enumerate}
If $\bff$ is in cases {\rm{(1), (2), (4)}}, then it is not $1$-flabby.
\end{lemma}

\begin{proof} 
(1)(2) By a linear change of variables, any linear factor of $\bff$ can be changed to $x$, and if $\bff$ has a second, non-proportional linear factor, that may be changed to $y$. Thus, we may assume that $\bff$ equals $ax^d$ or $qx^ay^b$ or $g(x,y)xy$ for some homogeneous $g$.  Parts (1) and (2) thus follow from Proposition \ref{zzpro1.1}(1) and Lemma \ref{zzlem4.1}.

(3) Here we may assume $\bff = g(x,y)xy$ where $g \in \Bbbk[x,y]$ is homogeneous with a linear factor $ax+by$ for some $a,b \in \kx$.  The further change of variables $x \rightarrow a^{-1}x$ and $y \rightarrow b^{-1}y$ allows us to assume that $a=b=1$.

(4) If $\bff$ has at most two different linear factors (up to scalars), the result follows from parts (1) and (2).  Otherwise, we may assume
that $\{x,y\}=xy(x+y)$ by part (3). Then 
$\{x^{-1},y^{-1}\}=x^{-1}+y^{-1}$, so after a change of generators
$x\to x^{-1}$, $y\to y^{-1}$, we get 
$K \cong K\{x+y\}$. By part (1), $K\cong \Kweyl$.

The final assertion follows from the results of \S4.1, which show that ${^1\Gamma}_0(K)$ 
is isomorphic to either $\Bbbk$ or $\Bbbk[x]$ if $K$ is 
$\Kweyl$, or $K_q$, or $K_{1,n,0}$. 
\end{proof}

For the rest of this section, we assume implicitly that 
$\deg \bff\geq 4$. To avoid some technicalities, we add some generic 
conditions as below. As already noted, $\bff$ is a product of 
linear forms $l_i := a_i x+b_i y \neq 0$. In somewhat more generality, we 
consider the case $\bff =l_1\cdots l_n$ where $l_i := 
a_i x+b_i y+c_i\not\in \Bbbk$ for scalars $a_i,b_i,c_i \in \Bbbk$. 
Here $\deg \bff =n$. We say $\bff$ 
\emph{has distinct linear forms} if $l_i\not\in {\Bbbk^\times} l_j$ 
for all $i\neq j$, which is equivalent to saying that $\bff$ does 
not have a square divisor of positive degree. 

By Lemma \ref{zzlem6.1}, many homogeneous polynomials in $\Bbbk[x,y]$ of degree 3 or 
less are not $1$-flabby. 
If $\bff$ is of the form $x^2\prod_{i=1}^{n} l_i$, 
then one can also check that $\bff$ is not $1$-flabby
(by constructing a $1$-valuation $\nu$ via Lemma \ref{zzlem3.4}
with $\nu(y)=-1$ and $\nu(x)\gg0$).
On the other hand,  polynomials of degree 4 or more are likely to be $1$-flabby under 
some generic conditions. Below is one case.

\begin{lemma}
\label{zzlem6.2}
Let $\bff \in \Bbbk[x,y]$ be a product
$\prod_{i=1}^n l_i$ of 
distinct linear forms, where $l_i=a_i x+b_i y+c_i\not\in \Bbbk$. 
Suppose 
\begin{equation}
\label{E6.2.1}\tag{E6.2.1}
{\text{
for each $i$, there are at least three $j$
such that $l_i$ and $l_j$ generate $\Bbbk[x,y]$. }}
\end{equation}
Then ${^1\Gamma_0}(K\{\bff\}) = \Bbbk[x,y]$, i.e., $\bff$ is $1$-flabby.  This holds in particular if $\bff$ is a homogeneous polynomial of degree $n\ge4$ and $\bff$ has distinct linear forms.

As a consequence, $\vht_1(K\{\bff\})= \fht(K\{\bff\})=\deg \bff$ and $K\{\bff\}$ is height 
cohereditary.
\end{lemma}

\begin{proof} 
Note that \eqref{E6.2.1} implies $n \ge 4$.
By Corollary \ref{zzcor3.9}, 
${^1\Gamma}_{0}(K\{\bff\})\subseteq \Bbbk[x,y]$. Since
$\sum_i \Bbbk l_i$ generates $\Bbbk[x,y]$, it suffices to show
that each $l_i\in {^1\Gamma}_{0}(K\{\bff\})$. By definition,
it remains to show that $\nu(l_i)\geq 0$ for all $i$ and 
for all $1$-valuations $\nu$ of $K\{\bff\}$.

Let $\nu$ be any $1$-valuation of $K\{\bff\}$. Set 
$$
\alpha :=\min( \nu(l_i)\mid i \in [1,n] ) \qquad\text{and}\qquad 
\beta :=\max( \nu(l_i)\mid i \in [1,n] ).
$$
Then 
both $\alpha$ and $\beta$ are finite. If $\alpha\geq 0$,
then $\nu(l_i)\geq \alpha\geq 0$ for all $i$, as required. 
It remains to consider the situation when $\alpha<0$. 

Case 1: $\alpha=\beta$. Then $\nu(l_i)=\alpha$ for all
$i$. After possible renumbering, $l_1$ and $l_2$ generate $\Bbbk[x,y]$. By an affine change of basis
$x\to l_1$, $y\to l_2$, we can assume that 
$l_1=x$ and $l_2=y$. Then 
$$
n \alpha =\nu(\bff)=\nu(\{x,y\})
\geq \nu(x)+\nu(y)-1 = 2\alpha -1 \,.
$$
Thus, $(n-2)\alpha \geq -1$.
Since $n \ge 4$ and $\alpha$ is an integer, we have
$\alpha\geq 0$, yielding a contradiction.

Case 2: $\alpha<\beta$. 
Without loss of generality, $\alpha=\nu(l_1)$
and $\beta=\nu(l_2)$. Then $l_1$ is not of the form
$al_2+b$ for any $a,b\in \Bbbk$. So $l_1$ and $l_2$
generate $\Bbbk[x,y]$. After an affine change of generators, we may assume that $l_1=x$ and $l_2=y$.

Subcase 2.1: $\beta > 0$.
For each $i\geq 3$, $l_i=a_i x+b_i y+c_i$ which 
has $\nu$-value $\alpha$ if $a_i\neq 0$ and $0$
if $a_i=0$ (since $l_i \ne b_i y$ when $i \ge 3$). By hypothesis \eqref{E6.2.1}, there are $n'\geq 3$
factors $l_j$ (including $l_1$) such that $l_j$ and $l_2$
generate $\Bbbk[x,y]$. Such $l_j= a_j x+b_j y+c_j$ where $a_j\neq 0$, and thus $\nu(l_j)=\alpha$. For all other $i\neq 2, j$, we have $l_i=b_i y+c_i$ 
(with $b_i,c_i\neq 0$) and consequently, $\nu(l_i)= 0$.  Now
$$n' \alpha + \beta =\nu(\bff)=\nu(\{x,y\})\geq \alpha+\beta-1,$$
which implies that $(n'-1)\alpha \geq -1$, yielding a contradiction.

Subcase 2.2: $\beta \le 0$.
For each $i\geq 3$, $l_i=a_i x+b_i y+c_i$ which 
has $\nu$-value $\alpha$ if $a_i\neq 0$ and $\nu$-value $\ge \min(\beta,0)=\beta$
if $a_i=0$. Thus $\nu(l_i) = \beta$ when $a_i=0$, by definition of $\beta$.

Suppose there are $n'$-many
$l_i$ with $\nu(l_i)=\alpha$ and $m'$-many $l_i$ with $\nu(l_i)=\beta$.
By hypothesis \eqref{E6.2.1}, $n'\geq 3$, and $m' \ge 1$ because $\nu(l_2) = \beta$. Now
$$n'\alpha + \beta \ge n'\alpha + m' \beta =\nu(\bff)=\nu(\{x,y\})\geq \alpha+\beta-1,$$
which implies that $(n'-1) \alpha\geq -1$, again yielding a contradiction.

 Therefore
$\alpha\geq 0$, which implies that $\bff$ is $1$-flabby.

The mentioned particular case is clear by verifying \eqref{E6.2.1}.

The general consequences follow from Proposition \ref{zzpro3.18}.
\end{proof}

For $\bff$ as in Lemma \ref{zzlem6.2}, the 
Poisson field $K\{\bff\}$ is not isomorphic 
to $\Kweyl$ or any $K_q$ or any $K_{1,n,0}$ 
with $n \ge 2$, as we see from 
Lemmas \ref{zzlem4.3}--\ref{zzlem4.5}. Also, 
$K\{\bff\}$ is not isomorphic to 
$K\{p(x)xy\}$ for any $p(x) \in \Bbbk[x]$ 
with $\deg p(x) \ge 2$, in view of 
Proposition \ref{zzpro5.2}. 

The last statement of Lemma \ref{zzlem6.2} 
may fail if $\bff$ is not homogeneous.  
For example, take 
$\bff = (x+3)(x+4)\cdots(x+n)xy$ for some 
$n \ge 4$.  Then $\bff$ has distinct 
linear forms, but 
${^1\Gamma}_{0}(K\{\bff\}) = \Bbbk[x]$ 
by Proposition \ref{zzpro5.2}.

\begin{proposition}
\label{zzpro6.3} 
Let $\bff$ and $\bfg$ be two nonzero polynomials in $\Bbbk[x,y]$.
Suppose that 
\begin{enumerate}
\item[(a)]
$\phi$ is a Poisson morphism from $K\{\bff\}\to K\{\bfg\}$, and
\item[(b)]
$\bff$ is a product of distinct linear forms and satisfying
\eqref{E6.2.1}.
\end{enumerate}
Then the following hold.
\begin{enumerate}
\item[(1)]
$\phi(\bff) \in \Bbbk[x,y] \bfg$ and $\deg \bff\leq \deg \bfg$.  
\item[(2)]
If $\deg \bff=\deg \bfg$, then  $\phi$ preserves the degree.
Consequently, $\phi$ is a $\Bbbk$-algebra isomorphism sending 
$x\mapsto c_{11}x+c_{12}y+c_1$ and $y\mapsto c_{21}x+c_{22}y+c_2$ 
for some $c_1,c_2\in \Bbbk$ and some $\begin{pmatrix} c_{11}&c_{12}\\
c_{21}&c_{22}\end{pmatrix} \in GL_2(\Bbbk)$. Further, $\phi(\bff) = (c_{11}c_{22} - c_{12}c_{21}) \bfg$.
\item[(3)] 
If $\bff$ is homogeneous and $\deg \bff=\deg \bfg$, then  $\phi(x)$ and $\phi(y)$ are homogeneous of degree 1 up to an affine change of variables. In
this setting, $\bfg$ is also homogeneous.
\item[(4)]{\rm{(Dixmier Problem)}}
$K\{\bff\}$ has the Dixmier property.
\end{enumerate}
\end{proposition}

\begin{proof} 
(1) By Lemma \ref{zzlem6.2}, ${^1\Gamma}_{0}(K\{\bff\})=\Bbbk[x,y]$. By Lemma 
\ref{zzlem3.7}(1), $\phi$ maps ${^1\Gamma}_{0}(K\{\bff\})$ to 
${^1\Gamma}_{0}(K\{\bfg\})$ injectively. By Corollary \ref{zzcor3.9}, 
${^1\Gamma}_{0}(K\{\bfg\})\subseteq \Bbbk[x,y]$. So $\phi$ restricts to an 
injective algebra endomorphism of $\Bbbk[x,y]$. 

Let $m=\deg \bff$ and write $\bff=\prod_{i=1}^{m} l_i$ where
$l_i=a_i x+b_i y+c_i\not\in \Bbbk$. 
Let $s=\min\{ \deg\phi(l_i) \mid i\in [1,m]\}$ and
$t=\max\{ \deg \phi(l_i) \mid i\in [1,m]\}$. Without
loss of generality, we may assume that $s=\deg 
\phi(l_1)$ and $t=\deg \phi(l_2)$. If $s<t$, then 
$l_1$ and $l_2$ generate $\Bbbk[x,y]$. (For if not, $l_2 = \alpha l_1 + \beta$ for some $\alpha,\beta \in \Bbbk$, which would imply $\deg \phi(l_2) = \deg \phi(l_1)$.)  If $s=t$, we
can choose an $i \ne 1,2$ such that $l_1$ and $l_i$ 
generate $\Bbbk[x,y]$. So we may also assume that 
$l_1$ and $l_2$ generate $\Bbbk[x,y]$. In both cases, 
after an affine change of basis $x\to l_1$, $y\to l_2$,
we may assume that $l_1=x$ and $l_2=y$. Or equivalently, 
$s= \deg \phi(x)$ and $t= \deg \phi(y)$. 

Let $\{-,-\}_1$ (resp., $\{-,-\}_2$) be the Poisson bracket on 
$K\{\bff\}$ (resp., $K\{\bfg\}$). Then
\begin{equation}
\label{E6.3.1}\tag{E6.3.1}
\begin{aligned}
\phi(x)\phi(y) 
\prod_{i=3}^{m} (a_i \phi(x)+ b_i \phi(y)+c_i) &= \phi(\bff) = \phi(\{x,y\}_1) = \{\phi(x),\phi(y)\}_2  \\
&= (\phi(x)_x \phi(y)_y-\phi(x)_y \phi(y)_x)\{x,y\}_2  \\
&=(\phi(x)_x \phi(y)_y-\phi(x)_y \phi(y)_x)\bfg \,,
\end{aligned}
\end{equation}
which implies $\phi(\bff) \in \Bbbk[x,y] \bfg$.

The last term of \eqref{E6.3.1} has degree at most 
$$(s+t-2)+\deg \bfg.$$
The degree of the first term of \eqref{E6.3.1} is at least 
$(m-3)s+3 t$ since \eqref{E6.2.1} holds.  Namely, since $l_1 = x$ there are at least three $l_i$ with $b_i \ne 0$, and $\deg \phi(l_i) = t$ for these $l_i$ (by definition if $s=t$, 
and otherwise because 
$\deg \phi(x) < \deg \phi(y) = t$).
So we have
$(s+t-2)+\deg \bfg\geq (m-3) s+3t$.
This implies that  
\begin{equation}
\label{E6.3.2}\tag{E6.3.2}
\deg \bfg-2\geq (m-4)s+ 2 t\geq (m-2)s \ge m-2.
\end{equation}
Therefore $\deg \bfg\geq m=\deg \bff$. 

(2) Since $m= \deg \bff=\deg \bfg$, the inequalities \eqref{E6.3.2} 
imply that $s=t=1$. Thus, $\phi$ is a $\Bbbk$-algebra isomorphism 
and both $\phi(x)$ and $\phi(y)$ have degree 1. The described forms of $\phi(x)$ and $\phi(y)$ follow.  Further,
\begin{align*}
\phi(\bff) &= \phi(\{x,y\}_1) = \{ c_{11}x+c_{12}y+c_1,\, c_{21}x+c_{22}y+c_2 \}_2  \\
&= (c_{11}c_{22} - c_{12}c_{21}) \{x,y\}_2 = (c_{11}c_{22} - c_{12}c_{21}) \bfg \,.
\end{align*}

(3) By part (2), $\deg \phi(x)=\deg 
\phi(y)=1$, consequently,
$\phi(x)$ and $\phi(y)$ generate
$\Bbbk[x,y]$. Up to an affine change of variables in $K\{\bfg\}$,
we may assume that $\phi(x)=x$ and $\phi(y)=y$. The assertions 
follow.

(4) This is an immediate consequence of part (2).
\end{proof}

\begin{corollary} [Isomorphism Problem]  \label{zzcor6.4}
Let $\bff,\bfg \in \Bbbk[x,y]$ be two nonzero polynomials.
\begin{enumerate}
\item[(1)]
Suppose there are some $c_1,c_2 \in \Bbbk$ and some $\begin{pmatrix} c_{11}&c_{12}\\  c_{21}&c_{22}\end{pmatrix} \in GL_2(\Bbbk)$ such that
\begin{equation}  \label{E6.4.1}  \tag{E6.4.1}
\bfg = \frac{\bff(c_{11}x+c_{12}y+c_1\,,\, c_{21}x+c_{22}y+c_2)}{c_{11}c_{22} - c_{12}c_{21}} \,.
\end{equation}
Then $K\{\bff\} \cong K\{\bfg\}$.
\item[(2)]
Now assume that $\bff$ and $\bfg$ are products of distinct linear forms and both satisfy \eqref{E6.2.1}.  If $K\{\bff\} \cong K\{\bfg\}$, then there exist $c_1,c_2 \in \Bbbk$ and $\begin{pmatrix} c_{11}&c_{12}\\  c_{21}&c_{22}\end{pmatrix} \in GL_2(\Bbbk)$ such that \eqref{E6.4.1} holds.
\item[(3)] 
Assume that $\bff$ and $\bfg$ are each products of $4$ or more distinct homogeneous linear forms {\rm(}up to scalars{\rm)}.  Then $K\{\bff\} \cong K\{\bfg\}$ if and only if there is some $\begin{pmatrix} c_{11}&c_{12}\\  c_{21}&c_{22}\end{pmatrix} \in GL_2(\Bbbk)$ such that 
$$
\bfg = \frac{\bff(c_{11}x+c_{12}y\,,\, c_{21}x+c_{22}y)}{c_{11}c_{22} - c_{12}c_{21}} \,.
$$
\end{enumerate}
\end{corollary}

\begin{proof}
(1) Under these conditions, $\xhat := c_{11}x+c_{12}y+c_1$ and $\yhat := c_{21}x+c_{22}y+c_2$ generate the field $\Bbbk(x,y)$ over $\Bbbk$.  Since
$$
\{ \xhat,\yhat \}_{K\{\bfg\}} = (c_{11}c_{22} - c_{12}c_{21}) \{x,y\}_{K\{\bfg\}} =  (c_{11}c_{22} - c_{12}c_{21}) \bfg = \bff(\xhat,\yhat),
$$
there is a Poisson isomorphism $K\{\bff\} \rightarrow K\{\bfg\}$ sending $x \mapsto \xhat$ and $y \mapsto \yhat$.

(2) This follows from part (1) and Proposition \ref{zzpro6.3}(1)(2).

(3) Under the given conditions, $\bff$ and $\bfg$ both satisfy \eqref{E6.2.1}, and one implication is given by part (1).  If $K\{\bff\} \cong K\{\bfg\}$, there exist $c_i$ and $c_{ij}$ as in (1) such that \eqref{E6.4.1} holds.  Defining $\xhat$ and $\yhat$ as above, we have that $\bff(\xhat,\yhat)$ is homogeneous with respect to $x$, $y$.  It follows that if $l_1,\dots,l_d$ are the linear factors of $\bff$, then the $l_i(\xhat,\yhat)$ are homogeneous with respect to $x$, $y$.  Now $l_k = a_kx+b_ky$ for some $a_k,b_k \in \Bbbk$, and the homogeneity of $l_i(\xhat,\yhat)$ implies that $a_kc_1+b_kc_2=0$.  But $d\ge4$ and $l_2$ is not a scalar multiple of $l_1$, so $c_1=c_2=0$.  Therefore \eqref{E6.4.1} gives the required expression for $\bfg$.
\end{proof}

\begin{remark}  \label{zzrem6.5} 
Corollary \ref{zzcor6.4}(1) applies in particular when $\bff,\bfg \in \Bbbk[x,y]$ are homogeneous polynomials with degree different from $2$ such that $\bfg \in \kx \bff$.  Say $\bff$ and $\bfg$ have degree $n \ne 2$ and $\bfg = c \bff$ for some $c \in \kx$.  Choose $a \in \kx$ with $a^{n-2} = c$, and observe that $\frac{\bff(ax,ay)}{a^2} = a^{n-2} \bff(x,y) = \bfg$.

On the other hand, there are many Poisson isomorphisms $K\{\bff\} \cong K\{\bfg\}$ where $\bff$ and $\bfg$ are homogeneous and products of distinct linear forms but $\bfg \notin \kx \bff$.  For instance, take
$$
\bff := x y (x+y) (x+2y) \qquad \text{and} \qquad \bfg := y (x+y) (x+2y) (x+3y).
$$
Then \eqref{E6.4.1} holds with $c_1=c_2=0$ and $\bigl( \begin{smallmatrix} c_{11}&c_{12}\\ c_{21}&c_{22} \end{smallmatrix} \bigr) = \bigl( \begin{smallmatrix} 1&1\\ 0&1 \end{smallmatrix} \bigr)$.
\end{remark}

Now we can compute the automorphism group of $K\{\bff\}$
in the situation of Proposition \ref{zzpro6.3}.

\begin{theorem}[Automorphism Problem]
\label{zzthm6.6}  
Let $\bff\in \Bbbk[x,y]$ be a polynomial
of degree $n \ge 4$ with distinct linear forms and satisfying \eqref{E6.2.1}. 
\begin{enumerate}
\item[(1)]
There is an exact sequence of groups
$$1\to C\to \Aut_{Poi}(K\{\bff\})\to {\mathbb S}_n$$
where $C$ is a subgroup of $C_{n-2}$.
As a consequence, $|{\Aut_{Poi}}(K\{\bff\})|\leq (n-2) n!$.
\item[(2)]
If $f$ is homogeneous, then $C = C_{n-2}$ and $n-2 \le |{\Aut_{Poi}}(K\{\bff\})|$.
\end{enumerate}
\end{theorem}

\begin{proof} Write $\bff=\prod_{i=1}^n l_i$ where
the $l_i$ are the different linear forms 
of $\bff$.

(1) Let $\phi$ be an automorphism of $K\{\bff\}$. By Proposition 
\ref{zzpro6.3}(2), $\phi(x)$ and $\phi(y)$ are linear forms, and $\phi(\bff)=c\bff$ for some nonzero scalar 
$c$.  The $\phi(l_i)$ are then also linear 
forms of $\bff$, so the $\phi(l_i)= \gamma_i l_{\sigma(i)}$ for some 
$\sigma\in {\mathbb S}_n$ and some 
$\gamma_i\in \kx$. Then the map $\Phi:
\phi\to \sigma$ defines a group homomorphism  
$\Aut_{Poi}(K\{\bff\})\to {\mathbb S}_n$. It remains to show that
the kernel of $\Phi$ is (isomorphic to) a subgroup of $C' := \{ a \in \kx \mid a^{n-2} = 1 \} \cong C_{n-2}$.

Let $\phi\in \ker \Phi$. Then $\phi(l_i)= \gamma_i l_i$
for all $i$. Up to a change of variables,
we may assume that $l_1=x$ and $l_2=y$. Then
$\phi(x)= \gamma_1 x$ and $\phi(y)= \gamma_2 y$. By \eqref{E6.2.1}, there is an $i \ge 3$ such that $l_i = a_i x + b_i y + c_i$ with $a_i,b_i \in \kx$, $c_i \in \Bbbk$, and
$\phi(l_i)=\gamma_i l_i$ implies that $\gamma_{i}=\gamma_1=\gamma_2$.  Thus $\phi(x)=ax$ and $\phi(y)=ay$ for 
some $a\in \Bbbk^{\times}$.
Then 
$$
a^n \bff = \prod_{i=1}^n (al_i) = \phi(\bff) = \phi(\{x,y\}) = \{ax,ay\} = a^2 \bff,
$$
so we obtain
that $a^{n-2}=1$ and $a \in C'$.

(2) For any $a \in C'$, we have $\bff(ax,ay) = a^n \bff(x,y) = \{ax,ay\}$, so there is a Poisson
automorphism of $K\{\bff\}$ sending $x \mapsto ax$ and $y \mapsto ay$. The assertion follows.
\end{proof}

%%%%%%%%%%%%%%%%%%%%%%%%%

\section{Family (4): $K\{cf(x)g(y)\}$ where $c\in \Bbbk^{\times}$, $f \in \Bbbk[x]$, $g \in \Bbbk[y]$}
\label{zzsec7}

We concentrate on the special case when $f$ and $g$ are monic, split over $\Bbbk$, and squarefree, i.e.,
$f=\prod_{i=1}^m (x+\xi_i)$ and $g=\prod_{j=1}^{n} (y+\chi_j)$
where $(\xi_i)_{i=1}^{m}$ and $(\chi_j)_{j=1}^n$ 
are families of distinct scalars. 

If $m=0$ or $n=0$, then $K \cong \Kweyl$ by Proposition \ref{zzpro1.1}(1).  If $m\ge1$ and $n=1$, or if $m=1$ and $n\ge1$, then $K \cong K\{p(x)xy\}$ for some $p \in \Bbbk[x]$, placing $K$ in our second family (Section \ref{zzsec5}). When $n\geq m=2$ or $m\geq n=2$, we refer to the 
following easy lemma.

\begin{lemma}
\label{zzlem7.1}
Let $\bff= c f(x)g(y)$ with $\deg f(x)\geq \deg g(y)=2$.
\begin{enumerate}
\item[(1)]
If $g(y)=(y+a)^2$, then $K\{\bff\}\cong \Kweyl$.
\item[(2)]
If $g(y)=(y+a)(y+b)$ with $a\neq b$, then $K\{\bff\}\cong 
K\{ p(x)xy\}$ for some $p(x)\in \Bbbk[x]$.
\end{enumerate}
Similar statements hold when $x$ and $y$ are switched.
\end{lemma}

\begin{proof}
(1) Up to a change of variables, we may assume that $a=0$, so
$$\{x, y^{-1}\}
=-y^{-2} \{x,y\}=-y^{-2} cf(x) y^2
=-cf(x).$$
The assertion follows by Proposition \ref{zzpro1.1}(1).

(2) Up to a change of variables,
we may assume that $f(x)=x f_1(x)$ and $a=0$ (and hence $b\neq 0$). Let $y_1=y^{-1}$; then 
$$\{x, y_1\}=-y^{-2}\{x,y\}
=-y^{-2} c f(x) y (y+b)=
-c f(x) (1+b y_1).$$
Let $y_2= 1+b y_1$; then 
$$\{x, y_2\}=b\{x,y_1\}=-bcf(x) y_2=(-bc f_1(x)) xy_2.$$
The assertion follows by setting $p(x)=-bcf_1(x)$.
\end{proof}
Throughout the rest of 
this section we assume that $m\geq 3$ and $n\geq 3$.

\begin{lemma}
\label{zzlem7.2} 
Suppose $\bff$ is of the form 
$c\prod_{i=1}^m (x+\xi_i) \prod_{j=1}^{n} (y+\chi_j)$
where $c \in \kx$ and $(\xi_i)_{i=1}^{m}$, $(\chi_j)_{j=1}^n$ 
are families of distinct scalars in $\Bbbk$.  If $m,n\ge3$, then $\bff$ is $1$-flabby, 
namely, ${^1\Gamma}_{0}(K\{\bff\})=\Bbbk[x,y]$.

As a consequence, $\vht_1(K\{\bff\})= \fht(K\{\bff\})=\deg \bff$ and $K\{\bff\}$ is height 
cohereditary.
\end{lemma}

\begin{proof} It is easy to see that \eqref{E6.2.1}
holds. The assertions follow from Lemma \ref{zzlem6.2}.
\end{proof}

For $\bff$ as in Lemma \ref{zzlem7.2}, the 
Poisson field $K\{\bff\}$ is not isomorphic 
to $\Kweyl$ or any $K_q$ or any $K_{1,n,0}$ 
with $n\ge2$, as we see from Lemmas 
\ref{zzlem4.3}--\ref{zzlem4.5}. Also, 
$K\{\bff\}$ is not isomorphic to $K\{p(x)xy\}$ 
for any $p(x) \in \Bbbk[x]$ with 
$\deg p(x) \ge 2$, in view of 
Proposition \ref{zzpro5.2}.

\begin{proposition}
\label{zzpro7.3}
Let $\bff$ and $\bfg$ be two nonzero polynomials in $\Bbbk[x,y]$. Suppose that
\begin{enumerate}
\item[(a)]
$\phi$ is a Poisson morphism from $K\{\bff\}$ to $K\{\bfg\}$, and
\item[(b)]
$\bff$ is of the form 
$c\prod_{i=1}^m (x+\xi_i) \prod_{j=1}^{n} (y+\chi_j)$
where $c \in \kx$ and $(\xi_i)_{i=1}^{m}$, $(\chi_j)_{j=1}^n$ 
are families of distinct scalars in $\Bbbk$, with $m,n\ge3$.
\end{enumerate}
Then the following hold.
\begin{enumerate}
\item[(1)]
$\phi(\bff) \in \Bbbk[x,y] \bfg$ and $\deg \bff\leq \deg \bfg$.  
\item[(2)]
If $\deg \bff=\deg \bfg$, then  $\phi$ preserves the degree.
Consequently, $\phi$ is a $\Bbbk$-algebra isomorphism sending 
$x\mapsto c_{11}x+c_{12}y+c_1$ and $y\mapsto c_{21}x+c_{22}y+c_2$ 
for some $c_1,c_2\in \Bbbk$ and some $\begin{pmatrix} c_{11}&c_{12}\\
c_{21}&c_{22}\end{pmatrix} \in GL_2(\Bbbk)$. Further, $\phi(\bff) = (c_{11}c_{22} - c_{12}c_{21}) \bfg$.
\item[(3)]{\rm{(Dixmier Problem)}}
$K\{\bff\}$ has the Dixmier property.
\item[(4)]
$\Aut_{Poi}(K\{\bff\})$ is isomorphic to a subgroup of $\{\phi\in \Aut(\Bbbk[x,y])
\mid \phi(\bff) \in \kx \bff\}.$
\end{enumerate}
\end{proposition}

\begin{proof} 
(1)(2)(3) These assertions follow from Proposition \ref{zzpro6.3}. 

(4) The assertion follows from part (2). 
\end{proof}

The Isomorphism Problem in the current setting follows from Proposition \ref{zzpro7.3} just as Corollary \ref{zzcor6.4} follows from Proposition \ref{zzpro6.3}.  One implication is part of Corollary \ref{zzcor6.4}, so we only state the other.

\begin{corollary}[Isomorphism Problem]  \label{zzcor7.4} 
Let $\bff = \gamma_1f_1(x)g_1(y)$ and $\bfg = \gamma_2 f_2(x)g_2(y)$ where $\gamma_1,\gamma_2 \in \kx$, $f_1,f_2 \in \Bbbk[x]$, and $g_1,g_2 \in \Bbbk[y]$.  Assume also that $f_1,f_2,g_1,g_2$ are monic, squarefree, split over $\Bbbk$, and have degrees $\ge3$.
 If $K\{\bff\} \cong K\{\bfg\}$, then there exist $c_1,c_2 \in \Bbbk$ and $\begin{pmatrix} c_{11}&c_{12}\\  c_{21}&c_{22}\end{pmatrix} \in GL_2(\Bbbk)$ such that \eqref{E6.4.1} holds.
\end{corollary}

\begin{theorem}[Automorphism Problem]  \label{zzthm7.5} 
Let  $\bff = c\prod_{i=1}^m (x+\xi_i) \prod_{j=1}^{n} (y+\chi_j)$
where $c \in \kx$ and $(\xi_i)_{i=1}^{m}$, $(\chi_j)_{j=1}^n$ 
are families of distinct scalars in $\Bbbk$ and $m,n \ge 3$.  Define
\begin{itemize}
\item 
$G := \{ \phi \in \Aut_{Poi}(K\{\bff\}) \mid \phi(x) \in \Bbbk[x],\; \phi(y) \in \Bbbk[y] \}$;
\item 
$m^* := m-1$ if there is some $s \in [1,m]$ such that $m \xi_s = \sum_{i=1}^m \xi_i$, and $m^* := m$ otherwise;
\item 
$n^* := n-1$ if there is some $s \in [1,n]$ such that $n \chi_s = \sum_{j=1}^n \chi_j$, and $n^* := n$ otherwise;
\item 
$H := \{ (\gamma,\delta) \in (\kx)^2 \mid \gamma^{m^*} = \delta^{n^*} = \gamma^{m-1} \delta^{n-1} = 1 \}$.
\end{itemize}
Then the following hold.
\begin{enumerate}
\item[(1)] 
$G$ is a normal subgroup of $\Aut_{Poi}(K\{\bff\})$ with index at most $2$.
\item[(2)]
$G$ is isomorphic to a subgroup of $H$.
\item[(3)]
$\Aut_{Poi}(K\{\bff\}) = G$ except possibly when $m=n$ is even and $m^* = n^* = m$, in which case $G$ is cyclic of order at most $m$.
\item[(4)]
$|{\Aut_{Poi}}(K\{\bff\})| \le (m-1)(n-1)$.
\end{enumerate}
\end{theorem}

\begin{proof}
(1) We first claim that any $\phi \in \Aut_{Poi}(K\{\bff\})$ is either in $G$ or satisfies $\phi(x) \in \Bbbk[y]$, $\phi(y) \in \Bbbk[x]$.  By Proposition \ref{zzpro7.3}(2), we have $\phi(x) = c_{11}x+c_{12}y+c_1$ and $\phi(y) = c_{21}x+c_{22}y+c_2$ 
for some $c_1,c_2\in \Bbbk$ and some $\bigl( \begin{smallmatrix} c_{11}&c_{12}\\
c_{21}&c_{22}\end{smallmatrix} \bigr) \in GL_2(\Bbbk)$, and $\phi(\bff) \in \kx \bff$.  Thus,
\begin{equation}  \label{E7.5.1}  \tag{E7.5.1}
\phi(\bff) = c \prod_{i=1}^m (c_{11}x+c_{12}y+c_1 +\xi_i) \prod_{j=1}^n (c_{21}x+c_{22}y+c_2 + \chi_j).
\end{equation}
Since $\phi(\bff) \in \kx \bff$, the factor $c_{11}x+c_{12}y+c_1 +\xi_1$ must be a scalar multiple of some $x+\xi_i$ or of some $y+\chi_j$, and likewise for $c_{21}x+c_{22}y+c_2 + \chi_1$.  Note that if $\bigl( d_{ij} \bigr) = \bigl( c_{ij} \bigr)^{-1}$, then
\begin{align*}
d_{11}(c_{11}x+c_{12}y+c_1 +\xi_1) + d_{12}(c_{21}x+c_{22}y+c_2 + \chi_1) &= x + \text{a scalar}  \\
d_{21}(c_{11}x+c_{12}y+c_1 +\xi_1) + d_{22}(c_{21}x+c_{22}y+c_2 + \chi_1) &= y + \text{a scalar},
\end{align*}
so $c_{11}x+c_{12}y+c_1 +\xi_1$ and $c_{21}x+c_{22}y+c_2 + \chi_1$ cannot both be in $\Bbbk[x]$ or both in $\Bbbk[y]$.  Thus, either $c_{12} = c_{21} = 0$, in which case $\phi \in G$, or else $c_{11} = c_{22} = 0$, whence $\phi(x) \in \Bbbk[y]$, $\phi(y) \in \Bbbk[x]$.  This verifies the claim.

Observe from \eqref{E7.5.1} that we can only have $c_{11} = c_{22} = 0$ when $m=n$.  Thus, $\Aut_{Poi}(K\{\bff\}) = G$ when $m \ne n$.

It is clear that $G$ is a subgroup of $\Aut_{Poi}(K\{\bff\})$, and it follows from the claim that $G$ is normal.  Suppose there exists $\psi \in \Aut_{Poi}(K\{\bff\}) \setminus G$.  For any $\phi \in \Aut_{Poi}(K\{\bff\}) \setminus G$, we see from the claim that $\phi\psi^{-1} \in G$, whence $\phi G = \psi G$.  Consequently, 
$$
|{\Aut_{Poi}}(K\{\bff\}) / G | \le 2.
$$

(2) For any $\phi \in G$, we have, by definition,
$$
\phi(x) = c_{11}(\phi) x + c_1(\phi) \qquad \text{and} \qquad \phi(y) = c_{22}(\phi) y + c_2(\phi)
$$
for some $c_{11}(\phi), c_{22}(\phi) \in \kx$ and $c_1(\phi), c_2(\phi) \in \Bbbk$.  The rule $\phi \mapsto (c_{11}(\phi), c_{22}(\phi))$ defines a homomorphism $\Phi : G \rightarrow (\kx)^2$.  We need to show that $\Phi$ is injective and $\Phi(G) \subseteq H$.

Fix $\phi \in G$, and abbreviate $c_{ii} := c_{ii}(\phi)$, $c_i := c_i(\phi)$.  In view of Proposition \ref{zzpro7.3}(2), \eqref{E7.5.1} reduces to
\begin{equation}  \label{E7.5.2}  \tag{E7.5.2}
c_{11} c_{22} \bff = \phi(\bff) = c \prod_{i=1}^m (c_{11}x+c_1 +\xi_i) \prod_{j=1}^n (c_{22}y+c_2 + \chi_j).
\end{equation}
Comparing leading coefficients in this equation, we obtain $c_{11} c_{22} c = c c_{11}^m c_{22}^n$, and so $c_{11}^{m-1} c_{22}^{n-1} = 1$.  We will be done on showing that
\begin{enumerate}
\item[(a)]
$c_{11}^{m^*} = 1$, and $c_{11} = 1 \implies c_1 = 0$;
\item[(b)]
$c_{22}^{n^*} = 1$, and $c_{22} = 1 \implies c_2 = 0$.
\end{enumerate}
Since (a) and (b) are symmetric, we just prove (a).

Comparing factors from $\Bbbk[x]$ in \eqref{E7.5.2}, we see that there must be a permutation $\sigma \in \SB_m$ such that $c_{11}x + c_1 + \xi_i = c_{11}(x + \xi_{\sigma(i)})$ for all $i$, that is,
\begin{equation}  \label{E7.5.3}  \tag{E7.5.3}
c_{11} \xi_{\sigma(i)} = c_1 + \xi_i \qquad \forall\; i \in [1,m].
\end{equation}
Summing this equation over $i \in [1,m]$ yields 
\begin{equation}  \label{E7.5.4}  \tag{E7.5.4}
(c_{11}-1) \sum_{i=1}^m \xi_i = m c_1 \,.
\end{equation}
In particular, the implication $c_{11}  = 1 \implies c_1 = 0$ follows.

Since $1^{m^*} = 1$, we assume $c_{11} \ne 1$ for the remainder of the proof of (a).

\textbf{Claim 1:} $s \in [1,m]$ is a fixed point of $\sigma$ if and only if $\xi_s = c_1/(c_{11}-1)$, if and only if $m \xi_s = \sum_{i=1}^m \xi_i$.

The first forward implication follows from \eqref{E7.5.3}, and the second equivalence is immediate from \eqref{E7.5.4}.  If $\xi_s = c_1/(c_{11}-1)$, then $c_{11} \xi_s = c_1 + \xi_s = c_{11} \xi_{\sigma(s)}$, and so $\sigma(s) = s$ because the $\xi_i$ are distinct.

\textbf{Claim 2:} $\xi_{\sigma^t(i)} = (c_{11}^{-1} + \cdots + c_{11}^{-t}) c_1 + c_{11}^{-t} \xi_i$ for all $i \in [1,m]$ and $t \ge 0$.

This is an obvious induction on $t$.

\textbf{Claim 3:} Let $\calO$ be a $\sigma$-orbit of cardinality $t>1$.  Then $c_{11}$ is a primitive $t$-th root of unity, i.e., $t$ is the order of $c_{11}$ in the group $\kx$.

If $i \in \calO$, then $\sigma^t(i) = i$, and by Claim 2 we see that $(c_{11}^t - 1) \xi_i = \frac{c_{11}^t-1}{c_{11}-1}\, c_1$.  Since there are at least $2$ indices in $\calO$, we must have $c_{11}^t = 1$.  Thus, $c_{11}$ has order $d<\infty$ in $\kx$, and $d \mid t$.  Now $c_{11}^d \xi_i = (c_{11}^{d-1} + \cdots + 1) c_1 + \xi_i$ for any $i \in \calO$.  Claim 2 shows that $\xi_{\sigma^d(i)} = \xi_i$ and hence $\sigma^d(i) = i$.  This forces $t=d$, proving Claim 3.

We can now complete the proof of (a) and thus of (2).  If $\sigma$ has a fixed point $s$, then by Claim 1, $s$ is the only fixed point of $\sigma$, and $m^* = m-1$.  By Claim 3, all the non-singleton orbits of $\sigma$ have the same cardinality $t$, and $c_{11}^t = 1$.  Since $[1.m] \setminus \{s\}$ is a disjoint union of these latter orbits, $t \mid m-1$, and thus $c_{11}^{m^*} = 1$.

If $\sigma$ has no fixed points, then $m^* = m$ and $[1,m]$ is a disjoint union of $\sigma$-orbits of cardinality $t$.  In this case, $c_{11}^t = 1$ and $t \mid m$, so again $c_{11}^{m^*} = 1$.

(3) That $\Aut_{Poi}(K\{\bff\}) = G$ when $m \ne n$ was already noted in the proof of part (1).  Suppose there exists $\theta \in \Aut_{Poi}(K\{\bff\}) \setminus G$, so that $m=n$.  Now $\theta(x) = c_{12}y+c_1$ and $\theta(y) = c_{21}x+c_2$ for some $c_{12},c_{21} \in \kx$ and $c_1,c_2 \in \Bbbk$.  Observe that $\theta^2 \in G$ and
$$
\theta^2(x) = c_{12}c_{21} x + c_{12}c_2+c_1 \,, \qquad\qquad \theta^2(y) = c_{12}c_{21} y + c_{21}c_1+c_2 \,.
$$
Then $(c_{12}c_{21}, c_{12}c_{21}) \in H$, whence
$$
(c_{12}c_{21})^{m^*} = (c_{12}c_{21})^{n^*} = (c_{12}c_{21})^{2(m-1)} = 1.
$$

In the present case, \eqref{E7.5.1} reads
\begin{align*}
\theta(\bff) &= c \prod_{i=1}^m (c_{12}y+c_1+\xi_i) \prod_{j=1}^m (c_{21}x+c_2+\chi_j)  \\
&= (c_{12}c_{21})^m c \prod_{j=1}^m (x+c_{21}^{-1}(c_2+\chi_j)) \prod_{i=1}^m (y+c_{12}^{-1}(c_1+\xi_i)) \,,
\end{align*}
from which we conclude that $\theta(\bff) = (c_{12}c_{21})^m \bff$.  On the other hand, Proposition \ref{zzpro7.3}(2) says that $\theta(\bff) = - c_{12}c_{21} \bff$.  Consequently, $(c_{12}c_{21})^{m-1} = -1$.  In particular, $m^*,n^* \ne m-1$, and so $m^*=n^* = m$.  Thus $c_{12}c_{21} = -1$ and $m$ is even.

We now get $H = \{ (\gamma,\gamma^{-1}) \in (\kx)^2 \mid \gamma^m = 1 \}$, which is cyclic of order at most $m$, and then the same holds for $G$ by part (2).

(4)  We first show that $|H| \le (m-1)(n-1)$, which is clear in case $m^* = m-1$ and $n^* = n-1$.
If $m^* = m$, then all $(\gamma,\delta) \in H$ satisfy $\gamma = \delta^{n-1}$, whence $|H| \le n \le (m-1)(n-1)$.  Similarly, $|H| \le m \le (m-1)(n-1)$ in case $n^* = n$.  Therefore, we have $|{\Aut_{Poi}}(K\{\bff\})| \le (m-1)(n-1)$ whenever $\Aut_{Poi}(K\{\bff\}) = G$.

Finally, suppose that $\Aut_{Poi}(K\{\bff\}) \ne G$, whence $G$ has index $2$ and $|{\Aut_{Poi}}(K\{\bff\})| = 2 |G|$.  By part (3), $m=n$ is even and $|G| \le m$.  Since then $m \ge 4$, we conclude that $|{\Aut_{Poi}}(K\{\bff\})| \le 2m < (m-1)^2$ in this case.
\end{proof}

\begin{example}  \label{zzexa7.6}  
For some examples where $\Aut_{Poi}(K\{\bff\})$ is as large as possible, let $m,n \ge 3$ and assume that $\Bbbk$ contains a primitive $(m-1)$-st root of unity, say $\gamma$, and a primitive $(n-1)$-st root of unity, say $\delta$.  Set $\bff := x y \prod_{i=1}^{m-1} (x - \gamma^i) \prod_{j=1}^{n-1} (y - \delta^j)$ and $K := K\{\bff\}$.  

There is a field automorphism $\phi$ of $K$ sending $x \mapsto \gamma^{-1} x$ and $y \mapsto y$.  Since
$$
\{ \gamma^{-1} x, y \} = \gamma^{-1} \bff = \gamma^{-m} \bff = \bff(\gamma^{-1} x, y),
$$
we see that $\phi \in \Aut_{Poi}(K)$.  Similarly, there exists $\psi \in \Aut_{Poi}(K)$ sending $x \mapsto x$ and $y \mapsto \delta^{-1} y$.  In the notation of Theorem \ref{zzthm7.5}, $\phi,\psi \in G$.  Note that $\phi^s \psi^t = \Id$ if and only if $m-1 \mid s$ and $n-1 \mid t$.  Thus $G$ has order at least $(m-1)(n-1)$.  Theorem \ref{zzthm7.5}(4) therefore implies that $\Aut_{Poi}(K)$ has order $(m-1)(n-1)$, the maximum possible.
\end{example}

One next goal of this section is to show that 
the automorphism group of $K\{\bff\}$  is
trivial for some $\bff$ as in Theorem \ref{zzthm7.5}.  This of course holds if $\Bbbk$ has no nontrivial $m^*$-th or $n^*$-th roots of unity, but it also holds for suitable $\bff$ when $\Bbbk$ is algebraically closed.  We go to a situation somewhat parallel to that of Corollary \ref{zzcor5.7}(3).

\begin{corollary}  \label{zzcor7.7} 
Let $K = K\{\bff\}$ under the following hypotheses:
\begin{enumerate}
\item[(a)]
 $\bff = c\prod_{i=1}^m (x+\xi_i) \prod_{j=1}^{n} (y+\chi_j)$
where $c \in \kx$ and $(\xi_i)_{i=1}^{m}$, $(\chi_j)_{j=1}^n$ 
are families of distinct scalars in $\Bbbk$, with $m,n\ge3$.  
\item[(b)]
$m \ne n$, or $m=n$ is odd.
\item[(c)]
The families $(\xi_i-\xi_1)_{i=2}^m$ and $(\chi_j-\chi_1)_{j=2}^n$ are $\ZZ$-linearly independent.
\item[(d)]
$(\xi_i-\xi_1)^m \ne (\xi_1-\xi_j)^m$ for all $i\ne j$, and $(\chi_i-\chi_1)^n \ne (\chi_1-\chi_j)^n$ for all $i \ne j$.
\end{enumerate}

Then $\Aut_{Poi}(K)$ is trivial.
\end{corollary}

\begin{proof}
Set $x' := x+\xi_1$ and $y' := y+\chi_1$.  Then $K = \Bbbk(x',y')$ with
$$
\{x',y'\} = \{x,y\} = \bff = c\prod_{i=1}^m (x'+\xi'_i) \prod_{j=1}^{n} (y'+\chi'_j)
$$
where $\xi'_i := \xi_i-\xi_1$ for all $i$ and $\chi'_j := \chi_j-\chi_1$ for all $j$.  Thus, there is no loss of generality in assuming that $\xi_1 = \chi_1 = 0$.

Because of hypothesis (c), we cannot have $m \xi_s = \sum_{i=1}^m \xi_i$ for any $s \in [1,m]$, and so $m^* = m$ in the notation of Theorem \ref{zzthm7.5}.  Similarly, $n^*=n$.

Towards a contradiction, suppose there is some nontrivial $\phi \in \Aut_{Poi}(K)$.  Then $\phi(x) \ne x$ or $\phi(y) \ne y$, so by symmetry we may assume that $\phi(x) \ne x$.  By Theorem \ref{zzthm7.5}(3), $\phi(x) = c_{11}x+c_1$ and $\phi(y) = c_{22}y+c_2$ for some $c_{11},c_{22} \in \kx$ and $c_1,c_2 \in \Bbbk$.  Then $c_1 \ne 0$ or $c_{11} \ne 1$.

By \eqref{E7.5.4}, $(c_{11}-1) \sum_{i=1}^m \xi_i = mc_1$.  Since $\sum_{i=1}^m \xi_i = \sum_{i=2}^m \xi_i \ne 0$ by hypothesis (c), $c_1=0$ would imply $c_{11}=1$, contradicting our current assumptions.  Thus, $c_1 \ne 0$, and then $c_{11} \ne 1$ as well.  From the proof of Theorem \ref{zzthm7.5}(2), there is some $\sigma \in \SB_m$ such that \eqref{E7.5.3} holds.  Due to the fact, noted above, that $m \xi_s \ne \sum_{i=1}^m \xi_i$ for $s \in [1,m]$, Claim 1 in the previous proof shows that $\sigma$ has no fixed points.

Setting $k := \sigma(1) \ne 1$, we obtain $c_{11} \xi_k = c_1$ from \eqref{E7.5.3}.  Setting $l := \sigma^{-1}(1)$, we obtain $0 = c_1 + \xi_l$.  From the proof of Theorem \ref{zzthm7.5}(2), $c_{11}^m=1$, so
$$
\xi_k^m = (c_{11} \xi_k)^m = c_1^m = (- \xi_l)^m \,.
$$
Hypothesis (d) now implies that $k=l$, whence $-c_{11} c_1 = c_{11} \xi_l = c_{11} \xi_k = c_1$, and thus $c_{11} = -1$.

Equation \eqref{E7.5.3} now reads $\xi_{\sigma(i)} = - c_1 - \xi_i$ for all $i$.  It follows that $\xi_{\sigma^2(i)} = \xi_i$ for all $i$, and consequently $\sigma^2 = \Id$.  Since $\sigma$ has no fixed points, $[1,m]$ is a disjoint union of $2$-point $\sigma$-orbits.  The assumption that $m\ge3$ now yields $m\ge4$, and so $\sigma$ has an orbit $\{j,\sigma(j)\}$ different from $\{1,k\}$.  Since $\xi_j + \xi_{\sigma(j)} = - c_1 = \xi_k$, we again contradict hypothesis (c).

Therefore $\Aut_{Poi}(K)$ contains no nontrivial automorphisms.
\end{proof}

%%%%%%%%%%%%%%%%%%%%%%%%%%%%%%%%

\section{The polynomial flag problem}
\label{zzsec8}

In this section we will construct Poisson fields $K\{\bff\}$ 
with infinite flag height, that is, $K\{\bff\} \not\cong K\{\bfg\}$ for any polynomial $\bfg \in \Bbbk[x,y]$.  Various bounds on rational functions are needed, which we define in general.

Let $F:=\Bbbk(x_1,\cdots,x_n)$ be the field of rational 
functions in $n$ variables. If $n=2$, we also use $(x,y)$ for 
$(x_1,x_2)$. For an $n$-tuple $\bfs = (s_1,\cdots,s_n)$ of 
elements of $F$, let $\Bbbk[\bfs]$ denote the $\Bbbk$-subalgebra 
of $F$ generated by $s_1,\cdots, s_n$. When this notation is used, 
we do not assume that $s_1,\cdots,s_n$ are algebraically 
independent. Let $Q(\Bbbk[\bfs])$ denote the fraction field of 
$\Bbbk[\bfs]$, realized as a subfield of $F$. Let
$$\Phi = \Phi(F) :=\{ \bfs \in F^n \mid Q(\Bbbk[\bfs])=F\}.$$
The condition $Q(\Bbbk[\bfs])=F$ forces that $s_1,\cdots,s_n$
are algebraically independent over $\Bbbk$ and that $\Bbbk[\bfs]
\cong \Bbbk[x_1,\cdots,x_n]$. When $\bfs \in \Phi$, 
we call $\bfs$ a {\it fractional frame} of $F$. Let
$$\overline{\Phi}:=
\{ \Bbbk[\bfs]\mid \bfs \in \Phi\}.$$
By definition, the map $\phi:\Phi\to \overline{\Phi}$ defined 
by $\phi(\bfs)=\Bbbk[\bfs]$ is onto, but not one-to-one. 

Next we will define some numerical invariants of an element 
$h\in F$. In the following discussion we fix $h\in F$. For each 
$\bfs \in \Phi$, write $h$ as a rational function in 
$s_1,\cdots,s_n$,
\begin{equation}
\label{E8.0.1}\tag{E8.0.1}
h=\frac{f(s_1,\cdots,s_n)}{g(s_1,\cdots,s_n)} \,,
\end{equation}
where $f:=f(s_1,\cdots,s_n)$ and $g:=g(s_1,\cdots,s_n)$ are in 
$\Bbbk[\bfs]$ (namely, these are polynomials in $s_1,\cdots,s_n$) 
and where $\gcd(f,g)=1$ in $\Bbbk[\bfs]$. (We may further assume that $g$ is monic if necessary). Note that the pair 
$(f,g)$ in \eqref{E8.0.1} is unique up to multiplication by scalars from 
$\Bbbk$. For a given $\bfs \in \Phi$ and given \eqref{E8.0.1}, we 
consider a prime decomposition of the polynomial $g$ in 
$\Bbbk[\bfs]$:
\begin{equation}
\label{E8.0.2}\tag{E8.0.2}
g=p_1^{m_1}\cdots p_w^{m_w},
\end{equation}
where the $p_j$ are pairwise non-associate primes in $\Bbbk[\bfs]$ and 
the $m_j \in \Zpos$. (So $w=0$ if $g\in \Bbbk^{\times}$.)  The prime 
decomposition is unique up to scalar multiples of the $p_j$ and 
a permutation of $p_1,\cdots,p_w$. Let $\deg_{\bfs}$ denote the 
total degree of rational functions in $Q(\Bbbk[\bfs])$. Hence, 
$\deg_{\bfs} s_i=1$ for all $i=1,\cdots,n$. 

\begin{definition}
\label{zzdef8.1} 
Let $h\in F$.
\begin{enumerate}
\item[(1)]
The {\it denominator degree bound} of $h$ is defined to be
$$\ddb(h) :=\min\{\deg_{\bfs} g\mid \bfs \in \Phi,\; 
{\text{$(f,g)$ as in \eqref{E8.0.1}}}\}.$$
\item[(2)]
The {\it denominator prime-divisor bound} of $h$ is defined 
to be
$$\dpb(h) :=\min\{ w \mid \bfs \in \Phi,\; g=p_1^{m_1}\cdots p_w^{m_w}\;
\text{as in \eqref{E8.0.2}}\}.$$
\item[(3)]
The {\it framed degree bound} of $h$ is defined to be
$$\fdb(h) :=\min\{\deg_{\bfs} h \mid \bfs \in \Phi,\; 
{\text{$(f,g)$ 
as in \eqref{E8.0.1}, $\deg_{\bfs} g=\ddb(h)$}}\}.$$
\item[(4)]
$h$ is called a {\it framed polynomial} if $h\in 
\Bbbk[\bfs]$ for some $\bfs\in \Phi$.
\end{enumerate}
\end{definition}

By definition, $0\leq \dpb(h)\leq \ddb(h) < \infty$ and the following
are equivalent:
\begin{enumerate}
\item[(i)]
$\dpb(h)=0$,
\item[(ii)]
$\ddb(h)=0$,
\item[(iii)]
$h$ is a framed polynomial. 
\end{enumerate}

We can calculate $\ddb$, $\dpb$, and $\fdb$ for certain 
elements in $F$.

\begin{lemma}
\label{zzlem8.2}
Let $u\in F\setminus \Bbbk$ and let 
$\{a_i\}_{i=1}^{w'}\sqcup \{b_j\}_{j=0}^w$ be a set of distinct 
scalars in $\Bbbk$, where $w$ and $w'$ are non-negative 
integers. Let 
$h=\frac{(u-a_1)(u-a_2)\cdots (u-a_{w'})}{(u-b_0)(u-b_1)\cdots (u-b_w)}$.
\begin{enumerate}
\item[(1)]
$\ddb(h)\geq \dpb(h)\geq w$.
\item[(2)]
Let $\varphi: F\to F$ be a $\Bbbk$-algebra map. Then 
$\ddb(\varphi(h))\geq \dpb(\varphi(h))\geq w$.
\item[(3)]
If $u=x_1+b_0$ and $w+1\geq w'$, then $\ddb(h)=\dpb(h)= w$.
\item[(4)]
If $u=x_1+b_0$ and $w+1\geq w'$, then $\fdb(h)=1$.
\end{enumerate}
\end{lemma}

\begin{proof}
Replacing $u$ by $u-b_0$ and $a_i$, $b_j$ by $a_i-b_0$, $b_j-b_0$, we may assume that $b_0=0$.  The first hypothesis in (3), (4) then becomes $u = x_1$.

(1) Pick any fractional frame $\bfs$ and write $u$ as a 
rational function $u=\frac{f}{g}$ where $f$ and $g$ are 
co-prime polynomials in $\Bbbk[\bfs]$. Then 
\begin{equation}  
\label{E8.2.1}\tag{E8.2.1}
\begin{aligned}
h&=\frac{(u-a_1)(u-a_2)\cdots (u-a_{w'})}{u(u-b_1)\cdots (u-b_w)}\\
 &=g^{w-w'+1}\frac{(f-a_1 g)(f-a_2 g)
 \cdots (f-a_{w'}g)}{f(f-b_1 g)\cdots (f-b_w g)}
\end{aligned}
\end{equation}
and the polynomials in $G:=\{f-a_i g\}_{i=1}^{w'}\sqcup 
\{f-b_j g\}_{j=0}^{w}\sqcup\{g\}$ are pairwise coprime. 
Since $u$ is not in $\Bbbk$, at most one of $f$, $g$ is a scalar, so at most one element of $G$
is a scalar. Thus the denominator of $h$, which is either
$$
f(f-b_1 g)\cdots (f-b_w g) \qquad\text{or} 
\qquad g^{w'-w-1}f(f-b_1 g)\cdots (f-b_w g)
$$
depending on whether $w' \le w+1$ or $w'\geq w+1$, is a product 
of at least $w$ many non-scalar relatively coprime factors. So 
$\dpb(h)\geq w$. The inequality $\ddb(h)\geq \dpb(h)$ has 
already been noted.

(2) Note that $\varphi(h)$ has the form 
$\frac{(\varphi(u)-a_1)(\varphi(u)-a_2)\cdots (\varphi(u)-a_{w'})}
{(\varphi(u)-b_0)(\varphi(u)-b_1)\cdots (\varphi(u)-b_w)}$
where $\varphi(u)\in F\setminus \Bbbk$. The assertion follows 
from part (1).

(3) Let $\bfs =(x_1^{-1}, x_2,\cdots,x_n)$. Then
$$h=\frac{s_1^{w-w'+1}(1-a_1 s_1)(1-a_{2}s_1)\cdots (1-a_{w'}s_1)}
{(1-b_1 s_1)\cdots (1-b_w s_1)}$$
which implies that $\ddb(h)\leq w$. The assertion follows from 
part (1). In this case, one also sees that $\fdb(h)\leq 1$. 

(4) We continue the proof of part (1) under the extra hypotheses 
that $u=x_1$ and $w+1\geq w'$. Let $\bfs \in \Phi$ such that the 
denominator of $h$ in \eqref{E8.0.1} has $\deg_{\bfs} = w$. But 
in the form \eqref{E8.2.1}, the denominator of $h$ is 
$f(f-b_1 g)\cdots (f-b_w g)$, due to $w-w'+1 \ge 0$. Thus, one 
of the elements in $\{f-b_j g\}_{j=0}^w$ must be a scalar. This 
implies that none of $f-a_i g $ or $g$ is a scalar. So the 
numerator of $h$ has degree at least $w+1$, and hence 
$\deg_{\bfs} h \ge 1$. Thus $\fdb(h)\geq 1$. By the proof of 
part (3), $\fdb(h)\leq 1$. The assertion follows.
\end{proof}

\begin{definition}
\label{zzdef8.3}
Let $V$ be a nonempty subset of $F$. 
\begin{enumerate}
\item[(1)]
The {\it denominator degree bound} 
of $V$ is defined to be
$$\ddb(V) :=\sup\{\ddb(h)\mid h\in V\}.$$
\item[(2)]
The {\it denominator prime-divisor bound} 
of $V$ is defined to be
$$\dpb(V) :=\sup\{\dpb(h)\mid h\in V\}.$$
\end{enumerate}
\end{definition}

\begin{lemma}
\label{zzlem8.4}
Let $K$ be the Poisson field $K\{\bff\}$ for some $\bff\in
\Bbbk(x,y)$. Let $V$ be the subalgebra ${^1\Gamma}_{0}(K)$.
\begin{enumerate}
\item[(1)]
If $\bff\in \Bbbk[x,y]$, then $\ddb(V)=0$.
\item[(2)]
Write $\bff=\frac{a}{b}$ where $a,b$ are in $\Bbbk[x,y]$, and 
let $w$ be the number of non-associate prime divisors of $b$. 
Then $\dpb(V)\leq w \le \deg b$.
\end{enumerate}
\end{lemma}

\begin{proof}
(1) By Corollary \ref{zzcor3.9}, $V$ is a subset of $\Bbbk[x,y]$.
Since every element in $\Bbbk[x,y]$ has $\ddb$ 0, the assertion
follows.

(2) Let $A$ be the localization $\Bbbk[x,y][b^{-1}]$. Then 
$A\subseteq K$ is a noetherian normal domain such that $A$ is a Poisson 
subalgebra of $K$ and $Q(A)=K$. By Theorem \ref{zzthm3.8}, 
$V$ is a subset of $A$. It remains to show that every
element in $A$ has $\dpb$ bounded by $w$.

Write $b=p_1^{m_1}\cdots p_{w}^{m_w}$ as a prime decomposition. 
Then $w\leq \deg b$. Since every element $h$ in $A$ is of the 
form $c b^{-n}$ with $c \in \Bbbk[x,y]$ and $n \in \Znn$, one 
can write $h$ as $\frac{c}{p_1^{m_1n}\cdots p_{w}^{m_wn}}$. 
Hence $\dpb(h)\leq w$. The assertion follows.
\end{proof}

Now we are ready to give an example where $K\{\bff\}$ is not 
isomorphic to $K\{\bfg\}$ for any $\bfg\in \Bbbk[x,y]$.

\begin{example}
\label{zzexa8.5}  
Let $h \in \Bbbk(x,y)$ be an element with positive $\ddb$, see 
Lemma \ref{zzlem8.2}. Let $f(t) \in \Bbbk[t]$ be a polynomial of degree 
$\geq 2$ and $\bff :=xyf(h)$. By Corollary \ref{zzcor3.11}, 
$h\in {^1\Gamma}_{0}(K)$ where $K:=K\{\bff\}$. Since 
$\ddb(h)\geq 1$, we have $\ddb(V)\geq 1$ where 
$V={^1\Gamma}_{0}(K)$. If $K\cong K\{\bfg\}$ for some 
$\bfg\in \Bbbk[x,y]$, then $\ddb(V)=0$ by Lemma 
\ref{zzlem8.4}(1), yielding a contradiction.
\end{example}

We can generalize the above example to the following theorem.
For every nonnegative integer $w$, 
let $G_{w}$ be the set of elements $\bfg \in \Bbbk(x,y)$ such that 
$\bfg =a/b$ for some coprime polynomials $a,b \in \Bbbk[x,y]$ and that the number of non-associate prime divisors of $b$ is no more than $w$.

\begin{theorem}
\label{zzthm8.6}
Let $w$ be a nonnegative integer and let 
$h\in \Bbbk(x,y)$ be a nonzero element such that 
$\dpb(\varphi(h))>w$ for every $\Bbbk$-algebra map 
$\varphi: \Bbbk(x,y) \to \Bbbk(x,y)$ {\rm{(}}see Lemma \ref{zzlem8.2}(2){\rm{)}}. 
Let $f(t) \in \Bbbk[t]$ be a 
polynomial of degree at least $2$. Then there is no Poisson 
algebra morphism from $K\{xyf(h)\}$ to $K\{\bfg\}$ for any $\bfg\in G_w$.
\end{theorem}

\begin{proof}
Let $K :=K\{xyf(h)\}$ and $P:=K\{\bfg\}$ where $\bfg\in G_w$. Suppose there is a 
Poisson algebra morphism $\varphi: K\to P$. So we may 
consider $K$ as a Poisson subfield of $P$ via the embedding 
$\varphi$. By Lemma \ref{zzlem3.7}(1) and Corollary 
\ref{zzcor3.11}, 
$h\in {^1\Gamma}_{0}(K) \subseteq {^1\Gamma}_{0}(P)$. 
Since $\dpb(\varphi(h))> w$ by assumption, we have $\dpb(V)>w$ 
where $V={^1\Gamma}_{0}(P)$. By Lemma \ref{zzlem8.4}(2), 
$\dpb(V)\leq w$, yielding a contradiction.
\end{proof}

Finally we list some open questions.

\begin{questions}
\label{zzque8.7}
Let $K$ be a Poisson field of transcendence degree at least 2.
\begin{enumerate}
\item[(1)]
Is there a connection between the following three properties
for Poisson fields $K$?
\begin{enumerate}
\item[(a)] 
$K$ has the Dixmier property;
\item[(b)]
$K$ has only finitely many Poisson subfields with the same transcendence degree as $K$;
\item[(c)]
$\Aut_{Poi}(K)$ is finite.
\end{enumerate}
\item[(2)]
What is the structure of the Poisson Cremona group $Cr_2(\Bbbk,1) = \Aut_{Poi}(\Kweyl)$?
\item[(3)]
Let $K = K\{\bff\}$ where $\bff$ is a nonzero polynomial in $\Bbbk[x,y]$.  Dumas proved in \cite[Proposition, p.12]{Dumas} that if $\deg \bff \le 2$, then $K$ is isomorphic to either $\Kweyl$ or some $K_q$ (Corollary \ref{zzcor1.2}).  This cannot hold when $\deg \bff > 2$, in view of Proposition \ref{zzpro4.6}. Can such Poisson fields (with $\bff \in \Bbbk[x,y]$) be classified?
\item[(4)]
Suppose $K$ is a Poisson field (say with $\trdeg K = 2$) such that $H = \Aut_{Poi}(K)$ is finite. Then $K^H$ is a Poisson field with the same transcendence degree as $K$.  Does $K^H$ have the Dixmier property?
\item[(5)]
Do there exist Poisson fields $K$ (say with $\trdeg K = 2$) which are \emph{minimal} in the sense that $K$ has no proper Poisson subfields with the same transcendence degree? 
\end{enumerate}
\end{questions}

%%%%%%%%%%%%%%%%%%%%%%%%%%%%%%%%
\section{Appendix: Some basic Poisson algebra items}
\label{zzsec9}

Recall that, in this paper, a Poisson algebra is commutative by definition.

\begin{lemma}  
\label{zzlem9.1}
Let $A$ be a Poisson algebra, $x,y \in A$, 
and $a,b,c,d \in \Znn$. Then
$$
\{ x^ay^b, x^cy^d \} 
= (ad-bc) x^{a+c-1} y^{b+d-1} \{x,y\}.
$$
If $x$ and $y$ are invertible in $A$, the 
same formula holds for all $a,b,c,d \in \ZZ$.
\end{lemma}

Note: For $A= K_q$ and the variables $x$, $y$, this is Equation \eqref{E1.5.1}.

\begin{proof}
This is a straight bracket calculation:
\begin{align*}
\{ x^ay^b, x^cy^d \} &= x^a \{y^b, x^cy^d\} + y^b \{x^a,x^cy^d\}  \\
&= b x^a y^{b-1} \{y,x^cy^d\} + a x^{a-1} y^b \{x, x^cy^d \}  \\
&= b x^a y^{b-1+d} \{y,x^c\} + a x^{a-1+c} y^b \{x,y^d\}  \\
&= bc x^{a+c-1} y^{b-1+d} \{y,x\} + ad x^{a-1+c} y^{b+d-1} \{x,y\}  \\
&= (ad-bc) x^{a+c-1} y^{b+d-1} \{x,y\}
\end{align*}
for allowable $a$, $b$, $c$, $d$.
\end{proof}

It is well known that a Poisson bracket on the polynomial algebra $\Bbbk[x,y]$ is uniquely determined by $\{x,y\}$, and that every polynomial in $\Bbbk[x,y]$ occurs as $\{x,y\}$ for a (unique) Poisson bracket $\{-,-\}$ (e.g., \cite[p.235, last paragraph of \S9.1.1]{LGPV}).  Here we note the corresponding facts for Poisson brackets on $\Bbbk(x,y)$.

\begin{lemma}  \label{zzlem9.2}
Let $K = \Bbbk(x,y)$ with a Poisson bracket $\{-,-\}$, and set $f := \{x,y\}$. Then $\{g,h\} = \{g,h\}_wf$ for all $g,h \in K$.
\end{lemma}

\begin{proof}
If $g=0$ or $h=0$, the result is trivial, so assume $g,h \ne 0$.

First consider $g = x^ay^b$ and $h = x^cy^d$ for $a,b,c,d \in \Znn$. Applying Lemma \ref{zzlem9.1} to both $\{-,-\}$ and $\{-,-\}_w$, we obtain
$$
\{ x^ay^b, x^cy^d \} = (ad-bc) x^{a+c-1} y^{b+d-1} f = \{ x^ay^b, x^cy^d \}_w f.
$$
It immediately follows that $\{g,h\} = \{g,h\}_w f$ for all $g,h \in \Bbbk[x,y]$.

In the general case, $g = st^{-1}$ and $h = uv^{-1}$ for some nonzero $s,t,u,v \in \Bbbk[x,y]$. Then
\begin{align*}
\{g,h\} &= \{ st^{-1}, uv^{-1} \} = \bigl( \{s,uv^{-1}\} t - s \{t,uv^{-1}\} \bigr) t^{-2}  \\
&= \bigl( \{s,u\} v - u \{s,v\} \bigr) v^{-2} t^{-1} - s \bigl( \{t,u\} v - u \{t,v\} \bigr) v^{-2} t^{-2}  \\
&= \{s,u\} v^{-1} t^{-1} - \{s,v\} uv^{-2}t^{-1} - \{t,u\} sv^{-1}t^{-2} + \{t,v\} suv^{-2}t^{-2} \,.
\end{align*}
The same equation holds with $\{-,-\}$ replaced by $\{-,-\}_w$, and thus, invoking $\{s,u\} = \{s,u\}_w f$ etc., we conclude that $\{g,h\} = \{g,h\}_wf$.
\end{proof}

\begin{lemma}  \label{zzlem9.3}
Let $A$ be a  Poisson algebra and $f \in A$. Set $\{a,b\}' := \{a,b\} f$ for  $a,b\in A$. Then $\{-,-\}'$ is a Poisson bracket on $A$ if and only if
\begin{equation}  \label{E9.3.1}  \tag{E9.3.1}
\bigl( \{a,b\}\{f,c\} + \{c,a\}\{f,b\} + \{b,c\}\{f,a\} \bigr) f = 0 \qquad \forall\; a,b,c \in A.
\end{equation}
\end{lemma}

\begin{proof}
It is clear that $\{-,-\}'$ is a bilinear, skew-symmetric biderivation on $A$, so it is a Poisson bracket if and only if it satisfies the Jacobi identity. Observe that
$$
\{ \{u,v\}', w\}' = \{ \{u,v\} f, w\} f = \{ \{u,v\},w \} f^2 + \{u,v\} \{f,w\} f
$$
for $u,v,w \in A$. Consequently,
\begin{align*}
\{\{a,b\}',c\}' + \{\{c,a\}',b\}' &+ \{\{b,c\}',a\}' = \{\{a,b\},c\}  f^2 + \{a,b\}\{f,c\} f  \\
&\qquad\qquad\qquad\qquad + \{\{c,a\},b\} f^2 + \{c,a\}\{f,b\} f  \\
&\qquad\qquad\qquad\qquad + \{\{b,c\},a\} f^2 + \{b,c\}\{f,a\} f  \\
&= \{a,b\}\{f,c\} f + \{c,a\}\{f,b\} f + \{b,c\}\{f,a\} f
\end{align*}
for $a,b,c \in A$.
Therefore $\{-,-\}'$ satisfies the Jacobi identity if and only if \eqref{E9.3.1} holds.
\end{proof}

\begin{corollary}  \label{zzcor9.4}
Let $K = \Bbbk(x,y)$ and $\bff \in K$. Then there is a unique Poisson bracket $\{-,-\}$ on $K$ such that $\{x,y\} = \bff$.
\end{corollary}

\begin{proof}
Uniqueness is clear in case of existence. If such a bracket exists, it must equal $\{-,-\}_w \bff$. We apply Lemma \ref{zzlem9.3} with $\{-,-\}$ and $\{-,-\}'$ replaced by $\{-,-\}_w$ and $\{-,-\}_w \bff$.

Let $a,b \in K$, and consider the derivation
$$
\delta := \{a,b\}_w\{\bff,-\}_w + \{-,a\}_w\{\bff,b\}_w + \{b,-\}_w\{\bff,a\}_w \,.
$$
We compute that
$$
\delta(x) = (a_xb_y - a_yb_x) (- \bff_y) + a_y (\bff_xb_y - \bff_yb_x) + (- b_y) (\bff_xa_y - \bff_ya_x) = 0,
$$
and similarly $\delta(y)= 0$. Consequently, $\delta = 0$, that is,
$$
\{a,b\}_w\{\bff,c\}_w + \{c,a\}_w\{\bff,b\}_w + \{b,c\}_w\{\bff,a\}_w = 0
$$
for all $c \in K$. This verifies \eqref{E9.3.1} for $\{-,-\}_w$.
\end{proof}

\begin{example}  \label{zzexa9.5}
There exists a Poisson field $K$ with an element $f \in K$ such that the biderivation $\{-,-\} f$ on $K$ is not a Poisson bracket.

Take $K = \Bbbk(x_1,x_2,y_1,y_2)$ to be the rank $2$ Weyl Poisson field, where
$$
\{x_i,x_j\} = \{y_i,y_j\} = 0 \qquad\text{and}\qquad \{x_i,y_j\} = \delta_{ij}
$$
for $i,j=1,2$. Observe that
$$
\{x_1,y_1\} \{x_2,y_2\} + \{y_2,x_1\} \{x_2,y_1\} + \{y_1,y_2\} \{x_2,x_1\} = 1,
$$
whence \eqref{E9.3.1} fails with $a=x_1$, $b=y_1$, $c=y_2$ and $f=x_2$. Therefore $\{-,-\} x_2$ is not a Poisson bracket on $K$.
\end{example}

\subsection*{Acknowledgments} 
The authors thank Hongdi Huang, Xingting Wang, and Milen Yakimov for many valuable conversations and correspondences on the subject. 
The first-named author was supported by the US National Science Foundation 
grant DMS-1601184, and the second-named author by the US National Science Foundation 
grants DMS-2001015 and DMS-2302087.

\providecommand{\bysame}{\leavevmode\hbox to3em{\hrulefill}\thinspace}
\providecommand{\MR}{\relax\ifhmode\unskip\space\fi MR }
\providecommand{\MRhref}[2]{%

\href{http://www.ams.org/mathscinet-getitem?mr=#1}{#2} }
\providecommand{\href}[2]{#2}

\end{document}